\documentclass[11pt,reqno]{amsart}

\usepackage[utf8]{inputenc}
\usepackage[T1]{fontenc}
\usepackage[english]{babel}
\usepackage{amsmath,amssymb,amsthm,mathtools}
\usepackage{enumerate}
\usepackage{xcolor}
\usepackage{tikz}
\usetikzlibrary{arrows.meta}
\usepackage[colorlinks,citecolor=blue,urlcolor=blue,breaklinks=true]{hyperref}

\def\Xint#1{\mathchoice
{\XXint\displaystyle\textstyle{#1}}%
{\XXint\textstyle\scriptstyle{#1}}%
{\XXint\scriptstyle\scriptscriptstyle{#1}}%
{\XXint\scriptscriptstyle\scriptscriptstyle{#1}}%
\!\int}
\def\XXint#1#2#3{{\setbox0=\hbox{$#1{#2#3}{\int}$ }
\vcenter{\hbox{$#2#3$ }}\kern-.58\wd0}}

\def\avint{\;\Xint-}

\newcommand{\ve}{{\varepsilon}}
\newcommand{\R}{{\mathbb R}}
\newcommand{\N}{{\mathbb N}}
\newcommand{\Z}{{\mathbb Z}}
\newcommand{\W}{{\mathcal W}}
\newcommand{\DD}{{\mathcal D}}
\newcommand{\HH}{{\mathcal H}}
\newcommand{\LL}{{\mathcal L}}
\newcommand{\NN}{{\mathcal N}}
\newcommand{\BB}{{\mathcal B}}
\newcommand{\QQ}{{\mathcal Q}}
\newcommand{\CC}{{\mathcal C}}
\newcommand{\diam}{{\operatorname{diam}}}
\newcommand{\dist}{\operatorname{dist}}
\newcommand{\supp}{{\operatorname{supp}}}
\newcommand{\loc}{{\rm loc}}
\newcommand{\rf}[1]{{(\ref{#1})}}
\newcommand{\wt}[1]{{\widetilde{#1}}}

\definecolor{grisclar}{gray}{0.55}
\newcommand{\by}[1]{\stackrel{\text{\color{grisclar}#1}} }

\newcommand{\qedformula}{\pushQED{\qed}\qedhere\popQED}
\newcommand{\qedOut}{\renewcommand{\qedsymbol}{}\vspace{-1em}}

\newtheorem{theorem}{Theorem}[section]
\newtheorem{lemma}[theorem]{Lemma}
\newtheorem{proposition}[theorem]{Proposition}
\newtheorem{corollary}[theorem]{Corollary}

\theoremstyle{definition}
\newtheorem{definition}[theorem]{Definition}
\newtheorem{example}[theorem]{Example}

\newtheorem{remark}[theorem]{Remark}

\newtheorem{theoremletter}{Theorem}

\newtheorem{corollaryletter}[theoremletter]{Corollary}
\newtheorem{lemmaletter}[theoremletter]{Lemma}

\numberwithin{equation}{section}

\title[Weak reverse H\"older inequalities for doubling measures]
      {Weak reverse H\"older inequalities for doubling measures\\
       on closed subsets of $\R^d$}

\author{Mart\'i Prats}
\address{Departament de Matem\`atiques, Universitat Aut\`onoma de Barcelona,
         08193 Bellaterra (Barcelona), Catalonia}
\email{marti.prats@uab.cat}


\thanks{The author was supported by the grants PID2024-160507NB-I00 (AEI) 
and 2021-SGR-00087 (AGAUR)}

\begin{document}

\begin{abstract}
Let $\mu$ be a doubling Radon measure in $\R^d$ with support $X$. We prove a local and quantitative
form of the Bennewitz--Lewis theorem in this setting: if a Radon measure $\nu$
satisfies the Bennewitz--Lewis condition on a bounded relatively open set
$V\subset X$ with a small enough constant, then $\nu|_V\ll\mu|_V$ and its
density satisfies a reverse H\"older inequality on every compact subset
$K\subset V$. 
Along the way we
develop the dyadic toolkit these arguments run on --- triple cubes defined
through neighbours, the basis $\DD\cup3\DD$, and a Whitney decomposition --- and we collect the Muckenhoupt-type
classes of weights over that basis.
The same machinery yields a
local Gehring lemma for enlarged balls: both theorems are deduced from a
single bootstrapping lemma, stated once for $p=1$ and $p>1$. \end{abstract}

\maketitle

\section{Introduction}
\label{secIntro}

For the Lebesgue measure in $\R^d$, the Muckenhoupt class
$A_\infty=\bigcup_{p>1}A_p$ admits a long list of mutually equivalent
descriptions: a reverse H\"older inequality, a reverse Jensen inequality, an
$L\log L$ bound, a Fujii--Wilson maximal condition, and the quantitative
absolute continuity condition with which Muckenhoupt \cite{Muckenhoupt}
originally identified the class. Outside that classical setting the list
breaks apart. Duoandikoetxea, Mart\'in-Reyes and Ombrosi
\cite{DuoandikoetxeaMartinReyesOmbrosi} formulated eight of these conditions
over an arbitrary basis of sets and showed that for some bases no two of the
resulting classes coincide.

The situation studied here sits between these two extremes. Let $\mu$ be a
doubling Radon measure in $\R^d$ and put $X:=\supp\mu$, which we allow to be a
proper and possibly nowhere dense closed subset of $\R^d$ --- a Cantor set or
the boundary of a rough domain, for instance. All averages are taken
over the \emph{restricted balls}
$$\Delta_r(x):=B_r(x)\cap X,\qquad x\in X,\ r>0,$$
and over a system of dyadic cubes adapted to $X$ in the sense of Christ
\cite{Christ} and David \cite{DavidCubes}. We write $t\Delta_r(x):=\Delta_{tr}(x)$ for the
concentric dilate of a restricted ball.

The triple $(X,|\cdot|,\mu)$ is then a metric measure space which is doubling,
which is complete because $X$ is closed, and whose quasi-triangle constant
equals $1$; and the restricted balls are precisely its metric balls.
Consequently every theorem about balls in a doubling metric measure space, or
in a space of homogeneous type, applies here verbatim. 

Two classical implications break down here, and it is worth saying precisely
which. For $1<p<\infty$ consider the reverse H\"older inequality
\begin{equation}\label{eqStrongRHintro}
  \left(\avint_{\Delta} w^p\, d\mu\right)^{\frac1p}
  \leq C\avint_{\Delta} w\, d\mu
  \qquad\text{for every restricted ball } \Delta,
\end{equation}
in which \emph{both} averages are taken over the same ball. We call
\rf{eqStrongRHintro} the \emph{strong} reverse H\"older inequality, to
distinguish it from the \emph{weak} one, in which the right-hand side is
averaged over $2\Delta$ instead. The strong inequality implies the weak
one, trivially and in any of the settings below; the converse
fails already for the Lebesgue measure, as Remark \ref{remExpExample} recalls.
For the Lebesgue measure, however, the strong inequality sits inside a closed
web of equivalences: a weight lies in $A_\infty$ if and only if it satisfies
\rf{eqStrongRHintro} for some $p>1$ (see \cite{HytonenPerez,HytonenPerezRela}
and the references therein), and if it satisfies \rf{eqStrongRHintro} then it
satisfies the same inequality with $p+\ve$ in place of $p$, for some $\ve>0$
\cite{Gehring}.

That web comes apart once the underlying space is allowed to be a general
space of homogeneous type: both directions of the equivalence fail, and so
does the self-improvement. Anderson, Hyt\"onen and Tapiola
\cite[\S7]{AndersonHytonenTapiola} build a ``haircomb'' space --- an infinite
line in $\R^2$ with finite segments attached, carrying the $\ell^\infty$
metric and arclength measure --- together with a family of weights $f_h$ on
it. For a suitable choice of $h$
(\S7.3) it satisfies \rf{eqStrongRHintro} exactly for $p\leq 1/\alpha$, so
that no self-improvement is possible. The root cause was identified by
Anderson and Weirich \cite[Thm.~5.1]{AndersonWeirich}: those weights do not
define doubling measures, and \cite[Thm.~5.2]{AndersonWeirich} shows that a
doubling weight satisfies the strong Gehring lemma in any space of homogeneous
type. Korte and Kansanen \cite[Ex.~4.1]{KorteKansanen} give a
correspondingly elementary witness for the other direction of the equivalence:
a weight on the union of two hyperplanes satisfying \rf{eqStrongRHintro} with
$p=1+\ve$ whose associated measure is not doubling.

None of these examples, however, settles the matter for \emph{restricted
Euclidean} balls, and this is worth being precise about, because the setting
of this paper is a special case of a space of homogeneous type and not a
general one. All three are built on the fact that
``\emph{balls are actually squares}'', and rely on the same thing:
a flat piece of $X$ enters an $\ell^\infty$-ball all at once, whereas it
enters a Euclidean ball gradually. Figure \ref{figHaircomb} shows the
configuration on which all three turn. What does survive in our setting is the
root failure, and Example \ref{exLadder} below exhibits it: on a closed subset
of $\R$ carrying a doubling measure, and with restricted Euclidean balls,
\rf{eqStrongRHintro} does not force the weight to define a doubling measure.
By \cite[Thm.~5.2]{AndersonWeirich} that is exactly the hypothesis under which
the strong Gehring lemma would be recovered. Whether the Euclidean analogues
of \cite[\S7.2 and \S7.3]{AndersonHytonenTapiola} hold we do not know; see
\S\ref{secOpen}.

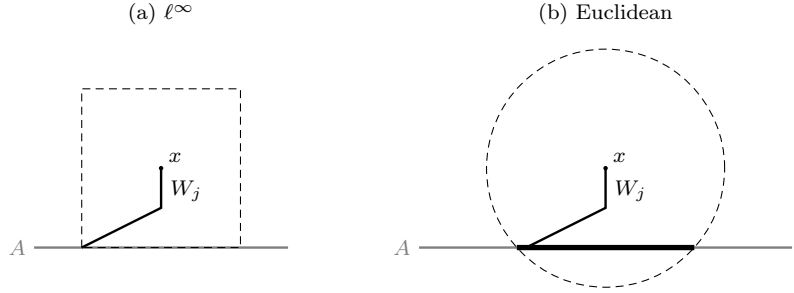
\begin{figure}[ht]
\centering
\begin{tikzpicture}[
	scale=1.05,
	xset/.style={line width=0.9pt},
	aline/.style={line width=0.9pt,gray},
	ball/.style={densely dashed},
	lbl/.style={font=\scriptsize,inner sep=1.5pt}
]
	\begin{scope}
		\draw[aline] (-0.6,0) -- (2.6,0);
		\node[lbl,gray,anchor=east] at (-0.65,0) {$A$};
		\draw[xset] (0,0) -- (1,0.5) -- (1,1);
		\node[lbl,anchor=west] at (1.06,0.72) {$W_j$};
		\draw[ball] (0,0) rectangle (2,2);
		\fill (1,1) circle (0.8pt);
		\node[lbl,anchor=south west] at (1.04,1.02) {$x$};
		\node[lbl,anchor=south] at (1,2.75) {(a) $\ell^\infty$};
	\end{scope}
	\begin{scope}[xshift=5.6cm]
		\draw[aline] (-1.35,0) -- (3.35,0);
		\node[lbl,gray,anchor=east] at (-1.4,0) {$A$};
		\draw[xset] (0,0) -- (1,0.5) -- (1,1);
		\node[lbl,anchor=west] at (1.06,0.72) {$W_j$};
		\draw[ball] (1,1) circle (1.5);
		\draw[line width=2pt] (-0.118,0) -- (2.118,0);
		\fill (1,1) circle (0.8pt);
		\node[lbl,anchor=south west] at (1.04,1.02) {$x$};
		\node[lbl,anchor=south] at (1,2.75) {(b) Euclidean};
	\end{scope}
\end{tikzpicture}
\caption{The configuration behind
\cite[\S7]{AndersonHytonenTapiola} and
\cite[Ex.~4.1]{KorteKansanen}. In (a) the $\ell^\infty$-ball centred at
$x=(10j+1,1)$ of radius $1$ meets $X$ exactly in the whisker $W_j$: the line
$A$ lies on its boundary and contributes nothing, so the average of the weight
over the ball sees only the spike. In (b) the foot of the whisker is at
Euclidean distance $\sqrt2$ from $x$, so a Euclidean ball that contains it has
radius at least $\sqrt2$ and therefore meets $A$ in a segment of length at
least $2$, whose contribution cannot be avoided. This is what the three
examples use, and what they lose over restricted Euclidean balls.}
\label{figHaircomb}
\end{figure}

The remedy, present already in Bennewitz and Lewis \cite{BennewitzLewis}, is to
allow a dilation: one compares an average of $w^p$ over $\Delta$ with an
average of $w$ over $2\Delta$. Anderson, Hyt\"onen and Tapiola
\cite{AndersonHytonenTapiola} developed the resulting classes systematically
in a space of homogeneous type, showed that they do not depend on the dilation
factor, identified the weak $A_\infty$ class with the union of the weak
reverse H\"older classes, and proved that the latter self-improve. Kinnunen,
Kurki and Mudarra \cite{KinnunenKurkiMudarra} sharpened the picture for
doubling metric measure spaces: they allow the weight to be defined only on an
open subset, and give ten equivalent formulations, among them the
quantitative absolute continuity condition
$$\mu(F)\leq\ve\,\mu(\Delta)\implies \nu(F)\leq\eta\,\nu(2\Delta).$$

One condition in this circle of ideas has no counterpart in that literature.

\begin{definition}\label{defBLintro}
Let $V\subset X$ be relatively open and let $0<\ve_0<1$. We say that a Radon
measure $\nu$ supported in $X$ belongs to $wBL_{\ve_0}(\mu,V)$ if there exist
constants $0<\delta_0<1$ and $C_0\geq1$ such that
\begin{equation}\label{eqBLintro}
  \mu(2\Delta\cap E)<\delta_0\,\mu(2\Delta)
  \implies
  \nu(\Delta)\leq \ve_0\,\nu(2\Delta\cap E)+C_0\,\nu(2\Delta\setminus E)
\end{equation}
for every Borel set $E\subset 2\Delta\subset V$.
\end{definition}

Condition \rf{eqBLintro} is formally weaker than the quantitative absolute
continuity condition above: it does not assert that $\nu(\Delta)$ is small,
only that it is controlled by a small multiple of the mass $\nu$ puts on the
large part of $E$, plus a fixed multiple of the mass $\nu$ puts on the
complement. This is what makes it usable: the conclusion of Theorem
\ref{theoAintro} below is a reverse H\"older inequality, and is therefore
strong information, but the hypothesis \rf{eqBLintro} is weak enough to be verified
for measures that one has little control over.

Bennewitz and Lewis introduced \rf{eqBLintro} in order to prove a weak reverse
H\"older inequality for harmonic measure on a domain satisfying a uniform
interior big pieces of Lipschitz graphs condition but no Harnack chain
condition, a setting in which harmonic measure need not be doubling. The
mechanism is their Lemma~3.1, which is entirely measure-theoretic; but it is
stated with the surface measure $\HH^{d-1}$ of the boundary of a bounded
Ahlfors regular domain as the underlying measure, and on a single ball as the
reference set.

The first result removes both restrictions.

\begin{theoremletter}\label{theoAintro}
Let $\mu$ be a doubling Radon measure in $\R^d$ with doubling constant
$C_\mu$, let $X=\supp\mu$, let $V\subset X$ be a bounded relatively open set
and let $K\subset V$ be compact with $\mu(K)>0$. Let $\nu$ be a Radon measure
supported in $X$ with
$$\nu\in wBL_{\ve_0}(\mu,V),\qquad \ve_0\leq\ve_*,$$
where $\ve_*>0$ depends only on $d$ and on $C_\mu$, and is made explicit in
\S\ref{secBL}.
Then $\nu|_V\ll\mu|_V$ and, writing $w=\frac{d\nu}{d\mu}$, there exists
$q_0>1$ depending on $d$, on $C_\mu$ and on the constants $\delta_0,C_0$ of
\rf{eqBLintro}, such that for every $1<q\leq q_0$
$$\left(\avint_K w^q\,d\mu\right)^{\frac1q}\leq C\,\frac{\nu(V)}{\mu(V)},$$
where $C$ depends in addition on $q$, on
$\dfrac{\diam K}{\dist(K,X\setminus V)}$ and on $\dfrac{\mu(V)}{\mu(K)}$.
\end{theoremletter}

Taking $V$ to be a ball and letting $K$ exhaust it recovers the weak reverse
H\"older inequality, and hence the missing link in the chain of equivalences:

\begin{corollaryletter}\label{coroBintro}
If $\mu$ is a doubling Radon measure in $\R^d$ and $V\subset X$ is relatively
open, then
$$wRH_1^U(\mu,V)=wRH_1^{QC}(\mu,V)=wRH_1^{BL}(\mu,V).$$
\end{corollaryletter}

The first of the two equalities is, in this generality, due to Kinnunen, Kurki
and Mudarra \cite[Thm.~4.1]{KinnunenKurkiMudarra}; what is proved here is the
inclusion $wRH_1^{BL}\subset wRH_1^U$. 

Figure \ref{figImplications} places this result, and the one below, among the
elementary implications between the classes involved, and says which of them
need a doubling hypothesis.

Behind that figure there are, over the basis $\DD\cup3\DD$ and for $\mu$
$\DD$-doubling in $Q_0$, only two classes. Write $QC(\mu,\BB_{Q_0})$ for the
set of measures $\nu$ for which there is one pair of constants
$0<\delta,\ve<1$ with
$$\mu(E)<\delta\,\mu(B)\implies\nu(E)\leq\ve\,\nu(B)$$
for every $B\in\DD(Q_0)\cup3\DD(Q_0)$ and every Borel $E\subset B$.
Then $\nu$ satisfies the reverse H\"older inequality \rf{eqStrongRHintro} over
that basis for some $p>1$ if and only if $\nu\in QC(\mu,\BB_{Q_0})$, and $\nu$
belongs to $A_p$ over that basis for some $p<\infty$ if and only if
$\nu\in QC(\mu,\BB_{Q_0})$ and $\nu$ is itself $\DD$-doubling. These are
Theorems \ref{theoEqualityRH} and \ref{theoEqualityA}; what the dyadic layer
buys is that a condition carrying two quantifiers is replaced by a single pair
of constants. The second class is strictly smaller than the first, already for
the Lebesgue measure in $\R$, by Example \ref{exStaircase}. For restricted
Euclidean balls the gap closes as soon as $\mu$ has thin annuli, by
Proposition \ref{propoThinAnnuli}, and Example \ref{exLadder} shows that
without them it need not.

The second result is the self-improvement, in the same local form.

\begin{theoremletter}\label{theoCintro}
Let $\gamma>1$, let $\mu$ be a doubling Radon measure in $\R^d$ with doubling
constant $C_\mu$, let $V\subset X$ be bounded and relatively open, and let
$K\subset V$ be compact with $\mu(K)>0$. Let $1<p<\infty$ and let
$\nu\ll\mu$, with density $w=\frac{d\nu}{d\mu}$, satisfy
$$\left(\avint_{\Delta}w^p\, d\mu\right)^{\frac1p}
\leq C_{RH}\avint_{\gamma\Delta}w\, d\mu
\qquad\text{for every restricted ball }\Delta\text{ with }\gamma\Delta\subset V .$$
Then there is $\ve_0>0$, depending on $d$, $C_\mu$ and $p$ and on $\gamma$
and $C_{RH}$ only through the product $\gamma^{d}C_{RH}$, such that for every
$0<\ve<\ve_0$
$$\left(\avint_{K}w^{p+\ve}\, d\mu\right)^{\frac1{p+\ve}}
\leq C\left(\avint_{V}w^{p}\, d\mu\right)^{\frac1p},$$
with $C$ depending in addition on $\ve$, on
$\frac{\diam K}{\dist(K,X\setminus V)}$ and on $\frac{\mu(V)}{\mu(K)}$.
\end{theoremletter}

Globally, and without the localization, Theorem \ref{theoCintro} is
\cite[Prop.~6.2]{AndersonHytonenTapiola}; the strong dyadic counterpart is
\cite[Thm.~3.1]{AndersonWeirich}, and for metric measure spaces with an
annular decay property the dilation can be removed altogether, by
\cite[Cor.~3.2]{Maasalo}. Theorem \ref{theoCintro} is not a formal
consequence of \cite[Prop.~6.2]{AndersonHytonenTapiola}. One can deduce a
version of it by extending $w$ to a globally defined weak $RH_p$ weight,
constant outside a slightly smaller set, and applying the global result; but
the constant of that extension carries a factor
$\left(\diam V/\dist(K,X\setminus V)\right)^{\log_2C_\mu/p}$, and the range of
exponents in \cite[Prop.~6.2]{AndersonHytonenTapiola} depends on the constant
of the weight, so the admissible $\ve$ would degenerate as $K$ approaches the
relative boundary of $V$. In Theorem \ref{theoCintro} it does not: $\ve_0$
depends on $d$, $C_\mu$, $p$ and $\gamma^{d}C_{RH}$ alone, and all the
geometry is carried by the constant $C$.

Indeed, both are corollaries of one lemma, which we isolate because it may be
of use elsewhere.

\begin{lemmaletter}\label{lemDintro}
Let $\mu$ be a measure, let $V$ be measurable, let $p\geq1$ and let $f\geq0$
belong to $L^p(\mu|_V)$. Put $\sigma\coloneqq0$ if $p=1$ and $\sigma\coloneqq1$
if $p>1$, let $\rho>1$, $\alpha_j\coloneqq\rho^j\alpha_0$ and
$A^j\coloneqq\{x\in V:\, f(x)>\alpha_j\}$. If
$$\int_{A^{j}}f^p\, d\mu\leq C_1\,\alpha_{j+1}^{\,p-\sigma}
\int_{A^{j}}f^{\,\sigma}\, d\mu\quad\text{for every }j\geq0,
\qquad\text{with }\ \rho^{\,p-\sigma}>2C_1>2,$$
then there is $q_0>p$, given explicitly in terms of $C_1$, $\rho$ and $p$,
such that
$$\int_{V}f^q\, d\mu\leq C\,\alpha_0^{q-p}\int_{V}f^p\, d\mu
\qquad\text{for every }p<q\leq q_0 .$$
\end{lemmaletter}

Here $f^0\equiv1$, so the hypothesis reads
$\int_{A^j}f\, d\mu\leq C_1\alpha_{j+1}\mu(A^j)$ when $p=1$ and
$\int_{A^j}f^p\, d\mu\leq C_1\alpha_{j+1}^{p-1}\int_{A^j}f\, d\mu$ when $p>1$.
Nothing about the underlying space is used in the proof --- no metric, no
covering argument, no finiteness assumption on $\mu$ --- and the two cases
are a single computation producing a single constant; see Lemma
\ref{lemLevelToQ} and Remark \ref{remRangeOfQ}.

The proof of Theorem \ref{theoAintro} follows the architecture of
\cite[Lemma 3.1]{BennewitzLewis}, and the one point where it departs from it
is worth isolating, because it is what buys the generality.

Bennewitz and Lewis run their argument not on the density $g$ but on the
auxiliary function
$$f(z)=\left(\frac{\delta(z)}{2\varrho}\right)^{d-1}g(z),
\qquad \delta(z)=2\varrho-|y-z|,$$
where $\delta(z)$ is the distance from $z$ to the boundary of the reference
ball $B(y,2\varrho)$. The weight $\delta(z)^{d-1}$ is available to them only
because $\partial G$ is Ahlfors $(d-1)$-regular: it is really the measure of a
ball of radius comparable to the distance to the boundary. When
$\mu$ is merely doubling there is no exponent to which to raise the distance
--- but a Whitney decomposition supplies the ball directly. We use
$$\phi(x):=\sum_{Q\in\W}\chi_Q(x)\,\mu(Q),$$
where $\W$ is a Whitney decomposition of $V$ adapted to the dyadic grid, so
that $\mu(Q)\approx\mu\big(\Delta_{\dist(x,X\setminus V)}(x)\big)$ for
$x\in Q$. 

The rest of the argument is then a Calder\'on--Zygmund iteration on $\phi w$
over each Whitney cube, producing an estimate for the level sets of $\phi w$,
followed by Lemma \ref{lemDintro}. Taking $p=1$ gives Theorem
\ref{theoAintro}; taking $p>1$ gives Theorem \ref{theoCintro}.

A dyadic grid on a closed set $X$ has no canonical notion of a triple cube,
and the classes considered here all involve a dilation, so one is needed. The
standard device is to work with several adjacent dyadic systems in the sense
of Hyt\"onen and Kairema \cite{HytonenKairema}, as
\cite[\S4]{AndersonHytonenTapiola} do. We take a different route: we call two
cubes of the same generation neighbours when suitable dilations of their
enveloping balls meet, define $3Q$ as the union of the neighbours of $Q$, and
work throughout with the basis $\DD\cup3\DD$. Section \ref{secGrids} develops
the properties of this basis that the later sections need; see Remark \ref{remAdjacentSystems}.

\subsection*{Organization}

Section \ref{secGrids} sets up the dyadic grid, the neighbour relation and the
triple cube, and constructs the Whitney decomposition; no measure enters there
beyond the requirement that $X$ be closed. Section \ref{secCZ} introduces the
local doubling condition we work with, proves the dyadic differentiation
theorem --- which in this setting needs no hypothesis --- and the
Calder\'on--Zygmund iteration on $\DD\cup3\DD$. Sections
\ref{secClasses}--\ref{secBalls} contain the classes themselves: their
definitions and the elementary implications between them, summarized in Figure
\ref{figImplications}; the equivalences valid on $\DD\cup3\DD$ when $\mu$ is
doubling; and the transfer of those equivalences to restricted balls. Section
\ref{secWeak} defines the three weak classes and proves the two implications
between them that are immediate. Section \ref{secBL} proves Theorem
\ref{theoAintro} and Corollary \ref{coroBintro}; Section \ref{secBootstrap}
isolates Lemma \ref{lemDintro}, which is a statement of real-variable theory
with no geometry in it; and Section \ref{secGehring} proves Theorem
\ref{theoCintro}. Section \ref{secOpen} collects the questions we leave open.

\section{Dyadic grids, triple cubes and a Whitney decomposition}
\label{secGrids}

\subsection{Dyadic grids adapted to a closed set}

Throughout the paper $X\subset\R^d$ is a closed set. By a \emph{dyadic grid}
associated to $X$ with constants $1<\rho_0,A_0<\infty$ we mean a family
$\DD=\bigcup_{k\in\Z}\DD_k$ of Borel subsets of $X$ with the following
properties.
\begin{enumerate}[D1.]
	\item \label{itemD1} \emph{Completeness}: $X=\bigcup_{Q\in\DD_k}Q$ for
	every $k\in\Z$.
	\item \label{itemD2} \emph{Nesting}: if $Q_j\in\DD_{k_j}$ for $j=0,1$ with
	$k_0\leq k_1$, then either $Q_1\subset Q_0$ or $Q_1\cap Q_0=\varnothing$.
	\item \label{itemD3} \emph{Tree structure}: every $Q\in\DD_{k}$ is
	contained in a unique $\widehat Q\in\DD_{k-1}$, its \emph{parent}, and the
	number of children of a cube is at least one and at most a dimensional
	constant.
	\item \label{itemD4} \emph{Scaling}: for every $Q\in\DD_k$ there is a point
	$z_Q\in Q$, the \emph{centre} of $Q$, such that
	$$\Delta_Q\subset Q\subset B_Q\subset B_{\widehat Q},
	\qquad\text{where}\quad
	\Delta_Q\coloneqq \Delta_{\frac12\rho_0^{-k}}(z_Q),
	\quad B_Q\coloneqq B_{\frac{A_0}2\rho_0^{-k}}(z_Q).$$
\end{enumerate}
We call $k$ the \emph{generation} of $Q\in\DD_k$ and write
$$\ell(Q)\coloneqq \rho_0^{-k},\qquad r_Q\coloneqq \tfrac{A_0}2\rho_0^{-k}
=\tfrac{A_0}2\ell(Q)$$
for its \emph{side length} and its \emph{exterior radius}, so that $B_Q$ is
the ball of radius $r_Q$ centred at $z_Q$ and $\Delta_Q$ is the restricted
ball of radius $A_0^{-1}r_Q$ centred at the same point. Given
$Q_0\in\DD$ we write $\DD(Q_0)\coloneqq \{Q\in\DD:\, Q\subset Q_0\}$, and
$\mathcal{P}(Q)\coloneqq \widehat Q$ for the parent map.

Such grids exist on every closed $X\subset\R^d$, with $\rho_0$ and $A_0$
universal; this is the construction of Christ \cite{Christ} and David \cite{DavidCubes}, in
the form given by Hyt\"onen and Kairema \cite{HytonenKairema} for doubling
metric spaces. The universality of the constants can be obtained as in \cite{DavidMattila}; it is explained in detail in \cite[Appendix]{PratsTolsaLlibre}. Nothing below uses any property of $\DD$ beyond
D\ref{itemD1}--D\ref{itemD4}.

\begin{remark}\label{remSkipping}
Passing from $\DD$ to the subfamily $\bigcup_{k\in j\Z}\DD_{k}$ of every
$j$-th generation again yields a grid, with $\rho_0$ replaced by $\rho_0^{j}$
and $A_0$ unchanged. We refer to this as \emph{skipping generations}. It lets
us assume that $\rho_0$ is as large as we please; we shall use it once, to
secure $\rho_0\geq 6$ in Remark \ref{rem3QNeighbors}.
\end{remark}

\subsection{Neighbours and triple cubes}

The weak classes of weights we are concerned with all compare an average over a set
with an average over a dilate of it. A dyadic grid on a closed set carries no
canonical dilation, and we now supply one.

\begin{definition}\label{defNeighbors}
Let $\lambda\geq1$. Two cubes of the same generation $Q,R\in\DD_k$ are
\emph{neighbours} if $\lambda B_Q\cap\lambda B_R\neq\varnothing$, and we write
$R\in\NN(Q)$. The \emph{triple cube} of $Q$ is
$$3Q\coloneqq \bigcup_{R\in\NN(Q)}R,$$
and we write $3Q\in3\DD$. Since $Q\in\NN(Q)$ and every $R\in\NN(Q)$ satisfies
$R\subset B_R\subset B(z_Q,\lambda r_Q+2r_R)$, we have
\begin{equation}\label{eqTrivialInclusionDQP}
	\lambda B_Q\cap X\subset 3Q\subset 3\lambda B_Q .
\end{equation}
\end{definition}

For the standard dyadic grid of $\R^d$ the definition is of course replaced by
its natural counterpart: $3Q$ is the concentric cube with triple side length,
and the neighbours of $Q$ are the cubes of its generation sharing part of its
boundary.

\begin{remark}\label{rem3QNeighbors}
The neighbouring relation satisfies
\begin{enumerate}[N1.]
	\item \label{eqrem3QNeighbors0} $P\in\NN(Q)\implies Q\subset 3P$.
\end{enumerate}
Moreover, we may fix $\lambda$ depending on $\rho_0$ so that, skipping
generations if necessary, the following also hold:
\begin{enumerate}[N1.]
	\setcounter{enumi}{1}
	\item \label{eqrem3QNeighbors2} $\left(\frac4{A_0}+1\right)B_Q\cap X\subset 3Q;$
	\item \label{eqrem3QNeighbors1} $\widehat Q\subset 3Q\subset 3\widehat Q;$
	\item \label{eqrem3QNeighbors3} $\lambda B_Q\subset\lambda B_{\widehat Q};$
	\item \label{eqrem3QNeighbors4} if $Q,P\in\DD$ with $Q\cap P=\varnothing$
	and $z_Q\in3\lambda B_P$, then $3\lambda B_Q\subset 10\lambda^2A_0B_P;$ and
	\item \label{eqrem3QNeighbors5} if $P,R\in\DD_k$ with $R\in\NN(P)$, and
	$R\subset R'\in\DD$ with $r_{R'}\geq\rho_0 r_P$, then $3P\subset 3R'$.
\end{enumerate}
From now on $A_0$, $\rho_0$ and $\lambda$ are dimensional constants, and
dependence on them is not displayed.
\end{remark}

\begin{proof}
N\ref{eqrem3QNeighbors0} is immediate from the symmetry of Definition
\ref{defNeighbors}.

By Remark \ref{remSkipping} we may assume $\rho_0\geq 6$. Set
\begin{equation}\label{defLambda}
	\lambda\coloneqq 2\max\left\{(1-3\rho_0^{-1})^{-1},\,\rho_0\right\}\geq 6 .
\end{equation}
Property N\ref{eqrem3QNeighbors2} follows at once from
\rf{eqTrivialInclusionDQP}.

Let us prove N\ref{eqrem3QNeighbors1}. We first check that
$\lambda\geq(1-3\rho_0^{-1})^{-1}$ implies $3Q\subset3\widehat Q$. Let
$x\in3Q\cap R$ with $R,\widehat Q\in\DD_{k}$. To see that $R$ and $\widehat Q$
are neighbours it suffices to show that $x\in\lambda B_{\widehat Q}$, and
indeed
$$|x-z_{\widehat Q}|\leq|x-z_Q|+|z_Q-z_{\widehat Q}|
\by{\rf{eqTrivialInclusionDQP}}<3\lambda r_Q+r_{\widehat Q}
=(3\lambda\rho_0^{-1}+1)r_{\widehat Q}
\by{\rf{defLambda}}\leq\lambda r_{\widehat Q}.$$
For the inclusion $\widehat Q\subset3Q$, assume $\lambda>\rho_0$ and let
$R\in\DD_k$ with $\widehat R=\widehat Q$. Then
$$|z_Q-z_R|\leq\diam(B_{\widehat Q})=2\rho_0 r_Q=2\rho_0 r_R ,$$
so $\lambda B_Q\cap\lambda B_R\neq\varnothing$; note that the intersection
point need not belong to $X$.

The remaining properties are proved by the same techniques and are of similar difficulty, so we leave them to the reader.
%
%
\end{proof}

\begin{remark}\label{remAdjacentSystems}
The usual way of supplying a dilation on a closed set is different, and it is
worth recording what each device costs. Hyt\"onen and Kairema
\cite{HytonenKairema} produce finitely many \emph{adjacent} dyadic systems
$\DD^1,\dots,\DD^K$ with the property that every ball $B$ is contained in some
cube of some $\DD^t$ of comparable size; a condition on balls then becomes a
condition on ``some cube of some system''. This is the route taken in
\cite[\S4]{AndersonHytonenTapiola}, where the enlargement of a cube $Q$ is a
\emph{generalized dyadic parent}: a cube $Q^{*}$ of a coarser generation, with
$\ell(Q^{*})=\delta^{-2}\ell(Q)$, containing every cube of the generation of
$Q$ that meets $Q$. Since such a $Q^{*}$ is not unique, the resulting classes
are defined by a supremum over $Q$ and an infimum over the admissible $Q^{*}$.

Working with $\DD\cup3\DD$ instead has three consequences. First, everything
takes place in a single grid, and $3Q$ is determined by $Q$, so no infimum
over choices of enlargement enters the definitions. Second, the ratio between the diameters of
$3Q$ and $Q$ is governed by $\lambda$, which \rf{defLambda} ties to $\rho_0$,
whereas a generalized dyadic parent is always two generations coarser. Third,
and this is what we actually need below, $3Q$ is a \emph{union of cubes of the
generation of $Q$}. The proof in \S\ref{secBL} uses exactly this:
the stopping condition of the Calder\'on--Zygmund decomposition is read off on
each $R\in\NN(P)$ separately, and the estimate
$\nu(3P)\leq\nu(P)\bigl(1+C\mu(3P)/\mu(P)\bigr)$ is obtained by summing
$\mu(R)$ over $R\in\NN(P)$. A single coarser cube does not decompose in this
way. The price is that $3Q$ is not itself a dyadic cube, so the family $3\DD$
must be carried alongside $\DD$ as a second layer of the basis, and its
localizations $3\DD(Q_0)=\{3Q\cap Q_0:\, Q\in\DD(Q_0)\}$ appear throughout.
Both features survive the localization, and this is what makes the second
layer usable: by D\ref{itemD2} the set $3Q\cap Q_0$ is again a union of cubes
of the generation of $Q$ whenever $Q\subset Q_0$, and by N\ref{eqrem3QNeighbors1} it still contains
the parent $\widehat Q$, as soon as $\ell(Q)<\ell(Q_0)$.
\end{remark}

We close the subsection with a selection lemma that will be used twice, in
\S\ref{secBL} and in \S\ref{secGehring}, to replace a disjoint family of cubes
by a subfamily whose triples still cover the same set and now overlap boundedly.
It is the one point where the Besicovitch covering theorem is used, and it is
available because we are in $\R^d$.

\begin{lemma}\label{lemBesicovitchSelection}
Let $\mathcal F\subset\DD$ be a family of pairwise disjoint cubes contained in
a bounded set. Then there is a subfamily $\mathcal G\subset\mathcal F$ with
$$\bigcup_{P\in\mathcal F}3P\subset\bigcup_{P\in\mathcal G}10\lambda^2A_0B_P
\qquad\text{and}\qquad\sum_{P\in\mathcal G}\chi_{3P}\leq C_{\mathrm B} ,$$
where $C_{\mathrm B}$ depends only on $d$.
\end{lemma}

\begin{proof}
The centres $\{z_P\}_{P\in\mathcal F}$ form a bounded set, so the Besicovitch
covering theorem applied to the balls $\{3\lambda B_P\}_{P\in\mathcal F}$
yields $\mathcal G\subset\mathcal F$ with
$$\{z_P\}_{P\in\mathcal F}\subset\bigcup_{P\in\mathcal G}3\lambda B_P
\qquad\text{and}\qquad\sum_{P\in\mathcal G}\chi_{3\lambda B_P}\leq C_{\mathrm B} .$$
Given $P\in\mathcal F$, pick $P'\in\mathcal G$ with $z_P\in3\lambda B_{P'}$.
If $P'=P$ then $3P\subset3\lambda B_P\subset 10\lambda^2A_0B_P$ by
\rf{eqTrivialInclusionDQP}; and if $P'\neq P$ the two cubes are disjoint, so
N\ref{eqrem3QNeighbors4} gives
$3P\subset3\lambda B_P\subset10\lambda^2A_0B_{P'}$. The second assertion
follows from $3P\subset3\lambda B_P$.
\end{proof}

\subsection{A Whitney decomposition}

We now decompose a relatively open subset of $X$ into dyadic cubes whose sizes
are comparable to their distances to the complement. No doubling assumption is
made on any measure here; the construction is purely metric.

\begin{lemma}[Whitney decomposition]\label{lemWhitney}
Let $X\subset\R^d$ be closed, let $\DD$ be a dyadic grid associated to $X$
with $\rho_0$, $A_0$ and $\lambda\geq6$ as in Remark \ref{rem3QNeighbors}, and
let $\Omega\subset X$ be relatively open with $F\coloneqq X\setminus\Omega$
non-empty. Then there exist $\alpha=\alpha(\rho_0)>7$ and a collection
$\W\subset\DD$, a \emph{Whitney decomposition} of $\Omega$, such that:
\begin{enumerate}[Wh1.]
	\item \label{itemWhFull} $\Omega=\bigcup_{Q\in\W}Q$.
	\item \label{itemWhDisjoint} The cubes of $\W$ are pairwise disjoint, i.e.
	$\sum_{Q\in\W}\chi_Q=\chi_\Omega$.
	\item \label{itemWhRadius} $(\alpha\lambda-1)r_Q\leq\dist(Q,F)\leq
	\rho_0(\alpha\lambda+1)r_Q$ for every $Q\in\W$, and
	$3\widehat Q\subset \Omega$.
	\item \label{itemWhComparable} If $Q,R\in\W$ satisfy
	$3Q\cap3R\neq\varnothing$, then
	$\rho_0^{-1}\leq\frac{\ell(Q)}{\ell(R)}\leq\rho_0$.
	\item \label{itemWhOverlapping} The triple cubes have bounded overlap:
	$\sum_{Q\in\W}\chi_{3Q}\leq C_\W$, with $C_\W^{1/d}$ depending only on
	$\rho_0$, $A_0$ and $\lambda$.
\end{enumerate}
\end{lemma}

\begin{proof}
Let $\W$ be the family of maximal dyadic cubes satisfying
\begin{equation}\label{eqStoppingTimeWhitney}
	\alpha\lambda r_Q\leq\dist(z_Q,F).
\end{equation}
Property Wh\ref{itemWhFull} follows from D\ref{itemD1} and D\ref{itemD4},
since for $x\in\Omega$ the cubes $Q_k(x)\in\DD_k$ containing $x$ satisfy
$r_{Q_k(x)}=\frac{A_0}2\rho_0^{-k}\to0$ while
$\dist(z_{Q_k(x)},F)\to\dist(x,F)>0$ as $k\to\infty$. Property
Wh\ref{itemWhDisjoint} holds by construction and D\ref{itemD2}.

The lower bound in Wh\ref{itemWhRadius} follows from
\rf{eqStoppingTimeWhitney} and the triangle inequality. For the upper bound,
maximality means that $\widehat Q$ fails \rf{eqStoppingTimeWhitney}, so
\begin{equation}\label{eqDistCQAbove}
	\dist(z_Q,F)\leq\dist(z_{\widehat Q},F)+r_{\widehat Q}
	\by{\rf{eqStoppingTimeWhitney}}<(\alpha\lambda+1)r_{\widehat Q}
	=(\alpha\lambda+1)\rho_0 r_Q ,
\end{equation}
and the first part of Wh\ref{itemWhRadius} follows since
$\dist(Q,F)\leq\dist(z_Q,F)$. For the last part, assume
$\alpha\geq3\rho_0+1$ and let $x\in3\widehat Q$. Since
$3\widehat Q\subset3\lambda B_{\widehat Q}$ by \rf{eqTrivialInclusionDQP} and
$z_Q\in Q\subset B_{\widehat Q}$,
$$\begin{aligned}
|x-z_Q|&\leq(3\lambda+1)r_{\widehat Q}=(3\lambda+1)\rho_0 r_Q
=\left(3\rho_0+\frac{\rho_0}\lambda\right)\lambda r_Q\\
&<\alpha\lambda r_Q\by{\rf{eqStoppingTimeWhitney}}\leq\dist(z_Q,F),
\end{aligned}$$
because $\lambda\geq2\rho_0$ by \rf{defLambda}; hence $x\notin F$.

For Wh\ref{itemWhComparable}, assume $3R\cap3Q\neq\varnothing$ with
$r_R\leq r_Q$, and pick $x\in B(z_Q,3\lambda r_Q)\cap B(z_R,3\lambda r_R)$,
which exists by \rf{eqTrivialInclusionDQP}. Then
$$\dist(z_R,F)\geq\dist(z_Q,F)-(3\lambda r_Q+3\lambda r_R)
\by{Wh\ref{itemWhRadius}}\geq(\alpha\lambda-1-6\lambda)r_Q
\geq(\alpha-7)\lambda r_Q .$$
Combining with \rf{eqDistCQAbove},
$$\frac{r_R}{r_Q}\geq\frac{(\alpha-7)\lambda}{\alpha\lambda+1}\rho_0^{-1}
\geq\frac{\alpha-7}{\alpha+1}\rho_0^{-1}>\rho_0^{-2}$$
provided $\alpha$ is large enough. Since $R$ and $Q$ are dyadic cubes with
$r_R\leq r_Q$, we have $\frac{\ell(R)}{\ell(Q)}=\frac{r_R}{r_Q}=\rho_0^{-j}$
for some $j\in\N$, and the estimate above forces $j\leq1$.

It remains to prove Wh\ref{itemWhOverlapping}. Let $x\in\Omega$ and let
$Q\in\W$ be a cube of maximal $r_Q$ with $x\in3Q$; such a cube exists by
Wh\ref{itemWhComparable}. Write
$D_P\coloneqq B\bigl(z_P,\tfrac12 r(\Delta_P)\bigr)$ for the Euclidean ball
with half the radius of the inner ball of $P$. For any $P\in\W$ with
$x\in3P$ we have, using \rf{eqTrivialInclusionDQP} and
Wh\ref{itemWhComparable},
$$D_P\subset B\bigl(z_Q,3\lambda r_Q+3\lambda r_P+r(\Delta_P)\bigr)
\subset B\bigl(z_Q,(6\lambda+A_0^{-1})r_Q\bigr),$$
and moreover
$$r(\Delta_P)=A_0^{-1}r_P\by{Wh\ref{itemWhComparable}}
\geq A_0^{-1}\rho_0^{-1}r_Q .$$
The balls $\{D_P\}_{P\in\DD_k}$ are pairwise disjoint: if $P,P'\in\DD_k$ are
distinct, then $z_{P'}\in P'\subset X\setminus P$, so $z_{P'}\notin\Delta_P$
and hence $|z_P-z_{P'}|\geq r(\Delta_P)=r(\Delta_{P'})$. Comparing Lebesgue
measures we obtain
$$\begin{aligned}
\#\left\{P\in\W:\ 3P\cap3Q\neq\varnothing,\ r_P=\rho_0^{-1}r_Q\right\}
&\leq\frac{\bigl|B(z_Q,(6\lambda+A_0^{-1})r_Q)\bigr|}
{|B_1(0)|\left(\tfrac12 r(\Delta_P)\right)^d}\\
&\leq\bigl(2\rho_0(6\lambda A_0+1)\bigr)^d ,
\end{aligned}$$
and the same argument with $r_P=r_Q$ gives the analogous bound. By
Wh\ref{itemWhComparable} these are the only two possibilities, so
Wh\ref{itemWhOverlapping} holds with
$C_\W\leq(12\lambda A_0+2)^d(1+\rho_0)^d$.
\end{proof}

\section{Doubling measures and Calder\'on--Zygmund decompositions}
\label{secCZ}

\subsection{Doubling and $\DD$-doubling measures}

\begin{definition}\label{defDoubling}
A Radon measure $\mu$ in $\R^d$ is \emph{doubling} with constant $C_\mu$ if
$$\mu(B_{2r}(x))\leq C_\mu\,\mu(B_r(x))
\qquad\text{for every }x\in\supp\mu\text{ and every }r>0 .$$
\end{definition}

From now on $\mathcal D$ is a dyadic grid associated to $X=\supp\mu$. Given
$Q_0\in\DD\cup3\DD$ we write
$$\DD(Q_0)\coloneqq \{Q\in\DD:\, Q\subset Q_0\}
\qquad\text{and}\qquad
3\DD(Q_0)\coloneqq \{3Q\cap Q_0:\, Q\in\DD(Q_0)\} .$$

The condition we actually use is weaker than doubling, and local.

\begin{definition}\label{defDoublingConstant}
Let $\mu$ be a Radon measure in $\R^d$ and let $Q_0\in\DD\cup3\DD$. We say
that $\mu$ is \emph{$\DD$-doubling in $Q_0$} with constant $C_\mu$ if
$$\sup_{R\in\NN(Q)\cap\DD(Q_0)}\frac{\mu(3Q\cap Q_0)}{\mu(R)}\leq C_\mu
\qquad\text{for every }Q\in\DD(Q_0).$$
\end{definition}

\begin{lemma}\label{lemDoublingNeighbors}
Let $\mu$ be a Radon measure in $\R^d$ and let $Q_0\in\DD\cup3\DD$. Then $\mu$
is $\DD$-doubling in $Q_0$ with constant $C_\mu$ if and only if
\begin{equation}\label{eqDoublingComplement}
	\mu(3Q\cap Q_0\setminus R)\leq\left(1-\frac1{C_\mu}\right)\mu(3Q\cap Q_0)
\end{equation}
for every $Q\in\DD(Q_0)$ and every $R\in\NN(Q)\cap\DD(Q_0)$.
In that case:
\begin{enumerate}[DN1.]
	\item \label{itemDN1} $\mu(3S)\leq C_\mu\mu(T)$ for every $S\in\DD$ with
	$3S\subset Q_0$ and every $T\in\NN(S)$;
	\item \label{itemDN2} $\mu(\widehat S)\leq C_\mu\mu(S)$ for every $S\in\DD$
	with $\widehat S\subset Q_0$;
	\item \label{itemDN3} $\mu$ is $\DD$-doubling in $Q_1$ for every
	$Q_1\in\DD(Q_0)\cup3\DD(Q_0)$.
\end{enumerate}
Moreover, if $\mu$ is doubling with constant $C_\mu$ in the sense of
Definition \ref{defDoubling}, then it is $\DD$-doubling in every
$Q_0\in\DD\cup3\DD$ with constant $C_\mu^{N_0}$, where
$N_0=\lceil\log_2(5\lambda A_0)\rceil$.
\end{lemma}

\begin{proof}
Since $R\in\NN(Q)\cap\DD(Q_0)$ implies $R\subset 3Q\cap Q_0$, and
$\mu(3Q\cap Q_0)<\infty$, we have
$\mu(3Q\cap Q_0\setminus R)=\mu(3Q\cap Q_0)-\mu(R)$. Hence
$\mu(3Q\cap Q_0)\leq C_\mu\mu(R)$ holds if and only if
$\mu(3Q\cap Q_0\setminus R)\leq(1-C_\mu^{-1})\mu(3Q\cap Q_0)$, and the
equivalence follows on taking the supremum over $R$.

Properties DN\ref{itemDN1}, DN\ref{itemDN2} and DN\ref{itemDN3} follow easily by applying the definitions and the tree structure of $\DD$.

For the last assertion, let $R\in\NN(Q)$. By Definition \ref{defNeighbors} we
have $|z_Q-z_R|\leq2\lambda r_Q$, so \rf{eqTrivialInclusionDQP} gives
$3Q\cap Q_0\subset3\lambda B_Q\subset B(z_R,5\lambda r_Q)$, while
$\Delta_R\subset R$ has radius $A_0^{-1}r_Q$. Iterating Definition
\ref{defDoubling} $N_0$ times,
$$\mu(3Q\cap Q_0)\leq\mu\bigl(B(z_R,5\lambda r_Q)\bigr)
\leq C_\mu^{N_0}\mu(\Delta_R)\leq C_\mu^{N_0}\mu(R).\qedformula$$\qedOut
\end{proof}

In view of the last assertion, when $\mu$ is doubling we write $C_\mu$ for the larger of the two
constants, and we no longer distinguish them. The same iteration, combined
with N\ref{eqrem3QNeighbors2}, gives a constant $\wt C_\mu$ depending only on
$C_\mu$ and on the parameters of $\DD$ such that
\begin{equation}\label{eqEnlargedBallDoubling}
	\mu\left(10\lambda^2A_0B_P\right)\leq\wt C_\mu\,\mu(3P)
	\qquad\text{for every }P\in\DD ;
\end{equation}
this is the form in which the enlarged balls of Lemma
\ref{lemBesicovitchSelection} will be used.

\begin{lemma}\label{lemNeighborWhitney}
Let $\mu$ be a Radon measure in $\R^d$, let $\DD$ be a dyadic grid associated
to $X=\supp\mu$ with $\lambda\geq6$ as in Remark \ref{rem3QNeighbors}, and let
$U\subset X$ be relatively open with non-empty relative complement, such that
$\mu$ is $\DD$-doubling in every $R\in\DD\cup3\DD$ with $R\subset U$. Let $\W$
be a Whitney decomposition of $U$. If $Q\in\W$ and $P\subsetneq Q$, then every
$\wt Q\in\W$ with $\wt Q\cap3P\neq\varnothing$ satisfies
$$C_\mu^{-2}\mu(Q)\leq\mu(\wt Q)\leq C_\mu^{2}\mu(Q).$$
\end{lemma}

\begin{proof}
We may assume $\wt Q\neq Q$. Iterating N\ref{eqrem3QNeighbors1} gives
$3P\subset3Q$. Let $x\in\wt Q\cap3P$; by Definition \ref{defNeighbors} there
is a cube $R\in\NN(P)$ with $x\in R$. Since $3P\subset3Q$ and
$\wt Q\subset3\wt Q$, we get $3\wt Q\cap3Q\neq\varnothing$, so
Wh\ref{itemWhComparable} shows that if $Q\in\DD_{k_0}$ then
$$\wt Q\in\DD_{k_0-1}\cup\DD_{k_0}\cup\DD_{k_0+1}.$$
Since $P\subsetneq Q$ we deduce $R\subset\wt Q$, and
N\ref{eqrem3QNeighbors3} gives
$\lambda B_Q\cap\lambda B_{\wt Q}\supset\lambda B_P\cap\lambda B_R
\neq\varnothing$. In particular either $Q$ and $\wt Q$ are neighbours, or one
of them is a neighbour of the parent of the other.

By Lemma \ref{lemDoublingNeighbors} we have $\mu(3S)\leq C_\mu\mu(T)$ for
every $T\in\NN(S)$ and $\mu(\widehat S)\leq C_\mu\mu(S)$, whenever
$3S\subset U$; this holds for all the cubes involved, since $\alpha>7$ in
Lemma \ref{lemWhitney}. Applying it to $S=Q$ and $S=\wt Q$, and to their
parents when the generations differ, the lemma follows.
\end{proof}

\subsection{The dyadic differentiation theorem}

The Calder\'on--Zygmund decomposition below identifies the set where a
function exceeds a level with a union of stopping cubes, up to a null set, and
for that we need dyadic averages to converge. In a general space of
homogeneous type this is a delicate matter: Anderson and Weirich take it as a
standing hypothesis and observe that the argument usually quoted for it has
been questioned \cite[Rem.~3.2]{AndersonWeirich}. In the present setting there
is no difficulty, and no doubling is needed.

\begin{lemma}\label{lemDyadicLebesgue}
Let $\mu$ be a Radon measure in $\R^d$, let $X=\supp\mu$ and let $\DD$ be a
dyadic grid associated to $X$. Writing $Q_k(x)$ for the cube of $\DD_k$
containing $x$, every $f\in L^1_{\loc}(\mu)$ satisfies
$$\lim_{k\to\infty}\avint_{Q_k(x)}f\, d\mu=f(x)
\qquad\text{for }\mu\text{-a.e. }x\in X .$$
\end{lemma}

\begin{proof}
It suffices to argue on a fixed $Q_0\in\DD$ with $\mu(Q_0)<\infty$ and
$f\in L^1(\mu|_{Q_0})$. The dyadic maximal operator
$M_{\DD}f(x)\coloneqq \sup_{k}\avint_{Q_k(x)}|f|\,d\mu$ is of weak type
$(1,1)$ with constant $1$, with no hypothesis on $\mu$: the set
$\{M_\DD f>\lambda\}$ is the union of the maximal cubes $Q$ of $\DD(Q_0)$ with
$\avint_Q|f|\,d\mu>\lambda$, which are pairwise disjoint by D\ref{itemD2}, so
$$\mu(\{M_\DD f>\lambda\})=\sum_j\mu(Q_j)
\leq\frac1\lambda\sum_j\int_{Q_j}|f|\,d\mu\leq\frac{\|f\|_{L^1(\mu)}}\lambda .$$
On the other hand, if $g$ is continuous then $\avint_{Q_k(x)}g\,d\mu\to g(x)$
for every $x$, because $x\in Q_k(x)\subset B_{Q_k(x)}$ and
$r_{Q_k(x)}=\frac{A_0}2\rho_0^{-k}\to0$ by D\ref{itemD4}. Since $X$ is closed
in $\R^d$ and $\mu$ is Radon, continuous functions are dense in
$L^1(\mu|_{Q_0})$. The two facts combine in the usual way: given $\ve>0$,
pick $g$ continuous with $\|f-g\|_{L^1(\mu|_{Q_0})}<\ve$ and note that
$$\limsup_{k\to\infty}\left|\avint_{Q_k(x)}f\, d\mu-f(x)\right|
\leq M_\DD(f-g)(x)+|f-g|(x) ,$$
so that the set where the left-hand side exceeds $\lambda$ has measure at most
$4\ve/\lambda$, by the weak type $(1,1)$ bound and Chebyshev's inequality.
Letting $\ve\to0$ and then $\lambda\to0$ gives the conclusion.
\end{proof}

\subsection{The Calder\'on--Zygmund decomposition and its iteration}

We adapt the iterated Calder\'on--Zygmund decomposition (see for instance
\cite[Cor.~9.2.4]{GrafakosModern}) to the basis $\DD\cup3\DD$.

\begin{lemma}[Calder\'on--Zygmund decomposition]\label{lemCZDecomposition}
Let $\mu$ be a Radon measure in $\R^d$ and let $Q_0\in\DD\cup3\DD$ be such
that $\mu$ is $\DD$-doubling in $Q_0$ with constant $C_\mu$. Let
$Q\in\DD(Q_0)\cup3\DD(Q_0)$, of generation $k_0$, let $f\in L^1(Q)$ be
non-negative and let
$$t\geq C_\mu^{-1}\avint_Q f\, d\mu .$$
Then there exist a decomposition $Q=U_t\cup G_t\cup Z_t$ into disjoint
$\mu$-measurable sets and a family of pairwise disjoint cubes
$\mathcal{F}_t\subset\bigcup_{k\geq k_0}\DD_k\cup\{Q\}$ such that
\begin{enumerate}[CZ1.]
	\item \label{itemCZU} $U_t=\bigcup_{R\in\mathcal{F}_t}R$, and every
	$R\in\mathcal{F}_t$ satisfies $t<\avint_R f\, d\mu\leq C_\mu t$;
	\item \label{itemCZG} $f(x)\leq t$ for every $x\in G_t$;
	\item \label{itemCZZ} $\mu(Z_t)=0$.
\end{enumerate}
\end{lemma}

\begin{proof}
Let $\mathcal{F}_t$ be the family of maximal cubes
$R\in\bigcup_{k\geq k_0}\DD_k\cup\{Q\}$ with $R\subset Q$ and
$\avint_R f\, d\mu>t$; let $U_t$ be as in CZ\ref{itemCZU}, let $Z_t$ be the
exceptional set of Lemma \ref{lemDyadicLebesgue} for $f$, and let
$G_t\coloneqq Q\setminus(U_t\cup Z_t)$.

If $t<\avint_Q f\, d\mu$ then $\mathcal{F}_t=\{Q\}$, and CZ\ref{itemCZU} holds
because $t\geq C_\mu^{-1}\avint_Q f\,d\mu$ by hypothesis; so we may assume
$Q\notin\mathcal{F}_t$. Then for $R\in\mathcal{F}_t$ either $\widehat R$ is an
eligible cube at which the stopping condition fails, or $R$ is one of the
building blocks of $Q$ when $Q\in3\DD$. Using in the latter case the
convention $\widehat R=Q$, and Lemma \ref{lemDoublingNeighbors}, we get
$$\avint_R f\, d\mu\leq C_\mu\avint_{\widehat R}f\, d\mu\leq C_\mu t .$$
Conditions CZ\ref{itemCZG} and CZ\ref{itemCZZ} follow from Lemma
\ref{lemDyadicLebesgue}.
\end{proof}

\begin{lemma}[Calder\'on--Zygmund iteration]\label{lemCZIteration}
In the setting of Lemma \ref{lemCZDecomposition}, let $0<\tau<1$ and
$\alpha_0\geq C_\mu^{-1}\avint_Q f\, d\mu$, and put
$\alpha_j\coloneqq \left[C_\mu\tau^{-1}\right]^{j}\alpha_0$ for $j\geq1$. Then
the decompositions $Q=U^j\cup G^j\cup Z^j$ given by Lemma
\ref{lemCZDecomposition} at the levels $\alpha_j$, with
$U^j=\bigcup_{R\in\mathcal{F}^j}R$, satisfy:
\begin{enumerate}[CZ1.]
	\setcounter{enumi}{3}
	\item \label{itemCZIhat} if $R\in\mathcal{F}^j$ then $R\subset\hat R$ for
	some $\hat R\in\mathcal{F}^{j-1}$;
	\item \label{itemCZLayers} $\mu(U^{j+1})\leq\tau\mu(U^j)\leq\cdots\leq
	\tau^{j+1}\mu(U^0)\xrightarrow{j\to\infty}0$;
	\item \label{itemCZLayersByCubes} if $R\in\mathcal{F}^j$ then
	$\mu(R\cap U^{j+1})<\tau\mu(R)$.
\end{enumerate}
\end{lemma}

\begin{proof}
For CZ\ref{itemCZIhat}, note that $R\in\mathcal{F}^{j+1}$ satisfies
$\avint_R f\, d\mu>\alpha_{j+1}>\alpha_j$, so $R$ is contained in some maximal
cube with that property, possibly $R$ itself.

For CZ\ref{itemCZLayersByCubes}, let $R\in\mathcal{F}^j$. Then
$$
\begin{aligned}
	C_\mu\alpha_j
	&\by{CZ\ref{itemCZU}}\geq\avint_R f\, d\mu
	\geq\frac1{\mu(R)}\int_{R\cap U^{j+1}}f\, d\mu
	\by{CZ\ref{itemCZIhat}}=\frac1{\mu(R)}
	\sum_{R'\in\mathcal{F}^{j+1}:\,R'\subset R}\mu(R')\avint_{R'}f\, d\mu\\
	&\by{CZ\ref{itemCZU}}>\frac{\alpha_{j+1}}{\mu(R)}\mu(R\cap U^{j+1}),
\end{aligned}
$$
and CZ\ref{itemCZLayersByCubes} follows since
$\alpha_{j+1}=C_\mu\tau^{-1}\alpha_j$. Finally CZ\ref{itemCZLayers} is
obtained by summing CZ\ref{itemCZLayersByCubes} over $R\in\mathcal{F}^j$ and
using CZ\ref{itemCZIhat}.
\end{proof}

\section{Classes of weights}
\label{secClasses}

Every condition considered in this paper compares an average of $\nu$ with an
average of $\mu$ over the sets of a fixed family. We begin by fixing that
family.

\begin{definition}\label{defBasis}
A \emph{basis} is either the family of restricted balls of $X$ or the family
$\DD\cup3\DD$. In the first case $V\subset X$ denotes a relatively open set and
$$\BB_V\coloneqq \{\Delta\subset V:\, \Delta\text{ a restricted ball}\};$$
in the second, $V$ is a cube $Q_0\in\DD\cup3\DD$ and
$\BB_{Q_0}\coloneqq \DD(Q_0)\cup3\DD(Q_0)$. Statements made for $\BB_V$ with no
further qualification hold in both cases, and when $V=X$ we drop it from the
notation. Throughout the text, $\BB_{\mathcal{D}}$ will refer to the dyadic basis $\mathcal{D}\cup3\mathcal{D}$.
\end{definition}

What follows is a particular case of the formalism of Duoandikoetxea,
Mart\'in-Reyes and Ombrosi \cite[\S2]{DuoandikoetxeaMartinReyesOmbrosi}, who
work over an arbitrary basis in a $\sigma$-finite measure space; we write
$A_\infty$ and $RH_1$ for the two ends of the scale as in
\cite{BeznosovaReznikov}, and label the variants by initials rather than by
numbers as in \cite{MudarraPerfekt}.

\begin{definition}\label{defStrongClasses}
Let $\mu$ be a Radon measure in $\R^d$, let $X=\supp\mu$, let $\BB_V$ be as in
Definition \ref{defBasis}, let $w\in L^1_{\loc}(\mu)$ and let $1<p<\infty$.
\begin{itemize}
	\item $w\in A_p(\mu,\BB_V)$ if $w^{\frac1{1-p}}\in L^1_{\loc}(\mu)$ and
	$$\avint_{B}w\, d\mu\left(\avint_{B}w^{\frac1{1-p}}\, d\mu\right)^{p-1}
	\leq C\qquad\text{for every }B\in\BB_V .$$
	\item $w\in RH_p(\mu,\BB_V)$ if $w\in L^p_{\loc}(\mu)$, $w>0$ $\mu$-a.e.\
	and
	\begin{equation}\label{eqStrongRH}
		\left(\avint_{B}w^p\, d\mu\right)^{\frac1p}\leq C\avint_{B}w\, d\mu
		\qquad\text{for every }B\in\BB_V .
	\end{equation}
\end{itemize}
The least admissible constants are denoted $[w]_{A_p(\mu,\BB_V)}$ and
$[w]_{RH_p(\mu,\BB_V)}$, and we set
$$A_\infty^{U}(\mu,\BB_V)\coloneqq \bigcup_{1<p<\infty}A_p(\mu,\BB_V),
\qquad
RH_1^{U}(\mu,\BB_V)\coloneqq \bigcup_{1<p<\infty}RH_p(\mu,\BB_V),$$
the superscript standing for \emph{union}.
\end{definition}

\begin{definition}\label{defQC}
In the setting of Definition \ref{defStrongClasses}, let $\nu$ be a Radon
measure supported in $X$ and let $0<\delta,\ve<1$. We say that
$\nu\in QC_{\delta,\ve}(\mu,\BB_V)$ if
\begin{equation}\label{eqDefQC}
	\mu(E)<\delta\,\mu(B)\implies\nu(E)\leq\ve\,\nu(B)
	\qquad\text{for every Borel }E\subset B\in\BB_V ,
\end{equation}
and we put $QC(\mu,\BB_V)\coloneqq
\bigcup_{0<\delta,\ve<1}QC_{\delta,\ve}(\mu,\BB_V)$ and
$$\begin{aligned}
A_\infty^{QC}(\mu,\BB_V)&\coloneqq \bigcap_{0<\delta<1}\bigcup_{0<\ve<1}
QC_{\delta,\ve}(\mu,\BB_V),\\
RH_1^{QC}(\mu,\BB_V)&\coloneqq \bigcap_{0<\ve<1}\bigcup_{0<\delta<1}
QC_{\delta,\ve}(\mu,\BB_V).
\end{aligned}$$
Here $QC$ stands for \emph{quantitative (mutual absolute) continuity}, for the
reason given in Remark \ref{remWMutual}.
\end{definition}

The two classes of Definition \ref{defQC} differ only in the order of the
quantifiers, and this is the whole point: $A_\infty^{QC}$ requires that a set
small for $\mu$ be small for $\nu$, the threshold on the $\nu$ side being
chosen \emph{after} the one on the $\mu$ side, while $RH_1^{QC}$ requires the
same with the choice made in the opposite order. For the Lebesgue measure the
two coincide, and coincide with $A_\infty^{U}$ and $RH_1^{U}$
\cite{Muckenhoupt}; here they do not, and the classes of Definition
\ref{defStrongClasses} do not either.

\begin{remark}\label{remWMutual}
Condition \rf{eqDefQC} is a quantitative form of mutual absolute continuity.
First, it forces $\nu|_V\ll\mu|_V$. Indeed, let $E\subset V$ be Borel with
$\nu(E)>0$; applying the Lebesgue differentiation theorem for $\nu$ to
$\chi_E$ --- which is available for restricted balls because $\nu$ is a Radon
measure in $\R^d$, by the Besicovitch covering theorem, and for
$\DD\cup3\DD$ by Lemma \ref{lemDyadicLebesgue} --- we find $B\in\BB_V$,
arbitrarily small, with $\nu(E\cap B)>\ve\,\nu(B)$; then \rf{eqDefQC} gives
$\mu(E)\geq\mu(E\cap B)\geq\delta\,\mu(B)>0$.

Second, replacing $E$ by $B\setminus E$ shows that \rf{eqDefQC} is equivalent
to
\begin{equation}\label{eqDefQCbis}
	\mu(E)>(1-\delta)\,\mu(B)\implies\nu(E)\geq(1-\ve)\,\nu(B)
\end{equation}
for every Borel $E\subset B\in\BB_V$,
and hence, taking contrapositives, to
$\nu(E)<(1-\ve)\nu(B)\implies\mu(E)\leq(1-\delta)\mu(B)$. In particular
$\supp\mu=\supp\nu$ and $\mu|_V\ll\nu|_V$ as well.

We also write $\nu\in RH_\infty(\mu,\BB_V)$ when $\sup_Bw\leq C\avint_Bw\, d\mu$
for every $B\in\BB_V$; this is the strongest of the conditions below, and
$RH_\infty\subset RH_p$ for every $p$.

Consequently $d\nu=w\, d\mu$ with $0<w<\infty$ both $\mu$- and $\nu$-a.e.,
with $w\in L^1_{\loc}(\mu|_V)$ and $w^{-1}\in L^1_{\loc}(\nu|_V)$, and we
write indifferently $\nu\in QC_{\delta,\ve}(\mu,\BB_V)$ or
$w\in QC_{\delta,\ve}(\mu,\BB_V)$; the same convention applies to all the
classes above.
\end{remark}

Throughout, $p'$ denotes the conjugate exponent of $p$.

\begin{lemma}\label{lemTrivialImplications}
Let $\mu$ be a Radon measure in $\R^d$, let $\BB_V$ be as in Definition
\ref{defBasis} and let $\nu$ be a Radon measure supported in $X$. Then
\begin{equation}\label{eqTrivialImplications}
	A_\infty^{U}\subset A_\infty^{QC}\subset QC
	\supset RH_1^{QC}\supset RH_1^{U},
\end{equation}
all classes taken with respect to $(\mu,\BB_V)$, and
\begin{equation}\label{eqDualityWeights}
\begin{aligned}
	\nu\in A_p(\mu,\BB_V)&\iff\mu\in RH_{p'}(\nu,\BB_V),\\
	\nu\in QC_{\delta,\ve}(\mu,\BB_V)&\iff
	\mu\in QC_{1-\ve,1-\delta}(\nu,\BB_V),\\
	\nu\in A_\infty^{QC}(\mu,\BB_V)&\iff\mu\in RH_1^{QC}(\nu,\BB_V).
\end{aligned}
\end{equation}
Moreover, if $w\in RH_p(\mu,\BB_V)$ with constant $C_1$, then
$\nu\in QC_{\delta,\ve}(\mu,\BB_V)$ for every $0<\ve<1$, with
$\delta\coloneqq \left(C_1^{-1}\ve\right)^{p'}$.
\end{lemma}

\begin{proof}
The first line of \rf{eqDualityWeights} is
\cite[Rem.~2.3]{DuoandikoetxeaMartinReyesOmbrosi}, valid over any basis, with
$[w]_{A_p(\mu,\BB_V)}=[w^{-1}]_{RH_{p'}(\nu,\BB_V)}^{p}$; we do not reprove
it. The second line is the equivalence between \rf{eqDefQC} and
\rf{eqDefQCbis} established in Remark \ref{remWMutual}, read with the roles of
$\mu$ and $\nu$ interchanged.

The third line follows from the second by inspecting the quantifiers. Writing
$\delta'=1-\ve$ and $\ve'=1-\delta$, which is an involution of
$(0,1)\times(0,1)$, we get
$$\nu\in\bigcap_{\delta}\bigcup_{\ve}QC_{\delta,\ve}(\mu,\BB_V)
\iff\mu\in\bigcap_{\ve'}\bigcup_{\delta'}QC_{\delta',\ve'}(\nu,\BB_V),$$
because the involution exchanges the two coordinates: a $\delta$ quantified
first becomes an $\ve'$ quantified first, and an $\ve$ quantified second
becomes a $\delta'$ quantified second. 

For the last assertion, let $E\subset B\in\BB_V$ be Borel. By H\"older's
inequality and \rf{eqStrongRH},
$$\frac{\nu(E)}{\nu(B)}=\frac1{\nu(B)}\int_{E}w\, d\mu
\leq\frac{\mu(B)}{\nu(B)}\left(\avint_{B}w^{p}\, d\mu\right)^{\frac1p}
\left(\frac{\mu(E)}{\mu(B)}\right)^{\frac1{p'}}
\leq C_1\left(\frac{\mu(E)}{\mu(B)}\right)^{\frac1{p'}},$$
so $\mu(E)<\delta\mu(B)$ with $\delta=(C_1^{-1}\ve)^{p'}$ gives
$\nu(E)\leq\ve\,\nu(B)$.

Finally we obtain \rf{eqTrivialImplications}. The last assertion just proved
gives $RH_1^{U}\subset RH_1^{QC}$, and the inclusions
$A_\infty^{QC}\subset QC$ and $RH_1^{QC}\subset QC$ are immediate from the
definitions. For the remaining one, $\nu\in A_\infty^{U}(\mu,\BB_V)$ means
$\nu\in A_p(\mu,\BB_V)$ for some $p$, hence $\mu\in RH_{p'}(\nu,\BB_V)\subset
RH_1^{QC}(\nu,\BB_V)$ by the first and last lines of the proof, hence
$\nu\in A_\infty^{QC}(\mu,\BB_V)$ by the third line of
\rf{eqDualityWeights}.
\end{proof}

Figure \ref{figImplications} collects \rf{eqTrivialImplications} together with
the reverse inclusions proved in \S\ref{secDyadic} and with the weak classes of
\S\ref{secWeak}; it is the shape of the whole paper on one page.

\begin{figure}[ht]
\centering
\begin{tikzpicture}[
	>={Stealth[length=1.8mm]},
	scale=0.86, transform shape,
	cls/.style={draw,rounded corners=2pt,inner sep=3pt,font=\small},
	triv/.style={->},
	doub/.style={->,densely dashed},
	thann/.style={->,line width=1pt},
	lbl/.style={font=\scriptsize,inner sep=1.5pt,fill=white},
	row/.style={font=\footnotesize\itshape,anchor=west}
]
	\node[row] at (-0.6,1.5)
		{(a) strong classes over $\DD\cup3\DD$, inside a fixed $Q_0$};
	\node[cls] (aAU)  at (0,0)    {$A_\infty^{U}$};
	\node[cls] (aAQC) at (2.2,0)  {$A_\infty^{QC}$};
	\node[cls] (aQC)  at (5.6,0)  {$QC$};
	\node[cls] (aRQC) at (7.8,0)  {$RH_1^{QC}$};
	\node[cls] (aRU)  at (10.0,0) {$RH_1^{U}$};
	\node[cls] (aQCd) at (3.9,-1.35) {$QC^{d}$};
	\draw[triv] (aAU)  -- (aAQC);
	\draw[triv] (aAQC) -- (aQC);
	\draw[triv] (aRU)  -- (aRQC);
	\draw[triv] (aRQC) -- (aQC);
	\draw[doub] (aAQC) -- node[lbl,anchor=east,pos=0.55]
		{Lem \ref{lemBothDoubling}} (aQCd);
	\draw[triv] (aQCd) -- (aQC);
	\node[lbl] at (4.05,-.4) {(Ex.~\ref{exStaircase})};
	\node[lbl] at (4.3,-.65) {$\subsetneq$};
	\draw[doub] (aQCd.south) to[out=-150,in=-90]
		node[lbl,below]{Thm \ref{theoEqualityA}} (aAU.south);
	\draw[doub] (aQC.south) to[out=-60,in=-120]
		node[lbl,below]{Thm \ref{theoEqualityRH}} (aRU.south);

	\node[row] at (-0.6,-2.95)
		{(b) the same classes over restricted balls};
	\node[cls] (bAU)  at (0,-4.5)    {$A_\infty^{U}$};
	\node[cls] (bAQC) at (2.2,-4.5)  {$A_\infty^{QC}$};
	\node[cls] (bQC)  at (5.6,-4.5)  {$QC$};
	\node[cls] (bRQC) at (7.8,-4.5)  {$RH_1^{QC}$};
	\node[cls] (bRU)  at (10.0,-4.5) {$RH_1^{U}$};
	\node[cls] (bQCd) at (3.8,-6.0) {$QC^{d}$};
	\draw[triv] (bAU)  -- (bAQC);
	\draw[triv] (bAQC) -- (bQC);
	\draw[triv] (bRU)  -- (bRQC);
	\draw[triv] (bRQC) -- (bQC);
	\draw[doub] (bAQC) -- (bQCd);
	\draw[triv] (bQCd) to[bend left=14] (bQC);
	\draw[thann] (bQC) to[bend left=14] (bQCd);
	\node[lbl,anchor=west] at (5.05,-5.6) {Cor \ref{coroThinAnnuli}};
	\node[lbl] at (3.85,-4.9) {(Ex.~\ref{exLadder})};
	\node[lbl] at (4.1,-5.15) {$\subsetneq$};
	\draw[doub] (bQCd.south) to[out=-150,in=-90]
		node[lbl,below]{Prop \ref{propoDoublingAllEquivalent}} (bAU.south);
	\draw[doub] (bQCd.south) to[out=-30,in=-90]
		node[lbl,below,pos=0.6]{Cor \ref{coroDoublingAllEquivalentDual}} (bRU.south);
	\draw[->,densely dotted] (bQC.north) to[out=50,in=130]
		node[lbl,above]{?\,(Rem.~\ref{remQCnotQCD})} (bRQC.north);

	\node[row] at (-0.6,-7.9)
		{(c) weak classes over restricted balls};
	\node[cls] (cRU)  at (1.6,-9.1) {$wRH_1^{U}$};
	\node[cls] (cRQC) at (5.3,-9.1) {$wRH_1^{QC}$};
	\node[cls] (cBL)  at (9.0,-9.1) {$wRH_1^{BL}$};
	\draw[triv] (cRU)  -- node[lbl,above]{H\"older} (cRQC);
	\draw[doub] (cRQC) -- (cBL);
	\draw[doub] (cBL.south) to[out=-120,in=-60]
		node[lbl,above]{Theorem \ref{theoAintro}} (cRU.south);

	\draw[triv] (0.6,-11.6) -- (1.3,-11.6);
	\node[lbl,anchor=west] at (1.4,-11.6) {immediate};
	\draw[doub] (3.7,-11.6) -- (4.4,-11.6);
	\node[lbl,anchor=west] at (4.5,-11.6) {$\mu$ doubling};
	\draw[thann] (7.0,-11.6) -- (7.7,-11.6);
	\node[lbl,anchor=west] at (7.8,-11.6) {$\mu$ doubling and thin annuli};
\end{tikzpicture}
\caption{What collapses, and under which hypothesis. An arrow is an inclusion;
its style records the hypothesis it needs, and its label the result that
proves it. Over cubes, row (a), there are exactly two classes, $QC$ and
$QC^{d}$ ---written $RH(\mu,\BB_{Q_0})$ and $A_\infty(\mu,\BB_{Q_0})$ in
Theorems \ref{theoEqualityRH} and \ref{theoEqualityA}---and the doubling of
$\nu$ is what separates them. Over balls, row
(b), the left-hand collapse survives but the right-hand one keeps the
superscript $^{d}$, and the two classes merge only under thin annuli;
Example \ref{exLadder} shows that without them they do not. Row (a) is
symmetric under the duality \rf{eqDualityWeights}, which reflects it about
$QC$ and exchanges $\mu$ with $\nu$; the asymmetry of the two dashed arrows is
exactly the failure of that reflection to preserve the doubling hypothesis.
Row (c) is the weak theory, where the closing arrow is Theorem
\ref{theoAintro}.}
\label{figImplications}
\end{figure}

\begin{remark}\label{remDMOwarning}
Outside the doubling setting nothing of the sort survives. Duoandikoetxea,
Mart\'in-Reyes and Ombrosi consider eight conditions of this kind over a
general basis, including $A_\infty^{U}$, $RH_1^{U}$ and $QC$, and show that for suitable sparse bases no two of them coincide
\cite[Cor.~6.2]{DuoandikoetxeaMartinReyesOmbrosi}. Related negative and positive results for particular bases are in
\cite{MudarraPerfekt} and in
\cite[Rem.~3.4]{HagelsteinLuqueParissis}.
\end{remark}

\section{The doubling case: dyadic equivalences}
\label{secDyadic}
In this section $\mu$ is $\DD$-doubling in a fixed $Q_0\in\DD\cup3\DD$, and we
show that $\nu\in A_\infty^{U}(\mu,\BB_{Q_0})$ if and only if
$\nu\in QC(\mu,\BB_{Q_0})$ and $\nu $ is $\mathcal{D}$-doubling in $Q_0$, and
that $\nu\in QC(\mu,\BB_{Q_0})$ if and only if
$\nu\in RH_1^{U}(\mu,\BB_{Q_0})$, that is, if and only if
$\nu\in RH_p(\mu,\BB_{Q_0})$ for some $p>1$. Over this basis there are
therefore only two classes, and \rf{eqDefQC} for a single pair
$(\delta_0,\ve_0)$ decides which, since $\delta_0$ large enough implies that
$\nu$ is doubling; see Lemma \ref{lemBothDoubling}.

To start, on the basis $\DD\cup3\DD$, and under a doubling hypothesis on $\mu$ alone,
the right-hand half of \rf{eqTrivialImplications} collapses.

\begin{theorem}\label{theoEqualityRH}
If $\mu$ is $\DD$-doubling in $Q_0\in\DD\cup3\DD$, then
\begin{equation}\label{eqEqualityBpDyadic}
	RH_1^{U}(\mu,\BB_{Q_0})=RH_1^{QC}(\mu,\BB_{Q_0})=QC(\mu,\BB_{Q_0}),
\end{equation}
and we write $RH(\mu,\BB_{Q_0})$ for this class.
\end{theorem}

This is immediate from \rf{eqTrivialImplications} and the following theorem,
which is where the Calder\'on--Zygmund iteration of \S\ref{secCZ} is used for
the first time.

\begin{theorem}\label{theoQCinfty}
Let $\mu$ and $\nu$ be Radon measures in $\R^d$ with
$\supp\nu\subset\supp\mu$, let $Q_0\in\DD\cup3\DD$, let $\mu$ be
$\DD$-doubling in $Q_0$ with constant $C_\mu$, and let $0<\delta_0,\ve_0<1$
be such that $\nu\in QC_{\delta_0,\ve_0}(\mu,\BB_{Q_0})$. Then the density
$w=\frac{d\nu}{d\mu}$ lies in $RH_p(\mu,\BB_{Q_0})$ whenever
\begin{equation}\label{eqPZero}
	1<p<p_0\coloneqq 1+\frac{\log(1/\ve_0)}{\log(C_\mu/\delta_0)} ,
\end{equation}
and, writing $p-1=\theta\,(p_0-1)$ with $0<\theta<1$,
$$\left(\avint_{Q}w^p\, d\mu\right)^{\frac1p}
\leq\left(1+\frac{\ve_0^{-\theta}}{1-\ve_0^{1-\theta}}\right)^{\frac1p}
\avint_{Q}w\, d\mu\qquad\text{for every }Q\in\BB_{Q_0} .$$
\end{theorem}

\begin{proof}
Fix $Q\in\BB_{Q_0}=\DD(Q_0)\cup3\DD(Q_0)$ and apply Lemma
\ref{lemCZIteration} with $f=w$, with
$\alpha_0=\avint_{Q}w\, d\mu=\frac{\nu(Q)}{\mu(Q)}$ and with
$\tau=\delta_0$, so that $\alpha_j=\left(C_\mu\delta_0^{-1}\right)^{j}\alpha_0$.
For $R\in\mathcal F^j$, property CZ\ref{itemCZLayersByCubes} reads
$\mu(R\cap U^{j+1})<\delta_0\,\mu(R)$, so \rf{eqDefQC} applied to
$E=R\cap U^{j+1}\subset R\in\BB_{Q_0}$ gives
$\nu(R\cap U^{j+1})\leq\ve_0\,\nu(R)$; summing over the disjoint family
$\mathcal F^j$,
\begin{equation}\label{eqNuLayers}
	\nu(U^{j+1})\leq\ve_0\,\nu(U^{j})\leq\cdots\leq\ve_0^{\,j+1}\nu(U^{0}) .
\end{equation}
By CZ\ref{itemCZZ} the sets $Q\setminus U^0$ and $U^j\setminus U^{j+1}$,
$j\geq0$, cover $Q$ up to a set that is null for both measures, and on each of
them CZ\ref{itemCZG} bounds $w$ by $\alpha_0$ and by $\alpha_{j+1}$
respectively. Hence
$$\begin{aligned}
\int_{Q}w^p\, d\mu
&=\int_{Q\setminus U^{0}}w^{p-1}w\, d\mu
+\sum_{j\geq0}\int_{U^{j}\setminus U^{j+1}}w^{p-1}w\, d\mu\\
&\leq\alpha_0^{p-1}\nu(Q\setminus U^{0})
+\sum_{j\geq0}\alpha_{j+1}^{p-1}\,\nu(U^{j})\\
&\by{\rf{eqNuLayers}}\leq\alpha_0^{p-1}
\left(1+\sum_{j\geq0}\left(C_\mu\delta_0^{-1}\right)^{(j+1)(p-1)}
\ve_0^{\,j}\right)\nu(Q) .
\end{aligned}$$
Now $p-1=\theta(p_0-1)$ is exactly
$\left(C_\mu\delta_0^{-1}\right)^{p-1}=\ve_0^{-\theta}$, so the series is
geometric of ratio $\ve_0^{1-\theta}<1$ and sums to
$\ve_0^{-\theta}\left(1-\ve_0^{1-\theta}\right)^{-1}$. Since
$\alpha_0=\frac{\nu(Q)}{\mu(Q)}$, dividing by $\mu(Q)$ gives the stated
inequality.
\end{proof}

Dualizing \rf{eqEqualityBpDyadic} by means of \rf{eqDualityWeights} requires
keeping track of which of the two measures is doubling, and this is what the
following notation is for. For
$\CC\in\{A_p,A_\infty^{U},A_\infty^{QC},QC,QC_{\delta,\ve}\}$ and
$Q_0\in\DD\cup3\DD$ we write
$$\CC^{d}(\mu,\BB_{Q_0})\coloneqq
\{\nu\in \CC(\mu,\BB_{Q_0}):\, \nu\text{ is }\DD\text{-doubling in }Q_0\} .$$
With this notation, applying \rf{eqDualityWeights} to \rf{eqEqualityBpDyadic}
gives, for $\mu$ $\DD$-doubling in $Q_0$,
\begin{equation}\label{eqEqualityApDyadic}
	A_\infty^{U,d}(\mu,\BB_{Q_0})=A_\infty^{QC,d}(\mu,\BB_{Q_0})
	=QC^{d}(\mu,\BB_{Q_0}).
\end{equation}
The superscript $^{d}$ can then be removed from the middle term, and this is
the point at which the two halves of \rf{eqTrivialImplications} cease to be
mirror images: a weight in $A_\infty^{QC}$ is automatically doubling, whereas
one in $RH_1^{QC}$ need not be.

\begin{lemma}\label{lemBothDoubling}
Let $\mu$ be a Radon measure in $\R^d$, $\DD$-doubling in
$Q_0\in\DD\cup3\DD$ with constant $C_\mu$, and let
$\nu\in QC_{\delta_0,\ve_0}(\mu,\BB_{Q_0})$ with
$1-C_\mu^{-1}<\delta_0<1$ and $0<\ve_0<1$. Then $\nu$ is $\DD$-doubling in
$Q_0$, with constant $(1-\ve_0)^{-1}$. In particular
$$A_\infty^{QC,d}(\mu,\BB_{Q_0})=A_\infty^{QC}(\mu,\BB_{Q_0}).$$
\end{lemma}

\begin{proof}
By Lemma \ref{lemDoublingNeighbors}, $\mu$ being $\DD$-doubling in $Q_0$ means
that $\mu(3Q\cap Q_0\setminus R)\leq(1-C_\mu^{-1})\mu(3Q\cap Q_0)$ for every
$Q\in\DD(Q_0)$ and every $R\in\NN(Q)\cap\DD(Q_0)$. Since
$\delta_0>1-C_\mu^{-1}$, condition \rf{eqDefQC} applied to
$E=3Q\cap Q_0\setminus R$ inside $B=3Q\cap Q_0\in\BB_{Q_0}$ gives
$$\nu(3Q\cap Q_0\setminus R)\leq\ve_0\,\nu(3Q\cap Q_0),$$
that is, $\nu(3Q\cap Q_0)\leq(1-\ve_0)^{-1}\nu(R)$. The last assertion follows
by choosing any $\delta_0\in\left(1-C_\mu^{-1},1\right)$ in the definition of
$A_\infty^{QC}$.
\end{proof}

Combining \rf{eqEqualityApDyadic}, Lemma \ref{lemBothDoubling} and
\rf{eqTrivialImplications} we obtain the left-hand counterpart of Theorem
\ref{theoEqualityRH}.

\begin{theorem}\label{theoEqualityA}
If $\mu$ is $\DD$-doubling in $Q_0\in\DD\cup3\DD$, then
\begin{equation}\label{eqBothDoubling}
	A_\infty^{U}(\mu,\BB_{Q_0})=A_\infty^{QC}(\mu,\BB_{Q_0})
	=QC^{d}(\mu,\BB_{Q_0}),
\end{equation}
and we write $A_\infty(\mu,\BB_{Q_0})$ for this class.
\end{theorem}

The two classes are different, and one need not leave the Lebesgue measure to
see it.

\begin{example}\label{exStaircase}
Let $\mu$ be the Lebesgue measure in $\R$, let $\DD$ be the standard dyadic
grid, let $Q_0=[0,1)$ and put
$$I_0\coloneqq \left[0,\tfrac12\right),\quad
I_k\coloneqq \left[1-2^{-k},1-2^{-k-1}\right)\ \ (k\geq1),\quad
w\coloneqq \sum_{k\geq0}2^{-k^{2}}\chi_{I_k} .$$
Then $w>0$ and $\nu=w\mu$ lies in $QC(\mu,\BB_{Q_0})$, hence in
$RH_1^{U}(\mu,\BB_{Q_0})$ by Theorem \ref{theoEqualityRH}, but $\nu$ is not
$\DD$-doubling in $Q_0$, so $\nu\notin A_\infty^{U}(\mu,\BB_{Q_0})$ by Theorem
\ref{theoEqualityA}. Let us check these two facts.

Write $J_m\coloneqq \left[1-2^{-m},1\right)=\bigcup_{j\geq m}I_j$, so that
$J_m\in\DD_m$ has children $I_m$ and $J_{m+1}$, and the ancestors of $I_k$ are
exactly $J_k,J_{k-1},\dots,J_0$. Consequently every $Q\in\DD(Q_0)$ either is
contained in some $I_k$, where $w$ is constant, or equals some $J_m$, where
$$\sup_{J_m}w=2^{-m^{2}}\qquad\text{and}\qquad
\avint_{J_m}w\, d\mu\geq2^{m}\,2^{-m^{2}}\left|I_m\right|=2^{-m^{2}-1} ;$$
in both cases $\sup_Qw\leq2\avint_Qw\, d\mu$. If now
$B=3Q\cap Q_0\in3\DD(Q_0)$ with $Q\in\DD_m$, then $B$ is a union of at most
$N=N(\lambda)$ cubes of $\DD_m$ by Remark \ref{remAdjacentSystems}, and
applying the previous bound to the one on which $w$ is largest gives
$\sup_Bw\leq2N\avint_Bw\, d\mu$. So $w\leq2N\avint_Bw\, d\mu$ on every
$B\in\BB_{Q_0}$, whence
$$\mu(E)<\delta\,\mu(B)\implies\nu(E)\leq2N\delta\,\nu(B)
\qquad\text{for every Borel }E\subset B\in\BB_{Q_0},$$
which is \rf{eqDefQC} with $\ve=2N\delta$, admissible for $\delta<(2N)^{-1}$.

On the other hand $\nu(J_m)=\sum_{j\geq m}2^{-j^{2}-j-1}\leq2^{-m^{2}-m}$,
while $\widehat{J_m}=J_{m-1}$ contains $I_{m-1}$, so
$$\frac{\nu(\widehat{J_m})}{\nu(J_m)}
\geq\frac{2^{-(m-1)^{2}}\left|I_{m-1}\right|}{2^{-m^{2}-m}}
=2^{\,2m-2}\xrightarrow[m\to\infty]{}\infty ,$$
and DN\ref{itemDN2} rules out
$\DD$-doubling. That $w$ lies in no $A_p$ can also be seen directly: since $w$
decays superexponentially along the $I_j$, the average
$\avint_{J_m}w^{-\frac1{p-1}}\, d\mu$ is infinite for every $p>1$.
\end{example}

\begin{figure}[ht]
\centering
\begin{tikzpicture}[
	scale=1.0,
	lbl/.style={font=\scriptsize,inner sep=1.5pt},
	drop/.style={densely dotted,gray}
]
	\draw[->,gray] (-0.35,0) -- (9.55,0);
	\node[lbl,gray,anchor=north] at (9.55,-0.08) {$x$};
	\draw[->,gray] (0,-0.05) -- (0,5.75);
	\node[lbl,gray,anchor=east] at (-0.08,5.7) {$w$};
	\foreach \x/\t in {4.5/{$\frac12$},6.75/{$\frac34$},7.875/{$\frac78$},9/{$1$}}
		{\draw[gray] (\x,0) -- (\x,-0.1);
		 \node[lbl,anchor=north] at (\x,-0.1) {\t};}
	\node[lbl,anchor=north east] at (0,-0.1) {$0$};
	\draw[very thick] (0,5.3) -- (4.5,5.3);
	\draw[very thick] (4.5,5.0) -- (6.75,5.0);
	\draw[very thick] (6.75,4.1) -- (7.875,4.1);
	\draw[very thick] (7.875,2.6) -- (8.4375,2.6);
	\draw[very thick] (8.4375,0.5) -- (8.71875,0.5);
	\draw[drop] (4.5,5.0) -- (4.5,5.3);
	\draw[drop] (6.75,4.1) -- (6.75,5.0);
	\draw[drop] (7.875,2.6) -- (7.875,4.1);
	\draw[drop] (8.4375,0.5) -- (8.4375,2.6);
	\draw[drop] (8.71875,0.1) -- (8.71875,0.5);
	\node[lbl,anchor=south] at (2.25,5.33) {$1$};
	\node[lbl,anchor=south] at (5.625,5.03) {$2^{-1}$};
	\node[lbl,anchor=south] at (7.31,4.13) {$2^{-4}$};
	\node[lbl,anchor=south] at (8.16,2.63) {$2^{-9}$};
	\node[lbl,anchor=south] at (8.58,0.53) {$2^{-16}$};
	\draw[|-|] (0.03,0.3) -- (4.47,0.3);
	\node[lbl,anchor=south] at (2.25,0.33) {$I_0$};
	\draw[|-|] (4.53,0.3) -- (6.72,0.3);
	\node[lbl,anchor=south] at (5.625,0.33) {$I_1$};
	\draw[|-|] (6.78,0.3) -- (7.845,0.3);
	\node[lbl,anchor=south] at (7.31,0.33) {$I_2$};
	\draw[|-|] (4.53,-0.85) -- (8.97,-0.85);
	\node[lbl,anchor=north] at (6.75,-0.88) {$J_1$};
	\draw[|-|] (6.78,-1.5) -- (8.97,-1.5);
	\node[lbl,anchor=north] at (7.875,-1.53) {$J_2$};
	\node[lbl,anchor=west] at (0.45,3.5)
		{$J_m=I_m\cup J_{m+1} ,\qquad \widehat{J_m}=J_{m-1} ,$};
	\node[lbl,anchor=west] at (0.45,2.6)
		{$\dfrac{\nu(\widehat{J_m})}{\nu(J_m)}
		 \geq\dfrac{2^{-(m-1)^{2}}\left|I_{m-1}\right|}{\nu(J_m)}
		 \geq2^{\,2m-2}\xrightarrow[\ m\to\infty\ ]{}\infty .$};
\end{tikzpicture}
\caption{Example \ref{exStaircase}: the weight $w=\sum_k2^{-k^{2}}\chi_{I_k}$,
drawn on a logarithmic vertical scale. On each $I_k$, and on each $J_m$, the supremum of $w$ is comparable
to its average, which is what puts $\nu$ in $QC(\mu,\BB_{Q_0})$ and hence in
$RH_p(\mu,\BB_{Q_0})$ for some $p>1$. Doubling fails for the opposite reason:
passing from $J_m$ to its parent $J_{m-1}$ adds the interval $I_{m-1}$, which
is twice as long and carries $2^{\,2m-1}$ times the density, so the step up is
unbounded. By Theorems \ref{theoEqualityRH} and \ref{theoEqualityA}, this is
exactly a member of $RH(\mu,\BB_{Q_0})\setminus A_\infty(\mu,\BB_{Q_0})$.}
\label{figStaircase}
\end{figure}
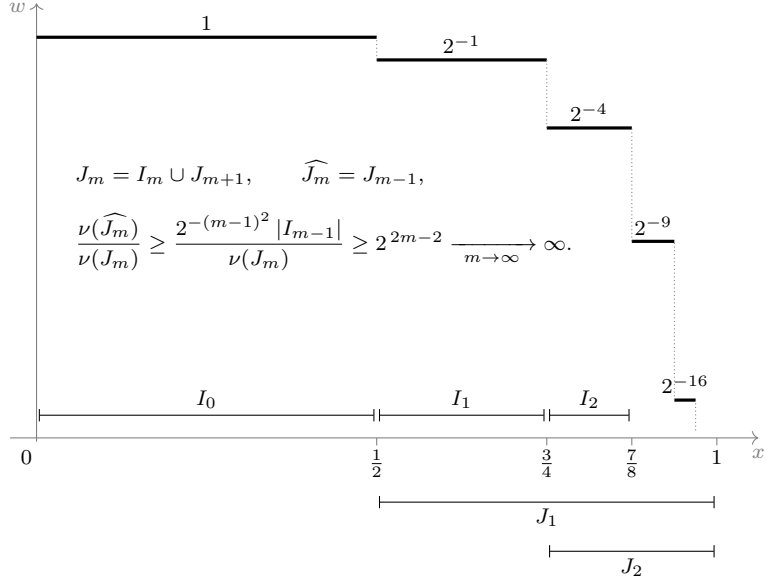

\begin{remark}
The mechanism of the preceding example is the one isolated in Remark \ref{remNoDyadicAnalogue} below:
here \rf{eqDefQC} holds only for small $\delta$, and a small $\delta$ cannot
see the single large step from $J_m$ to its parent.

The construction can be replicated for an arbitrary Radon measure $\mu$, as long as the measure of a cube is not concentrated on a particular child (to be precise, we need $\mu(Q_m)\leq (1-1/m) \mu(\hat Q_m)$). It may fail, for instance, when $\mu$ has point masses.
\end{remark}

\section{From cubes to balls}
\label{secBalls}

The equivalences of \S\ref{secDyadic} are statements about the basis
$\DD\cup3\DD$ inside a fixed $Q_0$. Transferring them to restricted balls
costs us the room between a cube and the balls inscribed in and circumscribed
about it, and this is why the resulting statements hold on the whole of $X$
rather than on a fixed $Q_0$. Throughout this section $\mu$ is doubling, and
for any of the classes $\CC$ of \S\ref{secClasses} we write
$$\CC^{d}(\mu)\coloneqq \{\nu\in\CC(\mu):\, \nu\text{ is doubling}\},$$
which is consistent with the notation of \S\ref{secDyadic} because $\nu$ is
doubling if and only if it is $\DD$-doubling in every $Q_0\in\DD$.

\begin{remark}\label{remDimensionalConstant}
Property D\ref{itemD4} attaches to each $Q\in\DD$ restricted balls
$\Delta_Q\subset Q\subset B_Q\cap X$ with $r(B_Q)=A_0\,r(\Delta_Q)$. A routine
verification from N\ref{eqrem3QNeighbors2} produces the same for every
$Q\in3\DD$ and every $Q\in3\DD(Q_0)$, at the cost of \emph{increasing} $A_0$
to a dimensional constant $\wt A_0$ with
$\wt A_0^{-1}B_Q\cap X\subset\Delta_Q$. From now on $A_0$ denotes the enlarged
constant.
\end{remark}

Fix $Q_0\in\DD\cup3\DD$. On the one hand there is a relatively open set
$V_{Q_0}\supset Q_0$ of comparable diameter containing $B_Q\cap X$ for every
$Q\in\BB_{Q_0}$. On the other hand every restricted ball $\Delta\subset Q_0$
is contained in a cube $\wt Q_\Delta\in3\DD(Q_0)$ with
$r(\wt Q_\Delta)\approx r(\Delta)$, and contains a cube $Q_\Delta\in\DD(Q_0)$
with $\ell(Q_\Delta)\approx r(\Delta)$. Since $\mu$ is doubling, there is a
constant $\widehat C_\mu$ (not to be confused with the constant
$\wt C_\mu$ of \rf{eqEnlargedBallDoubling}) such that
\begin{equation}\label{eqDoublingConstantsModified}
\begin{aligned}
	\widehat C_\mu^{-1}\,\mu(B_Q)  \leq &\, \mu(Q)\leq \widehat C_\mu\,\mu(\Delta_Q)\\
&\text{and}\\
	\widehat C_\mu^{-1}\,\mu(\wt Q_\Delta)\leq &\, \mu(\Delta)\leq \widehat C_\mu\, \mu(Q_\Delta)
\end{aligned}	
\end{equation}
for every $Q\in\BB_{Q_0}$ and every restricted ball $\Delta\subset Q_0$, see Remark
\ref{remDimensionalConstant}.

\begin{lemma}\label{lemApDyadicornot}
Let $\mu$ be a doubling Radon measure in $\R^d$, let $Q_0\in\DD\cup3\DD$ and
let $1<p<\infty$. Then
$$A_p(\mu,V_{Q_0})\subset A_p(\mu,\BB_{Q_0})\subset A_p(\mu,\mathring Q_0),$$
and the same with $RH_p$ in place of $A_p$. In particular
$A_p(\mu)=A_p(\mu,\BB_{\mathcal{D}})$ and $RH_p(\mu)=RH_p(\mu,\BB_{\mathcal{D}})$.
\end{lemma}

\begin{proof}
For the second inclusion, let $\Delta\subset\mathring Q_0$ be a restricted
ball and let $w\in A_p(\mu,\BB_{Q_0})$. Since $\Delta\subset\wt Q_\Delta$ and
$\mu(\wt Q_\Delta)\leq\widehat C_\mu\mu(\Delta)$ by
\rf{eqDoublingConstantsModified}, both averages over $\Delta$ are bounded by
$\widehat C_\mu$ times the corresponding averages over $\wt Q_\Delta$, so
$$\avint_{\Delta}w\, d\mu
\left(\avint_{\Delta}w^{\frac1{1-p}}\, d\mu\right)^{p-1}
\leq\widehat C_\mu^{\,p}\avint_{\wt Q_\Delta}w\, d\mu
\left(\avint_{\wt Q_\Delta}w^{\frac1{1-p}}\, d\mu\right)^{p-1}\leq C .$$
The first inclusion is identical, with $Q\subset B_Q\cap X\subset V_{Q_0}$ in
place of $\Delta\subset\wt Q_\Delta$, and the $RH_p$ statements are obtained
in the same way.
\end{proof}

\begin{lemma}\label{lemDyadicToBallsAinfty}
Let $\mu$ be a doubling Radon measure in $\R^d$, let $Q_0\in\DD\cup3\DD$ and
let $0<\delta,\ve<1$. Then
$$QC_{1-\widehat C_\mu^{-1}\delta,\,1-\ve}(\mu,V_{Q_0})
\subset QC_{1-\delta,\,1-\ve}(\mu,\BB_{Q_0})
\subset QC_{1-\widehat C_\mu\delta,\,1-\ve}(\mu,\mathring Q_0).$$
In particular
$$A_\infty^{QC}(\mu,V_{Q_0})\subset A_\infty^{QC}(\mu,\BB_{Q_0})
\subset A_\infty^{QC}(\mu,\mathring Q_0),$$
and therefore $A_\infty^{QC}(\mu)=A_\infty^{QC}(\mu,\BB_{\mathcal{D}})$.
\end{lemma}

\begin{proof}
By \rf{eqDefQCbis}, membership in $QC_{1-\eta,1-\ve}$ means that
$\mu(E)>\eta\,\mu(B)$ implies $\nu(E)\geq\ve\,\nu(B)$ for every Borel
$E\subset B$ in the relevant basis. Let then
$\nu\in QC_{1-\widehat C_\mu^{-1}\delta,1-\ve}(\mu,V_{Q_0})$ and let
$E\subset Q\in\BB_{Q_0}$ with $\mu(E)>\delta\,\mu(Q)$. By
\rf{eqDoublingConstantsModified} we have
$\mu(Q)\geq\widehat C_\mu^{-1}\mu(B_Q\cap X)$, so
$\mu(E)>\widehat C_\mu^{-1}\delta\,\mu(B_Q\cap X)$, and since
$B_Q\cap X\subset V_{Q_0}$ the hypothesis yields
$\nu(E)\geq\ve\,\nu(B_Q\cap X)\geq\ve\,\nu(Q)$. The second inclusion is the
same argument with $\Delta\subset\wt Q_\Delta$ in place of
$Q\subset B_Q\cap X$.
\end{proof}

\begin{lemma}\label{lemQCDoubling}
Let $\mu$ be a doubling Radon measure in $\R^d$ and let
$Q_0\in\DD\cup3\DD$. Then
$$QC^{d}_{\delta,1-\ve}(\mu,\mathring Q_0)
\subset QC^{d}_{\widehat C_\mu^{-1}\delta,\,1-\widehat C_\nu^{-1}\ve}(\mu,\BB_{Q_0}),$$
and conversely with the two bases interchanged, $\widehat C_\nu$ being the constant
of \rf{eqDoublingConstantsModified} for $\nu$. In particular
$QC^{d}(\mu,\BB_{\mathcal{D}})=QC^{d}(\mu)$.
\end{lemma}

\begin{proof}
Let $\nu\in QC^{d}_{\delta,1-\ve}(\mu,\mathring Q_0)$ and let
$E\subset Q\in\BB_{Q_0}$ be Borel. If $\nu(E\cap\Delta_Q)\leq(1-\ve)
\nu(\Delta_Q)$ then $\nu(\Delta_Q\setminus E)\geq\ve\,\nu(\Delta_Q)$, and the
doubling of $\nu$ gives $\nu(Q\setminus E)\geq\widehat C_\nu^{-1}\ve\,\nu(Q)$;
contrapositively,
$$\nu(E)>\left(1-\widehat C_\nu^{-1}\ve\right)\nu(Q)
\implies\nu(E\cap\Delta_Q)>(1-\ve)\,\nu(\Delta_Q).$$
Applying the hypothesis to the ball $\Delta_Q\subset\mathring Q_0$, in the
form \rf{eqDefQCbis}, and then the doubling of $\mu$, we get
$\mu(E\cap\Delta_Q)\geq\delta\,\mu(\Delta_Q)$ and hence
$\mu(E)\geq\widehat C_\mu^{-1}\delta\,\mu(Q)$, which is the assertion. The converse
is the same with $Q_\Delta\subset\Delta$ in place of $\Delta_Q\subset Q$.
\end{proof}

\begin{proposition}\label{propoDoublingAllEquivalent}
If $\mu$ is a doubling Radon measure in $\R^d$ and $\DD$ is a dyadic grid
associated to $X=\supp\mu$, then
$$A_\infty^{U}(\mu,\BB_{\mathcal{D}})=A_\infty^{QC}(\mu,\BB_{\mathcal{D}})
=QC^{d}(\mu,\BB_{\mathcal{D}})=QC^{d}(\mu)=A_\infty^{QC}(\mu)=A_\infty^{U}(\mu).$$
\end{proposition}

\begin{proof}
Lemmas \ref{lemTrivialImplications} and \ref{lemDyadicToBallsAinfty} give
$$A_\infty^{U}(\mu,V_{Q_0})\subset A_\infty^{QC}(\mu,V_{Q_0})
\subset A_\infty^{QC}(\mu,\BB_{Q_0}),$$
while \rf{eqBothDoubling} and Lemma \ref{lemApDyadicornot} give
$$A_\infty^{QC}(\mu,\BB_{Q_0})\by{\rf{eqBothDoubling}}=
A_\infty^{U}(\mu,\BB_{Q_0})\subset A_\infty^{U}(\mu,\mathring Q_0).$$
More precisely, a weight $\nu\in A_\infty^{QC}(\mu,\BB_{\mathcal{D}})$ lies in
$A_q(\mu,\mathring Q_0)$ with $q$ depending only on the $A_\infty^{QC}$
constants, which are independent of $Q_0$; since every restricted ball is
contained in some $\mathring Q_0$ with $Q_0\in3\DD$ of comparable radius
(Remark \ref{remDimensionalConstant}), this gives $\nu\in A_q(\mu)$. Hence
$$\begin{aligned}
\nu\in A_\infty^{U}(\mu)&\implies\nu\in A_\infty^{QC}(\mu)
\implies\nu\in A_\infty^{QC}(\mu,\BB_{\mathcal{D}})\\
&\implies\nu\in A_\infty^{U}(\mu,\BB_{\mathcal{D}})\implies\nu\in A_\infty^{U}(\mu),
\end{aligned}$$
which gives the first, second, fifth and sixth equalities. The remaining ones are
\rf{eqBothDoubling} and Lemma \ref{lemQCDoubling}.
\end{proof}

Dualizing by \rf{eqDualityWeights}, we obtain the reverse H\"older counterpart,
in which the superscript $^{d}$ is now unavoidable.

\begin{corollary}\label{coroDoublingAllEquivalentDual}
If $\mu$ is a Radon measure in $\R^d$ and $\DD$ is a dyadic grid associated to
$\supp\mu$, then
$$RH_1^{U,d}(\mu,\BB_{\mathcal{D}})=RH_1^{QC,d}(\mu,\BB_{\mathcal{D}})=RH_1^{QC,d}(\mu)
=RH_1^{U,d}(\mu),$$
and these classes are empty unless $\mu$ is doubling. Consequently, if $\mu$
is doubling,
$$A_\infty^{U}(\mu)=A_\infty^{QC}(\mu)=QC^{d}(\mu)=RH_1^{QC,d}(\mu)
=RH_1^{U,d}(\mu).$$
\end{corollary}

\begin{remark}\label{remAllAreTheSame}
Thus, when $\mu$ and $\nu$ are doubling measures with common support, all the
conditions of \S\ref{secClasses} agree, and to verify any of them it suffices
to verify a reverse H\"older inequality, or to bound an $A_p$ constant, or to
exhibit a single pair $(\delta,\ve)$ with $\nu\in QC_{\delta,\ve}(\mu)$, or
the same on the basis $\DD\cup3\DD$. In this generality the statement is due
to Korte and Kansanen \cite[Cor.~4.19, Rem.~4.20]{KorteKansanen}, for complete
doubling metric measure spaces.
\end{remark}

A warning is in order before we turn to the superscript $^{d}$ itself.

\begin{remark}\label{remQCnotQCD}
Proposition \ref{propoDoublingAllEquivalent} does not yield
$QC(\mu)=QC(\mu,\BB_{\mathcal{D}})$. Writing $\delta=1-\delta_0$, the first inclusion
of Lemma \ref{lemDyadicToBallsAinfty} sends
$$QC_{\delta,\ve}(\mu)\quad\text{into}\quad
QC_{1-\widehat C_\mu\delta_0,\,\ve}(\mu,\BB_{Q_0}),$$
and the new parameter is
admissible only when $\delta_0<\widehat C_\mu^{-1}$; a measure that belongs to
$QC_{\delta,\ve}(\mu)$ only for small $\delta$ therefore need not belong to
any $QC_{\delta',\ve'}(\mu,\BB_{Q_0})$. The class $A_\infty^{QC}$ is
unaffected, since there $\delta$ may be taken as close to $1$ as one pleases;
$RH_1^{QC}$ is not, and this is why Corollary
\ref{coroDoublingAllEquivalentDual} has to be obtained by duality rather than
directly. Whether $QC(\mu)=QC(\mu,\BB_{\mathcal{D}})$ we do not know; see
\S\ref{secOpen}.
\end{remark}

\subsection{Thin annuli}

The superscript $^{d}$ in Corollary \ref{coroDoublingAllEquivalentDual} is the
price of working with balls rather than cubes, and it can be traced to a
single property of $\mu$. Say that $\mu$ has \emph{thin annuli} if
\begin{equation}\label{eqThinAnnuli}
	\forall\theta>0\ \exists\eta>0:\quad
	\mu\left(\Delta_{(1+\eta)r}(x)\setminus\Delta_{r}(x)\right)
	\leq\theta\,\mu\left(\Delta_{(1+\eta)r}(x)\right)
\end{equation}
for every $x\in X$ and every $r>0$. The Lebesgue measure in $\R^d$ has thin
annuli, the quotient in \rf{eqThinAnnuli} being $1-(1+\eta)^{-d}$; a doubling
measure on a closed subset need not, as Example \ref{exLadder} shows.
Condition \rf{eqThinAnnuli} is implied by the annular decay property used by
Maasalo \cite[\S3]{Maasalo}.

\begin{proposition}\label{propoThinAnnuli}
Let $\mu$ be a doubling Radon measure in $\R^d$ with thin annuli and let
$\nu\in QC_{\delta,\ve}(\mu)$ for \emph{some} pair $0<\delta,\ve<1$. Then
$\nu$ is doubling, with constant $(1-\ve)^{-N}$, where
$N=\lceil\log2/\log(1+\eta)\rceil$ and $\eta$ is the parameter associated with
$\theta=\delta/2$ in \rf{eqThinAnnuli}.
\end{proposition}

\begin{proof}
Put $\Delta'=\Delta_r(x)$ and $\Delta=\Delta_{(1+\eta)r}(x)$. The set
$E=\Delta\setminus\Delta'$ satisfies
$\mu(E)\leq\frac\delta2\mu(\Delta)<\delta\,\mu(\Delta)$, so \rf{eqDefQC} gives
$\nu(E)\leq\ve\,\nu(\Delta)$ and hence $\nu(\Delta')\geq(1-\ve)\nu(\Delta)$.
Iterating $N$ times yields
$\nu(\Delta_r(x))\geq(1-\ve)^{N}\nu(\Delta_{2r}(x))$.
\end{proof}

Since $QC(\mu)$ is the largest of the classes of \S\ref{secClasses}, the whole
of row (a) of Figure \ref{figImplications} collapses.

\begin{corollary}\label{coroThinAnnuli}
If $\mu$ is a doubling Radon measure in $\R^d$ with thin annuli, then
$$A_\infty^{U}(\mu)=A_\infty^{QC}(\mu)=QC(\mu)=RH_1^{QC}(\mu)=RH_1^{U}(\mu),$$
and the superscript $^{d}$ may be omitted throughout. Combined with
\cite[Thm.~5.2]{AndersonWeirich}, such a $\mu$ also satisfies the strong
Gehring lemma.
\end{corollary}

\begin{proof}
By Proposition \ref{propoThinAnnuli} every $\nu\in QC(\mu)$ is doubling, so
$QC(\mu)=QC^{d}(\mu)$, which Proposition \ref{propoDoublingAllEquivalent}
identifies with $A_\infty^{U}(\mu)$ and $A_\infty^{QC}(\mu)$. In particular
$RH_1^{QC}(\mu)\subset QC(\mu)$ consists of doubling measures, so Corollary
\ref{coroDoublingAllEquivalentDual} gives
$$RH_1^{QC}(\mu)=RH_1^{QC,d}(\mu)=RH_1^{U,d}(\mu)=RH_1^{U}(\mu).\qedformula$$\qedOut
\end{proof}

\begin{remark}\label{remNoDyadicAnalogue}
Proposition \ref{propoThinAnnuli} has no counterpart on the basis
$\DD\cup3\DD$, and the reason is not that dyadic cubes have bad boundaries.
Its proof goes from $\Delta_r(x)$ to $\Delta_{2r}(x)$ in $N$ steps, each of
which costs a factor $1-\ve$ because the annulus it adds carries little
$\mu$-mass; this needs a continuum of radii, and over $\DD\cup3\DD$ there is
nothing between $Q$ and $\widehat Q$, nor between $Q$ and $3Q$.

What the dyadic setting has instead is Lemma \ref{lemBothDoubling}, which
reaches the same conclusion in a single step and pays for it with a threshold:
the set $3Q\cap Q_0\setminus R$ removed in that step is a fraction of
$3Q\cap Q_0$ that is only known to be at most $1-C_\mu^{-1}$, so
\rf{eqDefQC} is of use there only when $\delta_0>1-C_\mu^{-1}$. This is why
the two halves of \rf{eqTrivialImplications} behave differently:
$A_\infty^{QC}$ quantifies over every $\delta$ and therefore meets the
threshold, so it forces $\nu$ to be $\DD$-doubling, whereas $RH_1^{QC}$ is
given its $\delta$ and may well miss it. Example \ref{exStaircase} shows that
it does: there $\nu$ satisfies \rf{eqDefQC} for every small $\delta$ and is
not $\DD$-doubling. 

Restoring the small steps would mean replacing $3\DD$ by a family of
enlargements $(1+\ve)^j Q$ assembled from triples of descendants of $Q$ of a deep
generation inside $3Q$. Such a family has worse nesting properties because intermediate cubes are no longer made of unions of cubes of the
generation of $Q$, and both properties are used throughout; see Remark
\ref{remAdjacentSystems}. We have not studied it.
\end{remark}

\begin{example}\label{exLadder}
Without thin annuli the conclusion of Corollary \ref{coroThinAnnuli} fails,
already on a closed subset of $\R$ and with restricted Euclidean balls. Put
$$X:=\{0\}\cup\{2^{k}:k\in\Z\},\qquad
\mu:=\sum_{k\in\Z}2^{k}\delta_{2^{k}},$$
$$b_k:=2^{k|k|},\qquad w(2^{k}):=b_k2^{-k},$$
so that $\nu=w\mu=\sum_{k\in\Z}b_k\delta_{2^{k}}$. Then $X$ is closed, $\mu$
is a doubling Radon measure with $\supp\mu=X$, and $w>0$ away from the origin, while
$$w\in RH_\infty(\mu)\subset RH_1^{U}(\mu)
\qquad\text{and}\qquad \nu\text{ is not doubling} .$$
In particular $w\notin A_\infty^{U}(\mu)$, so $RH_1^{U}(\mu)$ is strictly
larger than $RH_1^{U,d}(\mu)$ and the superscript cannot be dropped. We leave the computations to the reader.
\end{example}

\begin{figure}[ht]
\centering
\begin{tikzpicture}[
	scale=1.12,
	lbl/.style={font=\scriptsize,inner sep=1.5pt},
	cut/.style={densely dashed,gray}
]
	\draw[->,gray] (-0.45,0) -- (9.1,0);
	\node[lbl,gray,anchor=north] at (9.1,-0.08) {$\R$};
	\fill (0,0) circle (0.5pt);
	\node[lbl,anchor=north] at (0,-0.1) {$0$};
	\foreach \k/\x/\rr in {-3/0.125/0.019,-2/0.25/0.027,-1/0.5/0.039,
		0/1/0.055,1/2/0.078,2/4/0.110,3/8/0.155}
		{\fill (\x,0) circle (\rr);}
	\node[lbl,anchor=north] at (1,-0.12) {$1$};
	\node[lbl,anchor=north] at (2,-0.14) {$2$};
	\node[lbl,anchor=north] at (4,-0.16) {$4$};
	\node[lbl,anchor=north] at (8,-0.2) {$8$};
	\draw[cut] (3.75,-0.55) -- (3.75,0.95);
	\node[lbl,gray,anchor=south east] at (3.72,0.97) {$r$};
	\draw[cut] (4.3,-0.55) -- (4.3,0.95);
	\node[lbl,gray,anchor=south west] at (4.33,0.97) {$(1+\eta)r$};
	\draw[|-|] (0,0.55) -- (3.72,0.55);
	\node[lbl,anchor=south] at (1.9,0.58)
		{$\mu\bigl(\Delta_r(0)\bigr)=4$};
	\draw[|-|] (0,-0.5) -- (4.28,-0.5);
	\node[lbl,anchor=north] at (2.1,-0.53)
		{$\mu\bigl(\Delta_{(1+\eta)r}(0)\bigr)=8$};
	\node[lbl,anchor=south west] at (4.15,0.14)
		{$\mu(\{4\})=4$, \ $\nu(\{4\})=b_2$};
\end{tikzpicture}
\caption{Example \ref{exLadder}: the atoms of $\mu$, drawn with radius
proportional to $\sqrt{\mu(\{2^{k}\})}$. Taking $r$ slightly below $2^{k}$,
the single atom $2^{k}$ lying in the annulus carries as much $\mu$-mass as the
whole of $\Delta_r(0)$, whatever $\eta$ is: the annuli are not thin, and the
quotient in \rf{eqThinAnnuli} equals $\frac12$. The weight is chosen so that
the same atom carries incomparably more $\nu$-mass than everything inside, so
that $\nu$ is not doubling, while it remains increasing and therefore stays in
$RH_\infty(\mu)$.}
\label{figLadder}
\end{figure}
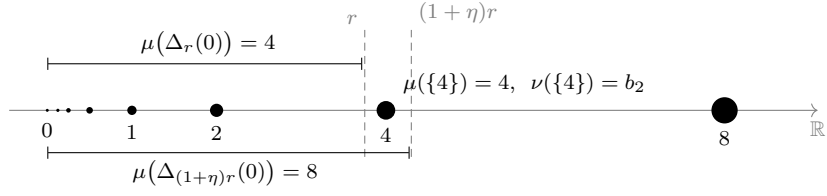

\begin{remark}\label{remLadderNonAtomic}
The atoms are a convenience, not the point. Replacing each $2^{k}$ by a closed
interval $P_k$ centred there, of length $\delta_k$ much smaller than the gap
$2^{k-1}$ and carrying mass $2^{k}$ distributed uniformly, produces a
non-atomic example with the same properties: a ball meeting $P_k$ and reaching
anything else has radius larger than the gap and therefore contains $P_k$
entirely, while a ball meeting only $P_k$ sees a constant weight. Such a $\mu$
is of course not Ahlfors regular, and cannot be: the gaps between consecutive
scales are what makes the annuli fat.
\end{remark}

\section{The weak classes}
\label{secWeak}

We now introduce the three conditions in which every comparison between $\nu$
and $\mu$ is allowed to involve a dilated ball, and record the two implications
between them that are immediate. The third one, which closes the circle, is
the content of \S\ref{secBL}.

\begin{definition}\label{defWeakClasses}
Let $\mu$ be a Radon measure in $\R^d$, let $X=\supp\mu$, let $V\subset X$ be
relatively open and let $\nu$ be a Radon measure supported in $X$.
\begin{itemize}
	\item Given $1<p<\infty$, we say that $\nu\in wRH_p(\mu,V)$ if $\nu\ll\mu$,
	the density $w=\frac{d\nu}{d\mu}$ is positive $\mu$-a.e., and
	\begin{equation}\label{eqWeakRH}
		\left(\avint_{\Delta}w^p\, d\mu\right)^{\frac1p}
		\leq\frac{C_{RH}}{\mu(\Delta)}\int_{2\Delta}w\, d\mu
		=C_{RH}\,\frac{\nu(2\Delta)}{\mu(\Delta)}
		\qquad\text{whenever }2\Delta\subset V .
	\end{equation}
	\item Given $0<\delta_0,\ve_0<1$, we say that
	$\nu\in wQC_{\delta_0,\ve_0}(\mu,V)$ if
	\begin{equation}\label{eqWeakQC}
		\mu(E\cap\Delta)<\delta_0\,\mu(\Delta)
		\implies\nu(E\cap\Delta)\leq\ve_0\,\nu(2\Delta)
	\end{equation}
	for every Borel set $E$ and every restricted ball $\Delta$ with $2\Delta\subset V$.
	\item We say that $\nu\in wBL_{\ve_0}(\mu,V)$ if \rf{eqBLintro} holds for
	some $0<\delta_0<1$ and $C_0\geq1$, as in Definition \ref{defBLintro}.
\end{itemize}
Finally we put
$$\begin{aligned}
wRH_1^{U}(\mu,V)&\coloneqq \bigcup_{1<p<\infty}wRH_p(\mu,V),\\
wRH_1^{BL}(\mu,V)&\coloneqq \bigcap_{0<\ve<1}wBL_{\ve}(\mu,V),\\
wRH_1^{QC}(\mu,V)&\coloneqq \bigcap_{0<\ve<1}\bigcup_{0<\delta<1}
wQC_{\delta,\ve}(\mu,V) .
\end{aligned}$$
If one of these conditions holds for every relatively open $V\subset X$ we
omit $V$ from the notation.
\end{definition}

The quantifier pattern defining $wRH_1^{QC}$ is the weak counterpart of the
one defining $RH_1^{QC}$ in \S\ref{secClasses}: for every $\ve$ there is a
$\delta$, and not the other way round. The class obtained by exchanging the
two quantifiers is the weak analogue of $A_\infty^{QC}$, and it plays no role
below; see \S\ref{secOpen}.

Before going on it is worth recording that the dilation is not a device forced
on us by the roughness of $X$: it enlarges the classes already in the
classical setting.

\begin{remark}\label{remExpExample}
Let $\mu=\LL^1$ on $\R$ and $w(x)=e^{x}$. For an interval $I$ of length $\ell$
one has $\left(\avint_{I}w^p\right)^{1/p}\big/\avint_{I}w\approx
p^{-1/p}\ell^{1/p'}$ as $\ell\to\infty$, so $w\notin RH_p(\LL^1)$ for any
$p>1$; equivalently, $\nu=w\,\LL^1$ is not doubling, so
$w\notin A_\infty^{U}(\LL^1)$ either. On the other hand the average of $w$
over $2I$ exceeds that over $I$ by a factor $e^{\ell/2}/2$, which absorbs the
loss, and $w\in wRH_p(\LL^1)$ for every $p<\infty$. Thus
$RH_1^{U}\subsetneq wRH_1^{U}$ already for the Lebesgue measure. This is
\cite[Ex.~2.4]{KinnunenKurkiMudarra}; see also
\cite[Ex.~3.2]{AndersonHytonenTapiola}. The same is true over the basis
$\DD\cup3\DD$, and there the dilation is redundant precisely when $\nu$ is
doubling; see Remark \ref{remWeakDyadicDoubling} and Example
\ref{exWeakNotStrong}.
\end{remark}

\begin{lemma}\label{lemWeakTrivial1}
Let $1<p<\infty$ and let $\nu\in wRH_p(\mu,V)$ with constant $C_{RH}$. Then
$\nu\in wQC_{\delta,\ve}(\mu,V)$ for every $0<\ve<1$, with
$\delta=\left(C_{RH}^{-1}\ve\right)^{p'}$. In particular
$$wRH_1^{U}(\mu,V)\subset wRH_1^{QC}(\mu,V).$$
\end{lemma}

\begin{proof}
Let $E\subset\Delta$ be Borel with $2\Delta\subset V$. By H\"older's
inequality and \rf{eqWeakRH},
$$\nu(E)=\int_E w\, d\mu
\leq\left(\avint_{\Delta}w^{p}\, d\mu\right)^{\frac1p}
\mu(\Delta)^{\frac1p}\mu(E)^{1-\frac1p}
\leq C_{RH}\,\nu(2\Delta)
\left(\frac{\mu(E)}{\mu(\Delta)}\right)^{\frac1{p'}} ,$$
so $\mu(E)<\delta\mu(\Delta)$ with $\delta=(C_{RH}^{-1}\ve)^{p'}$ forces
$\nu(E)\leq\ve\,\nu(2\Delta)$.
\end{proof}

\begin{lemma}\label{lemWeakTrivial2}
Let $\nu\in wQC_{\delta_0,\ve_0}(\mu,V)$. Then $\nu$ satisfies \rf{eqBLintro}
with the same $\ve_0$, with $\delta_0$ replaced by
$\delta_0\frac{\mu(\Delta)}{\mu(2\Delta)}$ and with $C_0=1+\ve_0$. In
particular, if $\mu$ is doubling then
$$wQC_{\delta,\ve}(\mu,V)\subset wBL_{\ve}(\mu,V)
\qquad\text{and}\qquad
wRH_1^{QC}(\mu,V)\subset wRH_1^{BL}(\mu,V).$$
\end{lemma}

\begin{proof}
Let $E\subset2\Delta\subset V$ be Borel with
$\mu(2\Delta\cap E)<\delta_0\frac{\mu(\Delta)}{\mu(2\Delta)}\mu(2\Delta)$.
Then $\mu(\Delta\cap E)\leq\mu(2\Delta\cap E)<\delta_0\mu(\Delta)$, so
\rf{eqWeakQC} gives $\nu(E\cap\Delta)\leq\ve_0\nu(2\Delta)$ and hence
$$\nu(\Delta)=\nu(\Delta\cap E)+\nu(\Delta\setminus E)
\leq\ve_0\nu(2\Delta\cap E)+\left(1+\ve_0\right)\nu(2\Delta\setminus E).$$
If $\mu$ is doubling then $\mu(\Delta)/\mu(2\Delta)\geq C_\mu^{-1}$, so the
admissible $\delta_0$ may be taken independent of $\Delta$.
\end{proof}

The remaining inclusion, $wRH_1^{BL}\subset wRH_1^{U}$, is Theorem
\ref{theoAintro}, and with it Corollary \ref{coroBintro} follows: the three
classes coincide. We stress that the equality
$wRH_1^{U}(\mu,V)=wRH_1^{QC}(\mu,V)$ is, in this generality, a theorem of
Kinnunen, Kurki and Mudarra \cite[Thm.~4.1]{KinnunenKurkiMudarra}, proved for
a complete doubling metric measure space and a weight defined only on an open
subset; what \S\ref{secBL} adds is the third class.

\section{The Bennewitz--Lewis theorem}
\label{secBL}

This section proves Theorem \ref{theoAintro}, and with it the inclusion
$wRH_1^{BL}\subset wRH_1^{U}$ that closes Corollary \ref{coroBintro}. The
proof has three parts: the passage from restricted balls to the basis
$\DD\cup3\DD$, the absolute continuity of $\nu$ with respect to $\mu$, and the
estimate itself.

\subsection{The dyadic form of the condition}

\begin{definition}\label{defWeakDyadic}
Let $\mu$ be a Radon measure in $\R^d$, let $\DD$ be a dyadic grid associated
to $X=\supp\mu$, let $V\subset X$ be relatively open and put
$$\QQ_V\coloneqq \{Q\in\DD:\, 3Q\subset V\} .$$
Given a Radon measure $\nu$ supported in $X$, we say that
$\nu\in wBL_{\ve_0}(\mu,\QQ_V)$ if there are $0<\delta_0<1$ and $C_0\geq1$
such that
\begin{equation}\label{eqWeakBLDyadic}
	\mu(3Q\cap E)<\delta_0\,\mu(3Q)
	\implies
	\nu(Q)\leq\ve_0\,\nu(3Q\cap E)+C_0\,\nu(3Q\setminus E)
\end{equation}
for every $Q\in\QQ_V$ and every Borel set $E\subset3Q$; and we say that
$\nu\in wRH_p(\mu,\QQ_V)$ if $\nu\ll\mu$, $w=\frac{d\nu}{d\mu}>0$ $\mu$-a.e.\
and
$$\left(\avint_{Q}w^p\, d\mu\right)^{\frac1p}
\leq C_{RH}\,\frac{\nu(3Q)}{\mu(Q)}\qquad\text{for every }Q\in\QQ_V .$$
\end{definition}

\begin{lemma}\label{lemWeakBLDyadic}
Let $\mu$ be a doubling Radon measure in $\R^d$, let $V\subset X$ be
relatively open and let $\nu$ be a Radon measure supported in $X$. There is a
constant $C_\DD$ depending only on $d$, $C_\mu$ and the parameters of $\DD$ such
that
$$wBL_{\ve_0}(\mu,V)\subset wBL_{C_\DD\ve_0}(\mu,\QQ_V)
\qquad\text{whenever }\ve_0<C_\DD^{-1} .$$
In particular $wRH_1^{BL}(\mu,V)\subset wRH_1^{BL}(\mu,\QQ_V)$.
\end{lemma}

\begin{proof}
Let $Q\in\QQ_V$ and let $E\subset3Q$ be Borel. Choose a maximal
$\frac12\ell(Q)$-separated set of points of $Q$ and let $\{\Delta_j\}_j$ be
the restricted balls of radius $\ell(Q)$ centred at them. They cover $Q$; and
since $\diam Q\leq2r_Q=A_0\ell(Q)$, their number and the overlap of their
dilates are bounded by a dimensional constant,
$$Q\subset\bigcup_j\Delta_j,\qquad \sum_j\chi_{2\Delta_j}\leq C_\DD .$$
Each $2\Delta_j$ is a restricted ball of radius $2\ell(Q)=\frac4{A_0}r_Q$
centred at a point of $Q\subset B_Q$, so
$2\Delta_j\subset\left(\frac4{A_0}+1\right)B_Q\cap X\subset3Q\subset V$ by
N\ref{eqrem3QNeighbors2}. Also $3Q\subset3\lambda B_Q$ by
\rf{eqTrivialInclusionDQP}, and $2\Delta_j$ contains a restricted ball of
radius $\ell(Q)$ centred in $Q$, so iterating the doubling condition a
bounded number $n_d$ of times gives $\mu(3Q)\leq C_\mu^{n_d}\mu(2\Delta_j)$.

Let now $0<\delta_0<1$ and $C_0\geq1$ be the constants of \rf{eqBLintro} for
$\nu$, and set $\wt\delta_0\coloneqq C_\mu^{-n_d}\delta_0$. If
$\mu(3Q\cap E)<\wt\delta_0\mu(3Q)$, then for every $j$
$$\mu(2\Delta_j\cap E)\leq\mu(3Q\cap E)<\wt\delta_0\,\mu(3Q)
\leq\delta_0\,\mu(2\Delta_j),$$
so \rf{eqBLintro} applies to $\Delta_j$ and to $E\cap2\Delta_j$, and summing
in $j$,
$$\begin{aligned}
\nu(Q)&\leq\sum_j\nu(\Delta_j)
\leq\ve_0\sum_j\nu(2\Delta_j\cap E)+C_0\sum_j\nu(2\Delta_j\setminus E)\\
&\leq C_\DD\ve_0\,\nu(3Q\cap E)+C_\DD C_0\,\nu(3Q\setminus E).\qedformula
\end{aligned}$$\qedOut
\end{proof}

Before going on we record that the dilation is not redundant in Definition
\ref{defWeakDyadic} either, and that the obstruction is again the one isolated
in Remark \ref{remNoDyadicAnalogue}: it is the doubling of $\nu$, and nothing
else, that separates the weak dyadic condition from the strong one.

\begin{remark}\label{remWeakDyadicDoubling}
Let $\mu$ be $\DD$-doubling in $Q_0$ with constant $C_\mu$, let $\nu=w\mu$ be
$\DD$-doubling in $Q_0$ with constant $C_\nu$ and suppose that
$$\left(\avint_{Q}w^p\, d\mu\right)^{\frac1p}
\leq C_{RH}\,\frac{\nu(3Q\cap Q_0)}{\mu(Q)}\qquad
\text{for every }Q\in\DD(Q_0).$$
Then $\nu\in RH_p(\mu,\BB_{Q_0})$ with constant
$C_\mu^{1/p'}C_\nu C_{RH}$. Indeed $Q\in\NN(Q)$, so
$\nu(3Q\cap Q_0)\leq C_\nu\nu(Q)$ and the hypothesis becomes the strong
inequality on $\DD(Q_0)$; and if $B=3Q\cap Q_0\in3\DD(Q_0)$, one can also see that the strong reverse H\"older inequality holds by standard arguments.
\end{remark}

\begin{example}\label{exWeakNotStrong}
Without the doubling of $\nu$ the two conditions genuinely differ, and again
the Lebesgue measure suffices. Fix $1<p<\infty$ and $0<s<\frac1p$, and for
$n\geq1$ let
$$g_n(x)\coloneqq \begin{cases}
2^{n(1-s)}|x|^{-s}, & -1\leq x<0,\\
2^{n}, & 0\leq x<2^{-n},\\
1, & 2^{-n}\leq x<1 .
\end{cases}$$
Since $sp<1$, both $\avint_{[-1,1)}g_n\, d\mu$ and
$\big(\avint_{[-1,1)}g_n^p\, d\mu\big)^{1/p}$ are comparable to
$2^{n(1-s)}$, with constants depending only on $p$ and $s$. Write
$S_i\coloneqq [0,2^{-n+i})$ and $R_i\coloneqq [-2^{-n+i},0)$ for
$0\leq i\leq n$; these are adjacent dyadic cubes of the same generation, so
$R_i\subset3S_i$. A direct computation gives
$$1\leq\int_{S_i}g_n\, d\mu\leq2,\qquad
\int_{S_i}g_n^q\, d\mu\geq2^{n(q-1)},\qquad
\int_{R_i}g_n\, d\mu=\frac{2^{\,i(1-s)}}{1-s} ,$$
whence, for every $q>1$,
$$\frac{\big(\avint_{S_i}g_n^q\, d\mu\big)^{\frac1q}}
{\avint_{S_i}g_n\, d\mu}\geq\frac{2^{\,i(1-1/q)}}2,
\qquad
\frac{\big(\avint_{S_i}g_n^p\, d\mu\big)^{\frac1p}}
{\mu(S_i)^{-1}\int_{3S_i}g_n\, d\mu}
\leq(1-s)\,2^{\,\frac1p-\frac ip+si}\leq2 ,$$
the last step holding because $s<\frac1p$. Hence the weak inequality holds on every
$S_i$ with a constant independent of $n$, while the strong one fails on $S_n$
by a factor $2^{\,n(1-1/q)-1}$.

To assemble a single weight, let $\mu$ be the Lebesgue measure on $X=[0,1]$,
let $\DD$ be the standard dyadic grid, put $P_k\coloneqq [2^{-k-1},2^{-k})$
and let $\Phi_k:[-1,1)\to P_k$ be the increasing affine bijection, which maps
the dyadic grid onto $\DD(P_k)$. Define
$$w\coloneqq \sum_{k\geq1}
\frac{g_k\circ\Phi_k^{-1}}{\avint_{[-1,1)}g_k\, d\mu}\,\chi_{P_k},
\qquad\text{so that }\avint_{P_k}w\, d\mu=1\text{ for every }k .$$
Every cube of $\DD([0,1))$ is contained in some $P_k$, or is one of the
$T_m\coloneqq [0,2^{-m})=\bigcup_{k\geq m}P_k$. On the first kind the estimates
above apply verbatim; on the second, $\avint_{T_m}w\, d\mu=1$ and
$$\avint_{T_m}w^p\, d\mu=2^{m}\sum_{k\geq m}2^{-k-1}\avint_{P_k}w^p\, d\mu
\leq C(p,s),$$
because $\avint_{P_k}w^p\, d\mu=\avint_{[-1,1)}g_k^p\, d\mu\big/
\big(\avint_{[-1,1)}g_k\, d\mu\big)^{p}$ is bounded independently of $k$.
Hence $w\in wRH_p(\mu,\QQ_X)$, while for every $q>1$ the strong inequality
fails inside $P_k$ by a factor $2^{\,k(1-1/q)-1}\to\infty$, so
$w\notin RH_q(\mu,\BB_{Q_0})$ for any $q>1$, with $Q_0=[0,1)$. By Remark
\ref{remWeakDyadicDoubling} the measure $\nu=w\mu$ is not $\DD$-doubling, as
one also sees directly: $\nu(3S_i)\gtrsim2^{\,i(1-s)}\nu(S_i)$ inside each
$P_k$.
\end{example}

\begin{figure}[ht]
\centering
\begin{tikzpicture}[
	scale=0.95, transform shape,
	lbl/.style={font=\scriptsize,inner sep=1.5pt},
	aux/.style={densely dotted,gray}
]
	\node[lbl,anchor=west,font=\footnotesize\itshape] at (-5.1,3.25)
		{(a) the gadget $g_n$ on $[-1,1)$, drawn with $n=3$ and
		 $s=\frac12$, logarithmic vertical scale};
	\draw[->,gray] (-5.0,0) -- (5.1,0);
	\node[lbl,gray,anchor=north] at (5.1,-0.06) {$x$};
	\foreach \x in {-4.5,-2.25,0,0.5625,2.25,4.5}
		{\draw[gray] (\x,0) -- (\x,-0.09);}
	\node[lbl,anchor=north] at (-4.5,-0.09) {$-1$};
	\node[lbl,anchor=north] at (-2.25,-0.09) {$-\frac12$};
	\node[lbl,anchor=north east] at (0.02,-0.09) {$0$};
	\node[lbl,anchor=north west] at (0.60,-0.09) {$2^{-n}$};
	\node[lbl,anchor=north] at (2.25,-0.09) {$\frac12$};
	\node[lbl,anchor=north] at (4.5,-0.09) {$1$};
	\draw[very thick,domain=-1:-0.021,samples=160,smooth]
		plot ({4.5*\x},{0.5*(1.5+0.5*ln(-1/\x)/ln(2))+0.4});
	\draw[aux,->] (-0.0945,2.56) -- (-0.0945,2.95);
	\draw[very thick] (0,1.9) -- (0.5625,1.9);
	\fill (0,1.9) circle (1.1pt);
	\draw[aux] (0.5625,0.4) -- (0.5625,1.9);
	\draw[aux] (0,0.4) -- (0.5625,0.4);
	\draw[very thick] (0.5625,0.4) -- (4.5,0.4);
	\node[lbl,anchor=west] at (0.64,1.92) {$2^{n}$};
	\node[lbl,anchor=south] at (2.6,0.45) {$1$};
	\node[lbl,anchor=south east] at (-2.1,1.35) {$2^{n(1-s)}|x|^{-s}$};
	\draw[|-|] (-2.22,-0.62) -- (-0.03,-0.62);
	\node[lbl,anchor=north] at (-1.12,-0.65) {$R_{n-1}$};
	\draw[|-|] (0.03,-0.62) -- (2.22,-0.62);
	\node[lbl,anchor=north] at (1.12,-0.65) {$S_{n-1}$};
	\draw[|-|] (-2.22,-1.25) -- (4.47,-1.25);
	\node[lbl,anchor=north] at (1.12,-1.28) {$3S_{n-1}$};

	\node[lbl,anchor=west,font=\footnotesize\itshape] at (-5.1,-1.95)
		{(b) one rescaled copy of $g_k$ on each $P_k=[2^{-k-1},2^{-k})$};
	\draw[->,gray] (-5.0,-3.3) -- (5.1,-3.3);
	\node[lbl,gray,anchor=north] at (5.1,-3.36) {$x$};
	\foreach \x in {-4.5,-3.9375,-3.375,-2.25,0,4.5}
		{\draw[gray] (\x,-3.3) -- (\x,-3.39);}
	\node[lbl,anchor=north] at (-4.5,-3.39) {$0$};
	\node[lbl,anchor=north] at (0,-3.39) {$\frac12$};
	\node[lbl,anchor=north] at (4.5,-3.39) {$1$};
	\foreach \a/\m/\b/\h/\l in {0/2.25/4.5/0.45/0.30,
		-2.25/-1.125/0/0.60/0.15, -3.375/-2.8125/-2.25/0.75/0.08}
		{\draw[very thick] (\a,-3.22) -- (\m-0.75*\m+0.75*\a,-3.17)
			.. controls (\m-0.18,-3.15) and (\m-0.05,{-3.3+\h-0.15}) ..
			(\m-0.02,{-3.3+\h});
		 \draw[very thick] (\m,{-3.3+0.72*\h}) -- (\m+\l,{-3.3+0.72*\h});
		 \draw[very thick] (\m+\l,-3.22) -- (\b,-3.22);
		 \draw[aux] (\m+\l,-3.22) -- (\m+\l,{-3.3+0.72*\h});}
	\node[lbl,anchor=north] at (2.25,-3.65) {$P_1$};
	\node[lbl,anchor=north] at (-1.125,-3.65) {$P_2$};
	\node[lbl,anchor=north] at (-2.81,-3.65) {$P_3$};
	\node[lbl] at (-4.15,-3.0) {$\cdots$};
\end{tikzpicture}
\caption{Example \ref{exWeakNotStrong}. In (a) the whole of $g_n$ sits above
the value $1$ of its right-hand tail; the plateau $2^{n}$ occupies the
fraction $2^{-i}$ of $S_i=[0,2^{-n+i})$, which is what makes the strong
inequality on $S_i$ fail by $2^{\,i(1-1/q)}$, while the mass that the
singularity puts on the neighbour $R_i\subset3S_i$ is large enough to save the
weak one. The two cubes are adjacent and of the same generation, so nothing
but the dilation separates the two inequalities. In (b) a copy of $g_k$,
normalized to have average $1$, is placed on $P_k$ by an affine map, so that
its singularity sits at the midpoint of $P_k$; the copies become taller and
narrower as $k$ grows, and the strong inequality fails on every scale at
once.}
\label{figGadget}
\end{figure}
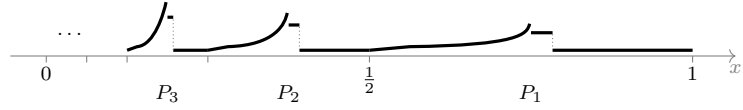

\subsection{Absolute continuity}

The hypothesis of Theorem \ref{theoAintro} does not presuppose $\nu\ll\mu$;
this has to be proved, and it is here that the ambient Euclidean structure is
used for the first time. The argument needs a Lebesgue differentiation
theorem for $\nu$, which is \emph{not} assumed doubling. In a general space of
homogeneous type this would be unavailable; in $\R^d$ it is a standard
consequence of the Besicovitch covering theorem, valid for every Radon
measure; see \cite{Mattila}, for instance.

\begin{lemma}\label{lemBLAC}
Let $\mu$ be a doubling Radon measure in $\R^d$, let $V\subset X$ be
relatively open and let $\nu\in wBL_{\ve_0}(\mu,\QQ_V)$ with
$\ve_0\leq\ve_1$, where $\ve_1>0$ depends only on $d$ and $C_\mu$. Then
$\nu|_V\ll\mu|_V$.
\end{lemma}

\begin{proof}
Put $\kappa\coloneqq 2\lambda A_0\rho_0$ and let $N$ be such that
$\mu(\kappa B)\leq C_\mu^{N}\mu(B)$ for every ball $B$ centred on $X$; both
are dimensional. Fix $C_\nu>2C_\mu^{N}$ and let
$$G\coloneqq \left\{x\in\supp\nu:\
\liminf_{r\to0}\frac{\nu(\kappa B_r(x))}{\nu(B_r(x))}\geq C_\nu\right\} .$$

We first check that $\nu(G)=0$. If $x\in G$ there is $r_x>0$ with
$\nu(\kappa B_r(x))\geq\frac{C_\nu}2\nu(B_r(x))$ for every $0<r<r_x$, whence,
writing $B=B_r(x)$ and iterating while $\kappa^kr\leq r_x$,
$$\frac{\nu(\kappa^kB)}{\mu(\kappa^kB)}
\geq\left(\frac{C_\nu}{2C_\mu^{N}}\right)^{k}\frac{\nu(B)}{\mu(B)} .$$
For each small $r$ take $k=k(r)$ maximal with $\kappa^kr\leq r_x$; then
$\kappa^kr>r_x/\kappa$, so the left-hand side is at most
$\nu(B_{r_x}(x))/\mu(B_{r_x/\kappa}(x))$, a quantity independent of $r$, while
$k(r)\to\infty$ as $r\to0$. Since $C_\nu>2C_\mu^{N}$ we conclude
\begin{equation}\label{eqLimitBallsE}
	\lim_{r\to0}\frac{\nu(B_r(x))}{\mu(B_r(x))}=0
	\qquad\text{for every }x\in G .
\end{equation}
Let $\ve>0$ and, restricting to a bounded portion of $G$ if necessary, let $U$
be open with $G\subset U$ and $\mu(U)<\infty$. By \rf{eqLimitBallsE} every
$x\in G$ carries a ball $B_x\subset U$ centred at $x$ with
$\nu(5B_x)<\ve\mu(5B_x)$; extracting by the $5r$-covering theorem a disjoint
subfamily $\{B_k\}$ whose dilates $5B_k$ cover $G$,
$$\nu(G)\leq\sum_k\nu(5B_k)\leq\ve\sum_k\mu(5B_k)
\leq C_\mu^{3}\ve\sum_k\mu(B_k)\leq C_\mu^{3}\ve\,\mu(U),$$
and letting $\ve\to0$ gives $\nu(G)=0$.

Suppose now that $\nu|_V\not\ll\mu|_V$, so that there is a Borel set
$F\subset V$ with $\nu(F)>0$ and $\mu(F)=0$. Since $\nu(G)=0$, we may pick a
point $x\in F\setminus G$ which is a point of $\nu$-density of $F$, and a
sequence of balls $B_j=B_{r_j}(x)$ with $r_j\to0$ and
\begin{equation}\label{eqDoublingBalls}
	\nu(\kappa B_j)\leq C_\nu\,\nu(B_j) .
\end{equation}
Cover $B_j\cap X$ by the cubes $\{Q_j^k\}_{k\leq N_d}$ of the generation
determined by $r(B_j)\leq\ell(Q_j^k)<\rho_0 r(B_j)$; their number $N_d$ is
dimensional. Since
$r(B_j)>\rho_0^{-1}\ell(Q_j^k)=2(\rho_0A_0)^{-1}r_{Q_j^k}$, one checks from
\rf{eqTrivialInclusionDQP} that
\begin{equation}\label{eqInclusionsImpressive}
	2B_j\cap X\subset3Q_j^k\subset\kappa B_j .
\end{equation}
Let $Q_j$ be the cube $Q_j^k$ of largest $\nu$-measure, so that
$\nu(B_j)\leq N_d\,\nu(Q_j)$. For $j$ large we have $3Q_j\subset V$, that is
$Q_j\in\QQ_V$, because $V$ is relatively open. Combining,
$$\frac1{N_dC_\nu}
\by{\rf{eqDoublingBalls}}\leq\frac{N_d^{-1}\nu(B_j)}{\nu(\kappa B_j)}
\by{\rf{eqInclusionsImpressive}}\leq\frac{N_d^{-1}\nu(B_j)}{\nu(3Q_j)}
\leq\frac{\nu(Q_j)}{\nu(3Q_j)} .$$
As $\mu(F\cap3Q_j)=0<\delta_0\mu(3Q_j)$ whatever $\delta_0$ is,
\rf{eqWeakBLDyadic} applies with $E=F\cap 3Q_j$ and yields
$$\frac1{N_dC_\nu}\leq\ve_0\frac{\nu(3Q_j\cap F)}{\nu(3Q_j)}
+C_0\frac{\nu(3Q_j\setminus F)}{\nu(3Q_j)}
\by{\rf{eqInclusionsImpressive}}\leq
\ve_0+C_0C_\nu\,\frac{\nu(\kappa B_j\setminus F)}{\nu(\kappa B_j)} .$$
The last quotient tends to $0$ because $x$ is a point of $\nu$-density of $F$,
so letting $j\to\infty$ gives $\frac1{N_dC_\nu}\leq\ve_0$. Thus the lemma
holds with $\ve_1\coloneqq \left(2N_dC_\nu\right)^{-1}$, which depends only on
$d$ and $C_\mu$.
\end{proof}

\subsection{The auxiliary function}

Let $V\subset X$ be relatively open and bounded, let $\W$ be a Whitney
decomposition of $V$ as in Lemma \ref{lemWhitney} if $V\neq X$ and
$\W\coloneqq \{X\}$ otherwise, and set
\begin{equation}\label{eqDefPhi}
	\phi(x)\coloneqq \sum_{Q\in\W}\chi_Q(x)\,\mu(Q),\qquad x\in V .
\end{equation}
Thus $\phi(x)$ is the measure of the Whitney cube containing $x$, and it
decreases as $x$ approaches the relative boundary of $V$. As explained in the
introduction, this is the function that replaces the power of the distance to
the boundary used in \cite[Lemma 3.1]{BennewitzLewis}. Both this section and
\S\ref{secGehring} use exactly two properties of $\phi$. The first is
immediate: $\phi\leq\mu(V)$. The second is the source of the geometric factor
in Theorems \ref{theoAintro} and \ref{theoCintro}.

\begin{lemma}\label{lemInfPhi}
Let $\mu$ be a doubling Radon measure in $\R^d$, let $V\subset X$ be a bounded
relatively open set and let $K\subset V$ be compact. Then
$$\inf_{K}\phi\geq C_\mu^{-k_d-2}
\left(\frac{\diam K}{\dist(K,X\setminus V)}\right)
^{-\frac{\log C_\mu}{\log\rho_0}}\mu(K),$$
with $k_d$ depending only on $d$ and on the parameters of $\DD$.
\end{lemma}

\begin{proof}
If $V=X$ then $\phi\equiv\mu(X)\geq\mu(K)$ and there is nothing to prove, so
let $F\coloneqq X\setminus V\neq\varnothing$. Let $x\in K$ and let $Q_x\in\W$
contain $x$. By Wh\ref{itemWhRadius} and the triangle inequality,
$$\dist(K,F)\leq\dist(x,F)\leq\dist(Q_x,F)+\diam Q_x
\leq\left(\rho_0(\alpha\lambda+1)+2\right)r_{Q_x},$$
so $r_{Q_x}$ is bounded below by a fixed multiple of $\dist(K,F)$. On the
other hand, writing $S=\mathcal P^{k}(Q_x)$ for the $k$-th ancestor of $Q_x$,
whose exterior radius is $\rho_0^{k}r_{Q_x}$, we have $x\in S\subset B_S$ and
$3S\supset\lambda B_S\cap X$ by \rf{eqTrivialInclusionDQP}, so $3S\supset K$
as soon as $\rho_0^{k}r_{Q_x}\geq\frac{\diam K}{\lambda-1}$. Combining the
two, a single $k_0$ with
\begin{equation}\label{eqKZero}
	k_d+\log_{\rho_0}\frac{\diam K}{\dist(K,F)}\leq k_0
	\leq k_d+1+\log_{\rho_0}\frac{\diam K}{\dist(K,F)}
\end{equation}
serves every $x\in K$. Since $\mu$ is doubling, Lemma
\ref{lemDoublingNeighbors} applies with $Q_0$ arbitrarily large, so its items
hold for every cube; iterating DN\ref{itemDN2} along the chain
$Q_x\subset\widehat{Q_x}\subset\cdots\subset\mathcal P^{k_0}(Q_x)$, and then
DN\ref{itemDN1} once with $T=S$ to pass from $\mathcal P^{k_0}(Q_x)$ to its
triple,
$$\phi(x)=\mu(Q_x)\geq C_\mu^{-k_0}\mu\left(\mathcal P^{k_0}(Q_x)\right)
\geq C_\mu^{-k_0-1}\mu\left(3\mathcal P^{k_0}(Q_x)\right)
\geq C_\mu^{-k_0-1}\mu(K) ,$$
and \rf{eqKZero} turns $C_\mu^{-k_0-1}$ into the stated quantity.
\end{proof}

\subsection{Proof of Theorem \ref{theoAintro}}

By Lemma \ref{lemWeakBLDyadic} we may work with the dyadic condition
\rf{eqWeakBLDyadic}, and by Lemma \ref{lemBLAC} the density
$w=\frac{d\nu}{d\mu}$ exists. What we prove is the following quantitative
statement, of which Theorem \ref{theoAintro} is the case of small $\ve_0$.

\begin{proposition}\label{propoBL}
Let $\mu$ be a doubling Radon measure in $\R^d$, let $V\subset X$ be a bounded
relatively open set, let $K\subset V$ be compact with $\mu(K)>0$, and let
$\nu\in wBL_{\ve_0}(\mu,\QQ_V)$ with constants $\delta_0$, $C_0$ and
\begin{equation}\label{eqEpsThreshold}
	\ve_0\leq\min\left\{\ve_1,\ \left(2+2C_\mu^4\right)^{-1}\right\},
\end{equation}
with $\ve_1$ as in Lemma \ref{lemBLAC}, and write $w=\frac{d\nu}{d\mu}$, which
exists by that lemma. Then there is $q_0>1$, depending on $d$, on $C_\mu$ and
on $\delta_0$, $C_0$, such that for every $1<q\leq q_0$
$$\left(\avint_{K}w^q\, d\mu\right)^{\frac1q}
\leq C\left(\frac{\diam K}{\dist(K,X\setminus V)}\right)
^{\frac{\log C_\mu}{\log\rho_0}}
\left(\frac{\mu(V)}{\mu(K)}\right)^{1+\frac1q}\frac{\nu(V)}{\mu(V)} ,$$
with $C$ depending on $d$, $C_\mu$, $\delta_0$, $C_0$ and $q$.
\end{proposition}

\begin{proof}
Let $\W$ and $\phi$ be as in \rf{eqDefPhi}. Since $\phi\equiv\mu(Q)$ on each
$Q\in\W$, we have
$\avint_Q\phi w\, d\mu=\nu(Q)\leq\nu(V)$, so we may run the
Calder\'on--Zygmund iteration of Lemma \ref{lemCZIteration} inside each
$Q\in\W$ with
$$f=\phi w,\qquad \alpha_0\coloneqq \nu(V),\qquad
\rho\coloneqq C_\mu\tau^{-1},\qquad \alpha_j\coloneqq \rho^{j}\alpha_0,$$
and $0<\tau<1$ to be fixed. Writing $\mathcal F^j$ for the union over
$Q\in\W$ of the resulting stopping families, $U^j\coloneqq
\bigcup_{R\in\mathcal F^j}R$, and
$$A^j\coloneqq \{x\in V:\, \phi(x)w(x)>\alpha_j\},$$
properties CZ\ref{itemCZU}--CZ\ref{itemCZLayers} give, for every $j\geq0$,
\begin{equation}\label{eqCZforBL}
\begin{aligned}
	&\alpha_j<\avint_{R}\phi w\, d\mu=\frac{\mu(Q)\nu(R)}{\mu(R)}
	\leq C_\mu\alpha_j
	\qquad\text{for } R\in\mathcal F^j\text{ with } R\subset Q\in\W,\\
	&A^j\subset U^j\cup Z^j \qquad\text{with }\mu(Z^j)=0,
\end{aligned}
\end{equation}
together with CZ\ref{itemCZIhat} and $\mu(U^j)\to0$.

\emph{Step 1.} We claim that if
\begin{equation}\label{eqTauSmall}
	\tau<\frac1{2C_0C_\mu^{2}} ,
\end{equation}
then for every $j\geq1$ there is a subfamily $\mathcal G\subset\mathcal F^j$
with
\begin{equation}\label{eqBLFirstStep}
	\begin{cases}
		\mu(A^{j-1}\cap 3P)\geq\delta_0\,\mu(3P) &\text{for every }
		P\in\mathcal G,\\
		U^{j}\subset\bigcup_{P\in\mathcal G}10\lambda^2A_0B_P, &\\
		\sum_{P\in\mathcal G}\chi_{3P}\leq C_{\mathrm B} . &
	\end{cases}
\end{equation}
Consider first the subfamily of cubes whose triples are maximal,
$$\wt{\mathcal F}^j\coloneqq \left\{P\in\mathcal F^j:\
3R\not\supset3P\ \text{ for every }R\in\mathcal F^j\setminus\{P\}\right\} ,$$
which is well defined and non-empty whenever $\mathcal F^j$ is, because $V$ is
bounded; an induction on the generation of the cubes gives
\begin{equation}\label{eqBigcup3PUj}
	\bigcup_{P\in\wt{\mathcal F}^j}3P=\bigcup_{R\in\mathcal F^j}3R\supset U^j .
\end{equation}
We show that no $P\in\wt{\mathcal F}^j$ can satisfy
\begin{equation}\label{eqBLFirstStepCont}
	\mu(A^{j-1}\cap3P)<\delta_0\,\mu(3P) .
\end{equation}
Let $P\in\wt{\mathcal F}^j$, say $P\subset Q\in\W$; note that
$\avint_Q\phi w\, d\mu\leq\alpha_0\leq\alpha_j$ forces $P\neq Q$, so
$P\subsetneq Q$. Every $R\in\NN(P)\setminus\mathcal F^j$ satisfies
\begin{equation}\label{eqNeighborBelowLevel}
	\avint_{R}\phi w\, d\mu\leq\alpha_j .
\end{equation}
Indeed, otherwise $R$ would be contained in a maximal cube
$R'\in\mathcal F^j$ with the same property; as $R\notin\mathcal F^j$ the
inclusion is strict, so $r_{R'}\geq\rho_0 r_R=\rho_0 r_P$, and
N\ref{eqrem3QNeighbors5} would give $3P\subset3R'$ with
$R'\in\mathcal F^j\setminus\{P\}$, contrary to the definition of
$\wt{\mathcal F}^j$.

Let $R\in\NN(P)$ and let $\wt Q\in\W$ be the Whitney cube containing it. By
\rf{eqNeighborBelowLevel} and \rf{eqCZforBL}, in either case
$\frac{\mu(\wt Q)\nu(R)}{\mu(R)}\leq C_\mu\alpha_j$, while \rf{eqCZforBL} at
$P$ gives $\alpha_j<\frac{\mu(Q)\nu(P)}{\mu(P)}$; and Lemma
\ref{lemNeighborWhitney} applied to $P\subsetneq Q$ yields
\begin{equation}\label{eqQwtQSimilar}
	\mu(\wt Q)\geq C_\mu^{-2}\mu(Q) .
\end{equation}
Combining the three and summing over $R\in\NN(P)$, we find that $P$ is
$\nu$-doubling:
\begin{equation}\label{eqPisNuDoubling}
\begin{aligned}
	\nu(3P)&\leq\nu(P)+C_\mu^{3}\frac{\nu(P)}{\mu(P)}
	\sum_{R\in\NN(P)}\mu(R)
	=\nu(P)\left(1+C_\mu^{3}\frac{\mu(3P)}{\mu(P)}\right)\\
	&\leq\left(1+C_\mu^{4}\right)\nu(P).
\end{aligned}
\end{equation}
On the other hand, \rf{eqBLFirstStepCont} and \rf{eqWeakBLDyadic} with
$E=A^{j-1}\cap3P$ give
$\nu(P)\leq\ve_0\nu(3P\cap A^{j-1})+C_0\nu(3P\setminus A^{j-1})$, so
\rf{eqPisNuDoubling} and $\ve_0(1+C_\mu^4)\leq\frac12$ yield
\begin{equation}\label{eqPby3PminusA}
	\nu(P)\leq2C_0\,\nu(3P\setminus A^{j-1}) .
\end{equation}
Finally, \rf{eqQwtQSimilar} gives $\mu(Q)\leq C_\mu^{2}\phi(x)$ for
$x\in3P$, so
$$\frac{\alpha_j\mu(P)}{\mu(Q)}\by{\rf{eqCZforBL}}<\nu(P)
\by{\rf{eqPby3PminusA},\ \rf{eqQwtQSimilar}}\leq
\frac{2C_0C_\mu^{2}}{\mu(Q)}\int_{3P\setminus A^{j-1}}\phi w\, d\mu
\leq\frac{2C_0C_\mu^{3}}{\rho}\,\frac{\alpha_j\mu(P)}{\mu(Q)} ,$$
whence $\tau=C_\mu\rho^{-1}\geq(2C_0C_\mu^{2})^{-1}$, contradicting
\rf{eqTauSmall}. Thus every $P\in\wt{\mathcal F}^j$ satisfies the first line of
\rf{eqBLFirstStep}, and so does every subfamily of $\wt{\mathcal F}^j$. Now
apply Lemma \ref{lemBesicovitchSelection} to $\wt{\mathcal F}^j$, whose cubes
are pairwise disjoint and contained in the bounded set $V$: the resulting
$\mathcal G$ satisfies the third line of \rf{eqBLFirstStep}, and the second
follows from \rf{eqBigcup3PUj}.

\emph{Step 2.} We claim that, with
$C_1\coloneqq 1+C_{\mathrm B}C_\mu\wt C_\mu\delta_0^{-1}$,
\begin{equation}\label{eqBLSecondStep}
	\int_{A^{j-1}}\phi w\, d\mu\leq C_1\,\alpha_j\,\mu(A^{j-1})
	\qquad\text{for every }j\geq1 ,
\end{equation}
with $\wt C_\mu$ as in \rf{eqEnlargedBallDoubling}. Indeed, since
$\mathcal F^j$ is a disjoint family, \rf{eqCZforBL} gives
$$\int_{A^{j}}\phi w\, d\mu\leq\int_{U^{j}}\phi w\, d\mu
=\sum_{P\in\mathcal F^j}\int_{P}\phi w\, d\mu
\leq C_\mu\alpha_j\sum_{P\in\mathcal F^j}\mu(P)=C_\mu\alpha_j\,\mu(U^j),$$
and the three lines of \rf{eqBLFirstStep}, in this order, give
$$\begin{aligned}
\mu(U^j)&\leq\sum_{P\in\mathcal G}\mu\left(10\lambda^2A_0B_P\right)
\leq\wt C_\mu\sum_{P\in\mathcal G}\mu(3P)\\
&\leq\frac{\wt C_\mu}{\delta_0}\sum_{P\in\mathcal G}\mu(3P\cap A^{j-1})
\leq\frac{C_{\mathrm B}\wt C_\mu}{\delta_0}\,\mu(A^{j-1}) .
\end{aligned}$$
Adding
$\int_{A^{j-1}\setminus A^{j}}\phi w\, d\mu\leq\alpha_j\mu(A^{j-1})$ gives
\rf{eqBLSecondStep}.

\emph{Step 3.} Fix now
\begin{equation}\label{eqTauChoice}
	\tau\coloneqq \min\left\{\frac{C_\mu}{4C_1},\ \frac1{4C_0C_\mu^{2}}\right\},
\end{equation}
so that \rf{eqTauSmall} holds and $\rho=C_\mu\tau^{-1}\geq4C_1>2C_1$. Then
\rf{eqBLSecondStep} is exactly hypothesis \rf{eqBasicLevels} of Lemma
\ref{lemLevelToQ} for the measure $\mu|_V$, the function $f=\phi w$ and
$p=1$, so for every $1<q\leq q_0$, with $q_0$ as in \rf{eqQZero},
\begin{equation}\label{eqBLThirdStep}
	\int_{V}(\phi w)^q\, d\mu
	\leq C_{C_1,\rho,1,q}\,\nu(V)^{q-1}\int_{V}\phi\, d\nu
	\leq C_{C_1,\rho,1,q}\,\nu(V)^{q}\,\mu(V),
\end{equation}
where we used $\phi\leq\mu(V)$. Restricting to $K$ and using $w\leq\phi
w/\inf_K\phi$ there,
$$\left(\avint_{K}w^q\, d\mu\right)^{\frac1q}
\leq\frac1{\mu(K)^{\frac1q}\inf_{K}\phi}
\left(\int_{V}(\phi w)^q\, d\mu\right)^{\frac1q}
\by{\rf{eqBLThirdStep}}\leq
\frac{C_{C_1,\rho,1,q}^{\frac1q}\,\mu(V)^{\frac1q}\,\nu(V)}
{\mu(K)^{\frac1q}\,\inf_{K}\phi} .$$

Inserting the lower bound for $\inf_K\phi$ given by Lemma \ref{lemInfPhi}
proves the proposition.
\end{proof}

Theorem \ref{theoAintro} is Proposition \ref{propoBL} combined with Lemma
\ref{lemWeakBLDyadic}, so that its threshold is
$$\ve_*=C_\DD^{-1}\min\left\{\ve_1,\ \left(2+2C_\mu^4\right)^{-1}\right\},$$
the factor $C_\DD$ being the price of passing from balls to the basis
$\DD\cup3\DD$. Letting $K$ exhaust a ball and taking $V=2\Delta$ gives
$\nu\in wRH_q(\mu,V)$ for $1<q\leq q_0$, that is
$wRH_1^{BL}\subset wRH_1^{U}$, and with Lemmas \ref{lemWeakTrivial1} and
\ref{lemWeakTrivial2} this proves Corollary \ref{coroBintro}.

\section{The bootstrapping lemma}
\label{secBootstrap}

The proofs of Theorems \ref{theoAintro} and \ref{theoCintro} both end by
feeding an estimate for the level sets of a function into the lemma of this
section. The bootstrapping argument is Gehring's \cite{Gehring}; the case
$p=1$ is the form in which Bennewitz and Lewis use it \cite{BennewitzLewis}.
We state it once, for all $p\geq1$, because the two cases differ only in the
exponents that appear, and we keep track of the constants, since the
admissible range of exponents in Theorems \ref{theoAintro} and
\ref{theoCintro} is exactly the range produced here.

Throughout this section
\begin{equation}\label{eqDefSigma}
	\sigma\coloneqq \begin{cases}0 &\text{if }p=1,\\ 1&\text{if }p>1,\end{cases}
\end{equation}
and $f^0\equiv1$, so that $\int_E f^\sigma\, d\mu$ means $\mu(E)$ when $p=1$
and $\int_E f\, d\mu$ when $p>1$.

\begin{lemma}\label{lemLevelToQ}
Let $\mu$ be a measure on a measurable space, let $V$ be a measurable set, let
$p\geq1$ and let $f\colon V\to[0,\infty]$ be measurable with
$f\in L^p(\mu|_V)$. Let $\alpha_0>0$ and $\rho>1$, put
$\alpha_j\coloneqq \rho^j\alpha_0$ and
$$A^j\coloneqq \{x\in V:\, f(x)>\alpha_j\},\qquad j\geq0 .$$
Assume that, with $\sigma$ as in \rf{eqDefSigma},
\begin{equation}\label{eqBasicLevels}
	\int_{A^{j}}f^p\, d\mu\leq C_1\,\alpha_{j+1}^{\,p-\sigma}
	\int_{A^{j}}f^{\,\sigma}\, d\mu\qquad\text{for every }j\geq0,
\end{equation}
and that
\begin{equation}\label{eqRhoLarge}
	\rho^{\,p-\sigma}>2C_1>2 .
\end{equation}
Then, for every $q>p$ for which
\begin{equation}\label{eqConstantLevelToQ}
	C_{C_1,\rho,p,q}\coloneqq
	\frac{\rho^{q-p}\left(\rho^{\,p-\sigma}-1\right)}
	{\rho^{\,q-\sigma}-1-2C_1\rho^{\,q-\sigma}
		\left(\rho^{\,q-\sigma}-\rho^{\,p-\sigma}\right)}
\end{equation}
is positive,
\begin{equation}\label{eqConclusionLevelToQ}
	\int_{V}f^q\, d\mu\leq C_{C_1,\rho,p,q}\,
	\alpha_0^{q-p}\int_{V}f^p\, d\mu .
\end{equation}
\end{lemma}

Here $p-\sigma$ equals $1$ if $p=1$ and $p-1$ if $p>1$; in particular it is
positive, so \rf{eqRhoLarge} is a condition on how fast the levels $\alpha_j$
are made to grow. Written out, \rf{eqBasicLevels} is
$\int_{A^{j}}f\, d\mu\leq C_1\alpha_{j+1}\mu(A^{j})$ when $p=1$, and
$\int_{A^{j}}f^p\, d\mu\leq C_1\alpha_{j+1}^{p-1}\int_{A^{j}}f\, d\mu$ when
$p>1$; correspondingly, \rf{eqConstantLevelToQ} is
$$\begin{aligned}
C_{C_1,\rho,1,q}&=\frac{\rho^{q-1}\left(\rho-1\right)}
{\rho^{q}-1-2C_1\rho^{q}\left(\rho^{q}-\rho\right)},\\
C_{C_1,\rho,p,q}&=\frac{\rho^{q-p}\left(\rho^{p-1}-1\right)}
{\rho^{q-1}-1-2C_1\rho^{q+p-2}\left(\rho^{q-p}-1\right)}
\qquad (p>1).
\end{aligned}$$

The denominator of \rf{eqConstantLevelToQ} equals $\rho^{p-\sigma}-1>0$ at
$q=p$ and is continuous in $q$, so the admissible exponents form a
non-degenerate interval. We record a concrete one.

\begin{remark}\label{remRangeOfQ}
Under \rf{eqRhoLarge}, the quantity \rf{eqConstantLevelToQ} is positive for
every $p<q\leq q_0$, where
\begin{equation}\label{eqQZero}
	q_0\coloneqq p+\frac1{\log\rho}
	\log\left(1+\frac{\rho^{\,p-\sigma}-1}{4C_1\rho^{\,2(p-\sigma)}}\right) .
\end{equation}
Indeed, write $\eta$ for the fraction inside the logarithm. By
\rf{eqRhoLarge} we have $\eta<\left(4C_1\rho^{\,p-\sigma}\right)^{-1}<1$, so
$q\leq q_0$ gives $\rho^{q-p}\leq1+\eta<2$ and hence
$\rho^{\,q-\sigma}=\rho^{q-p}\rho^{\,p-\sigma}<2\rho^{\,p-\sigma}$, whence
$$\begin{aligned}
2C_1\rho^{\,q-\sigma}\left(\rho^{\,q-\sigma}-\rho^{\,p-\sigma}\right)
&=2C_1\rho^{\,q-\sigma}\rho^{\,p-\sigma}\left(\rho^{q-p}-1\right)\\
&\leq4C_1\rho^{\,2(p-\sigma)}\eta=\rho^{\,p-\sigma}-1
<\rho^{\,q-\sigma}-1 .
\end{aligned}$$
Moreover, whenever \rf{eqConstantLevelToQ} is positive it is at least $1$:
writing $u\coloneqq \rho^{\,q-\sigma}$ and $a\coloneqq \rho^{\,p-\sigma}$, so
that $\rho^{q-p}=u/a$, the numerator minus the denominator of
\rf{eqConstantLevelToQ} equals
\begin{equation}\label{eqCAtLeastOne}
	\frac ua\left(a-1\right)-\left(u-1-2C_1u(u-a)\right)
	=\left(u-a\right)\left(2C_1u-\frac1a\right)\geq0 ,
\end{equation}
since $u\geq a>1$ and $2C_1>2$.
\end{remark}

\begin{proof}[Proof of Lemma \ref{lemLevelToQ}]
For $t\geq0$ and $N\geq1$ write
$$I_N(t)\coloneqq \int_{A^0\setminus A^N}f^t\, d\mu ,$$
so that $I_N(0)=\mu(A^0\setminus A^N)$. All these quantities are finite: on
$A^0\setminus A^N$ we have $\alpha_0<f\leq\alpha_N$, so
$f^t\leq\max\{\alpha_0^{\,t-p},\alpha_N^{\,t-p}\}f^p$ there and
$I_N(t)\leq\max\{\alpha_0^{\,t-p},\alpha_N^{\,t-p}\}\|f\|_{L^p(\mu|_V)}^p$ for
every $t\geq0$. Arguing in the same way on $A^N$, where $f>\alpha_N$, and on
$A^{j}$, where $f>\alpha_j$ and $\sigma\leq p$,
\begin{equation}\label{eqMuANtoZero}
	\mu(A^N)\leq\alpha_N^{-p}\|f\|_{L^p(\mu|_V)}^p\xrightarrow{N\to\infty}0,
	\qquad
	\int_{A^{j}}f^{\,\sigma}\, d\mu\leq\alpha_j^{\,\sigma-p}
	\|f\|_{L^p(\mu|_V)}^p<\infty .
\end{equation}

\emph{Step 1: a truncated form of the hypothesis.} We claim that
\begin{equation}\label{eqEverythingBegins}
	\int_{A^{j}\setminus A^N}f^p\, d\mu\leq2C_1\,\alpha_{j+1}^{\,p-\sigma}
	\int_{A^{j}\setminus A^N}f^{\,\sigma}\, d\mu
	\qquad\text{for every }0\leq j\leq N-1 .
\end{equation}
If $N=j+1$ this is immediate, and with $1$ in place of $2C_1$: on
$A^{j}\setminus A^{j+1}$ we have $f\leq\alpha_{j+1}$, so
$f^p\leq\alpha_{j+1}^{\,p-\sigma}f^{\,\sigma}$ there. Assume then
$N\geq j+2$. Since $f>\alpha_N$ on $A^N$ and $A^N\subset A^{j}$,
$$\int_{A^N}f^{\,\sigma}\, d\mu\leq\frac1{\alpha_N^{\,p-\sigma}}
\int_{A^N}f^{p}\, d\mu
\leq\frac1{\rho^{(p-\sigma)(N-j-1)}\alpha_{j+1}^{\,p-\sigma}}
\int_{A^{j}}f^{p}\, d\mu ,$$
because $\alpha_N=\rho^{N-j-1}\alpha_{j+1}$. Subtracting this from
\rf{eqBasicLevels}, which is legitimate because
$\int_{A^j}f^\sigma\, d\mu<\infty$, we get
$$\int_{A^{j}\setminus A^N}f^{\,\sigma}\, d\mu
\by{\rf{eqBasicLevels}}\geq\left(\frac1{C_1}
-\frac1{\rho^{(p-\sigma)(N-j-1)}}\right)
\frac1{\alpha_{j+1}^{\,p-\sigma}}\int_{A^{j}}f^{p}\, d\mu .$$
By \rf{eqRhoLarge} and $N-j-1\geq1$ we have
$\rho^{(p-\sigma)(N-j-1)}\geq\rho^{\,p-\sigma}>2C_1$, so the bracket is at
least $(2C_1)^{-1}$, and \rf{eqEverythingBegins} follows because
$A^{j}\setminus A^N\subset A^{j}$.

\emph{Step 2: summation by parts.} For every $t\geq0$ and every measurable
$g\geq0$ with $\int_{A^0\setminus A^N}g\, d\mu<\infty$,
\begin{equation}\label{eqSummationByParts}
	\sum_{j=0}^{N-1}\alpha_{j+1}^{t}\int_{A^{j}\setminus A^{j+1}}g\, d\mu
	=\alpha_1^{t}\int_{A^0\setminus A^N}g\, d\mu
	+\left(\rho^{t}-1\right)\sum_{k=1}^{N-1}\alpha_{k}^{t}
	\int_{A^{k}\setminus A^N}g\, d\mu .
\end{equation}
Indeed, writing $a_j\coloneqq \int_{A^{j}\setminus A^{j+1}}g\, d\mu$ and
exchanging the order of summation, the right-hand side equals
$\sum_{j=0}^{N-1}a_j\bigl(\alpha_1^t+(\rho^t-1)\sum_{k=1}^{j}\alpha_k^t\bigr)$,
and the geometric sum telescopes to
$\alpha_1^t+\alpha_0^t\rho^{t}\left(\rho^{jt}-1\right)=\alpha_{j+1}^{t}$.

\emph{Step 3: the bootstrap.} Let $q>p$. Since $f\leq\alpha_{j+1}$ on
$A^{j}\setminus A^{j+1}$,
$$\begin{aligned}
I_N(q)&\leq\sum_{j=0}^{N-1}\alpha_{j+1}^{\,q-p}
\int_{A^{j}\setminus A^{j+1}}f^{p}\, d\mu\\
&\by{\rf{eqSummationByParts}}=\alpha_1^{\,q-p}I_N(p)
+\left(\rho^{\,q-p}-1\right)\sum_{k=1}^{N-1}\alpha_{k}^{\,q-p}
\int_{A^{k}\setminus A^N}f^{p}\, d\mu .
\end{aligned}$$
By \rf{eqEverythingBegins} and $\alpha_{k+1}=\rho\alpha_k$, each term of the
last sum obeys
$$\alpha_{k}^{\,q-p}\int_{A^{k}\setminus A^N}f^{p}\, d\mu
\leq2C_1\rho^{\,p-\sigma}\,\alpha_k^{\,q-\sigma}
\int_{A^{k}\setminus A^N}f^{\,\sigma}\, d\mu .$$
Writing
$$\Theta\coloneqq \frac{2C_1\rho^{\,p-\sigma}\left(\rho^{\,q-p}-1\right)}
{\rho^{\,q-\sigma}-1}
=\frac{2C_1\left(\rho^{\,q-\sigma}-\rho^{\,p-\sigma}\right)}
{\rho^{\,q-\sigma}-1},$$
and applying \rf{eqSummationByParts} again, now with $t=q-\sigma$ and
$g=f^{\,\sigma}$, we obtain
$$I_N(q)\leq\alpha_1^{\,q-p}I_N(p)+\Theta\left(
\sum_{j=0}^{N-1}\alpha_{j+1}^{\,q-\sigma}
\int_{A^{j}\setminus A^{j+1}}f^{\,\sigma}\, d\mu
-\alpha_1^{\,q-\sigma}I_N(\sigma)\right).$$
Two bounds now close the argument. On $A^{j}$ we have
$\alpha_{j+1}=\rho\alpha_{j}<\rho f$, so
$\alpha_{j+1}^{\,q-\sigma}f^{\,\sigma}\leq\rho^{\,q-\sigma}f^{q}$ and the sum
inside the bracket is at most $\rho^{\,q-\sigma}I_N(q)$; and
\rf{eqEverythingBegins} with $j=0$ reads
$I_N(p)\leq2C_1\alpha_1^{\,p-\sigma}I_N(\sigma)$, that is,
$\alpha_1^{\,q-\sigma}I_N(\sigma)\geq(2C_1)^{-1}\alpha_1^{\,q-p}I_N(p)$.
Therefore
$$I_N(q)\left(1-\Theta\rho^{\,q-\sigma}\right)
\leq\alpha_1^{\,q-p}I_N(p)\left(1-\frac{\Theta}{2C_1}\right),$$
and since
$$1-\frac{\Theta}{2C_1}=\frac{\rho^{\,p-\sigma}-1}{\rho^{\,q-\sigma}-1},$$
the factor on the right is always positive, while the one on the left is
positive precisely when the denominator of \rf{eqConstantLevelToQ} is. In that
case, using $\alpha_1^{\,q-p}=\rho^{\,q-p}\alpha_0^{\,q-p}$,
\begin{equation}\label{eqINqbyINp}
	I_N(q)\leq C_{C_1,\rho,p,q}\,\alpha_0^{\,q-p}I_N(p) .
\end{equation}

\emph{Step 4: conclusion.} Since $f\leq\alpha_0$ on $V\setminus A^0$,
$$\int_{V}f^q\, d\mu\leq\alpha_0^{\,q-p}\int_{V\setminus A^0}f^{p}\, d\mu
+\int_{A^0}f^{q}\, d\mu .$$
The sets $A^0\setminus A^N$ increase to $A^0\setminus\bigcap_{N}A^N$, and
$\mu\bigl(\bigcap_{N}A^N\bigr)=0$ by \rf{eqMuANtoZero}, so
$I_N(q)\to\int_{A^0}f^q\, d\mu$ and $I_N(p)\to\int_{A^0}f^p\, d\mu$ by
monotone convergence. Letting $N\to\infty$ in \rf{eqINqbyINp} we get
$$\int_{V}f^q\, d\mu\leq\alpha_0^{\,q-p}\left(
\int_{V\setminus A^0}f^{p}\, d\mu
+C_{C_1,\rho,p,q}\int_{A^0}f^{p}\, d\mu\right),$$
and \rf{eqConclusionLevelToQ} follows because $C_{C_1,\rho,p,q}\geq1$ by
Remark \ref{remRangeOfQ}.
\end{proof}

We shall apply Lemma \ref{lemLevelToQ} twice: with $p=1$ in
\S\ref{secBL}, where \rf{eqBasicLevels} comes from the
Calder\'on--Zygmund iteration on $\phi w$ and the Bennewitz--Lewis condition,
and with $p>1$ in \S\ref{secGehring}, where it comes from the same iteration
and a weak reverse H\"older inequality.

\section{Gehring's lemma}
\label{secGehring}

We now prove Theorem \ref{theoCintro}: the weak reverse H\"older inequality
improves itself. The scheme is the one of \S\ref{secBL} with $p>1$ in place of
$p=1$ --- the same Whitney decomposition, the same auxiliary function
\rf{eqDefPhi}, the same selection lemma and the same bootstrap. What changes
is the stopping estimate: there, the Bennewitz--Lewis condition was used to
show that $3P$ carries a definite proportion of the previous level set; here
the reverse H\"older inequality is used to show that $3P$ carries a definite
amount of the integral over it. Accordingly we only write out the two steps
that differ.

\subsection{The dyadic form}

\begin{lemma}\label{lemGehringEnlarged}
Let $\mu$ be a doubling Radon measure in $\R^d$, let $V\subset X$ be a bounded
relatively open set, let $K\subset V$ be compact with
$\mu(K)>0$, let $1<p<\infty$ and let $\nu\in wRH_p(\mu,\QQ_V)$ with constant
$C_{RH}$, in the sense of Definition \ref{defWeakDyadic}. Then there is
$q_0>p$, depending on $d$, $C_\mu$, $p$ and $C_{RH}$, such that for every
$p<q\leq q_0$
$$\left(\avint_{K}w^{q}\, d\mu\right)^{\frac1q}
\leq C\left(\frac{\diam K}{\dist(K,X\setminus V)}\right)
^{\frac{\log C_\mu}{\log\rho_0}}
\left(\frac{\mu(V)}{\mu(K)}\right)^{1+\frac1q}
\left(\avint_{V}w^{p}\, d\mu\right)^{\frac1p} ,$$
with $C$ depending on $d$, $C_\mu$, $p$, $C_{RH}$ and $q$.
\end{lemma}

\begin{proof}
Let $\W$ be a Whitney decomposition of $V$, or $\W\coloneqq \{X\}$ if $V=X$,
and let $\phi$ be as in \rf{eqDefPhi}. Put
$$\alpha_0\coloneqq \mu(V)\left(\avint_{V}w^{p}\, d\mu\right)^{\frac1p},
\qquad \alpha_j\coloneqq \rho^{j}\alpha_0,\qquad
\rho\coloneqq C_\mu^{2}\tau^{-1},$$
with $0<\tau<1$ to be fixed. Since $\phi\equiv\mu(Q)$ on $Q\in\W$,
$$\avint_{Q}(\phi w)^p\, d\mu=\mu(Q)^{p-1}\int_{Q}w^p\, d\mu
\leq\mu(V)^{p-1}\int_{V}w^p\, d\mu=\alpha_0^{p},$$
so we may run the Calder\'on--Zygmund iteration of Lemma \ref{lemCZIteration}
inside each $Q\in\W$ with $f=(\phi w)^p$ and levels $\alpha_j^{p}$; in the
notation of that lemma this amounts to replacing $\tau$ by
$C_\mu^{1-2p}\tau^{p}$, which is admissible because $\rho^{p}\geq C_\mu$. We
write $\mathcal F^j$, $U^j$ as in \S\ref{secBL} and
$$A^j\coloneqq \{x\in V:\, \phi(x)w(x)>\alpha_j\}\subset U^j
\qquad (j\geq-1,\ U^{-1}\coloneqq V).$$

\emph{Step 1.} We claim that if
\begin{equation}\label{eqTauSmallGehring}
	\tau\leq\frac1{2C_{RH}C_\mu},
\end{equation}
then every $P\in\mathcal F^j$ satisfies
\begin{equation}\label{eqGehringStep1}
	\alpha_j\,\mu(3P)<2C_{RH}C_\mu^{3}\int_{3P\cap A^{j-1}}\phi w\, d\mu .
\end{equation}
Let $P\in\mathcal F^j$ and let $Q\in\W$ contain it; as in \S\ref{secBL},
$\avint_Q(\phi w)^p\, d\mu\leq\alpha_0^p\leq\alpha_j^p$ and CZ\ref{itemCZU}
force $P\subsetneq Q$, so Lemma \ref{lemNeighborWhitney} gives
$\mu(Q)\leq C_\mu^{2}\phi(x)$ for every $x\in3P$. Hence, by
CZ\ref{itemCZU}, by Definition \ref{defWeakDyadic} and by
$\mu(3P)\leq C_\mu\mu(P)$,
$$\begin{aligned}
\alpha_j&<\left(\avint_{P}(\phi w)^p\, d\mu\right)^{\frac1p}
=\mu(Q)\left(\avint_{P}w^p\, d\mu\right)^{\frac1p}\\
&\leq C_{RH}\,\mu(Q)\,\frac{\nu(3P)}{\mu(P)}
\leq C_{RH}C_\mu^{3}\avint_{3P}\phi w\, d\mu .
\end{aligned}$$
Splitting the last average over $A^{j-1}$ and its complement, where
$\phi w\leq\alpha_{j-1}=C_\mu^{-2}\tau\alpha_j$, gives
$$\alpha_j\,\mu(3P)<C_{RH}C_\mu^{3}\int_{3P\cap A^{j-1}}\phi w\, d\mu
+C_{RH}C_\mu\tau\,\alpha_j\,\mu(3P),$$
and \rf{eqTauSmallGehring} lets the last term be absorbed, proving
\rf{eqGehringStep1}.

\emph{Step 2.} We claim that, with
$C_1\coloneqq 1+2C_{\mathrm B}C_{RH}C_\mu^{4}\wt C_\mu$ and $\wt C_\mu$ as in
\rf{eqEnlargedBallDoubling},
\begin{equation}\label{eqGehringSecondStep}
	\int_{A^{j-1}}(\phi w)^p\, d\mu
	\leq C_1\,\alpha_j^{\,p-1}\int_{A^{j-1}}\phi w\, d\mu
	\qquad\text{for every }j\geq0 .
\end{equation}
Indeed, since $A^j\subset U^j$ and $\mathcal F^j$ is disjoint,
CZ\ref{itemCZU} gives
$$\int_{A^{j}}(\phi w)^p\, d\mu
\leq\sum_{P\in\mathcal F^j}\mu(P)\avint_{P}(\phi w)^p\, d\mu
\leq C_\mu\alpha_j^{p}\,\mu(U^j),$$
and Lemma \ref{lemBesicovitchSelection} applied to $\mathcal F^j$, followed by
\rf{eqGehringStep1}, yields
$$\begin{aligned}
\mu(U^j)&\leq\wt C_\mu\sum_{P\in\mathcal G}\mu(3P)
<\frac{2C_{RH}C_\mu^{3}\wt C_\mu}{\alpha_j}
\sum_{P\in\mathcal G}\int_{3P\cap A^{j-1}}\phi w\, d\mu\\
&\leq\frac{2C_{\mathrm B}C_{RH}C_\mu^{3}\wt C_\mu}{\alpha_j}
\int_{A^{j-1}}\phi w\, d\mu .
\end{aligned}$$
Adding
$\int_{A^{j-1}\setminus A^{j}}(\phi w)^p\, d\mu
\leq\alpha_j^{p-1}\int_{A^{j-1}\setminus A^{j}}\phi w\, d\mu$ gives
\rf{eqGehringSecondStep}.

\emph{Step 3.} Fix
\begin{equation}\label{eqTauChoiceGehring}
	\tau\coloneqq \min\left\{\frac{C_\mu^{2}}{(4C_1)^{\frac1{p-1}}},\
	\frac1{2C_{RH}C_\mu}\right\},
\end{equation}
so that \rf{eqTauSmallGehring} holds and $\rho^{\,p-1}\geq4C_1>2C_1$. Then
\rf{eqGehringSecondStep} is hypothesis \rf{eqBasicLevels} of Lemma
\ref{lemLevelToQ} for $\mu|_V$, $f=\phi w$ and this $p$, so for every
$p<q\leq q_0$, with $q_0$ as in \rf{eqQZero},
$$\int_{V}(\phi w)^q\, d\mu
\leq C_{C_1,\rho,p,q}\,\alpha_0^{\,q-p}\int_{V}(\phi w)^p\, d\mu
\leq C_{C_1,\rho,p,q}\,\mu(V)^{q+1}
\left(\avint_{V}w^{p}\, d\mu\right)^{\frac qp},$$
where we used $\phi\leq\mu(V)$ and the definition of $\alpha_0$. Since
$w\leq\phi w/\inf_K\phi$ on $K$,
$$\begin{aligned}
\left(\avint_{K}w^q\, d\mu\right)^{\frac1q}
&\leq\frac1{\mu(K)^{\frac1q}\inf_{K}\phi}
\left(\int_{V}(\phi w)^q\, d\mu\right)^{\frac1q}\\
&\leq\frac{C_{C_1,\rho,p,q}^{\frac1q}\,\mu(V)^{1+\frac1q}}
{\mu(K)^{\frac1q}\,\inf_{K}\phi}
\left(\avint_{V}w^{p}\, d\mu\right)^{\frac1p} ,
\end{aligned}$$
and Lemma \ref{lemInfPhi} bounds $\inf_K\phi$ from below, giving the stated
inequality.
\end{proof}

\subsection{From cubes back to balls}

\begin{proof}[Proof of Theorem \ref{theoCintro}]
Let $Q\in\QQ_V$, so that $3Q\subset V$. Cover $Q$ by restricted balls
$\{\Delta_i\}_{i\leq N}$ centred at points of $Q$ and of radius
$c\,\ell(Q)/\gamma$, with $c=c(A_0)$ chosen so that
$\gamma\Delta_i\subset3Q$; since $\diam Q\leq A_0\ell(Q)$, we may take
$N\leq C(d)\gamma^{d}$. Each $\gamma\Delta_i$ has radius comparable to
$\ell(Q)$, so $\mu(\gamma\Delta_i)\geq C_\mu^{-m_d}\mu(Q)$ with $m_d$
dimensional, and $\mu(\Delta_i)\leq\mu(3Q)\leq C_\mu\mu(Q)$. Hence, using
$\left(\sum_i a_i\right)^{1/p}\leq\sum_i a_i^{1/p}$,
$$\begin{aligned}
\left(\avint_{Q}w^p\, d\mu\right)^{\frac1p}
&\leq\sum_{i}\left(\frac{\mu(\Delta_i)}{\mu(Q)}
\avint_{\Delta_i}w^p\, d\mu\right)^{\frac1p}
\leq C_\mu C_{RH}\sum_{i}\avint_{\gamma\Delta_i}w\, d\mu\\
&\leq C(d,C_\mu)\,\gamma^{d}C_{RH}\,\frac{\nu(3Q)}{\mu(Q)} ,
\end{aligned}$$
because $\gamma\Delta_i\subset3Q$. Thus $\nu\in wRH_p(\mu,\QQ_V)$ with a
constant
$$C'_{RH}\leq C(d,C_\mu)\,\gamma^{d}C_{RH},$$
and Lemma \ref{lemGehringEnlarged} applies with $C'_{RH}$ in place of
$C_{RH}$. Since $\gamma$ enters only through $C'_{RH}$, and $C'_{RH}$ only
through the product $\gamma^{d}C_{RH}$, the exponent $\ve_0=q_0-p$ depends on
$\gamma$ and $C_{RH}$ only through that product.
\end{proof}

\subsection{The global self-improvement}

\begin{theorem}\label{theoGehringWeak}
Let $\mu$ be a doubling Radon measure in $\R^d$ and let $\nu\in wRH_p(\mu)$
with $1<p<\infty$. Then there is $\ve_0>0$, depending on $d$, $C_\mu$, $p$ and
$C_{RH}$, such that $\nu\in wRH_{p+\ve}(\mu)$ for every $0<\ve<\ve_0$.
\end{theorem}

\begin{proof}
Let $\Delta$ be a restricted ball. Since
$\frac{\nu(2\Delta)}{\mu(\Delta)}\leq C_\mu\avint_{2\Delta}w\, d\mu$, the
hypothesis of Theorem \ref{theoCintro} holds with $\gamma=2$ and constant
$C_\mu C_{RH}$; apply it with $V=\frac32\Delta$ and $K=\overline\Delta$. As
$\dist(K,X\setminus V)\geq\frac{r(\Delta)}2$, $\diam K\leq2r(\Delta)$ and
$\mu(V)\leq C_\mu\mu(K)$, the resulting constants do not depend on $\Delta$,
and for every $\ve<\ve_0$
$$\left(\avint_{\Delta}w^{p+\ve}\, d\mu\right)^{\frac1{p+\ve}}
\leq C\left(\avint_{\frac32\Delta}w^{p}\, d\mu\right)^{\frac1p} .$$
It remains to undo the enlargement. Cover $\frac32\Delta$ by restricted balls
$\{\Delta_i\}_{i\leq N}$ centred in $\frac32\Delta$ of radius
$\frac{r(\Delta)}8$, so that $2\Delta_i\subset2\Delta$, with $N$ dimensional
and $\mu(\Delta_i)\geq C_\mu^{-m_d}\mu(\Delta)$. Then \rf{eqWeakRH} gives
$$\int_{\frac32\Delta}w^p\, d\mu
\leq\sum_{i}\mu(\Delta_i)\avint_{\Delta_i}w^p\, d\mu
\leq C_{RH}^{p}\sum_{i}\frac{\nu(2\Delta_i)^{p}}{\mu(\Delta_i)^{p-1}}
\leq C\,\frac{\nu(2\Delta)^{p}}{\mu(\Delta)^{p-1}} ,$$
so $\left(\avint_{\frac32\Delta}w^p\, d\mu\right)^{1/p}
\leq C\frac{\nu(2\Delta)}{\mu(\Delta)}$ and the theorem follows.
\end{proof}

\begin{remark}\label{remGehringStrong}
The dilation cannot be dispensed with, and the dyadic and ball settings part
company here. On the basis $\DD\cup3\DD$ there is a \emph{strong} Gehring
lemma: if $\nu\in RH_p(\mu,\BB_{Q_0})$ with $1<p<\infty$ then
$\nu\in RH_{p+\ve}(\mu,\BB_{Q_0})$ for $\ve$ small, which is
\cite[Thm.~3.1]{AndersonWeirich}; one may also read it off from the proof of
Lemma \ref{lemGehringEnlarged} by taking $\W=\{Q_0\}$ and replacing $3P$ by
$P$ throughout. For balls in a general space of homogeneous type the
corresponding statement is false, by \cite[\S7.3]{AndersonHytonenTapiola};
for restricted Euclidean balls we do not know whether it holds; see \S\ref{secOpen}. For
metric measure spaces enjoying an annular decay property the dilation can be
removed, by \cite[Cor.~3.2]{Maasalo}.
\end{remark}

\section{Further remarks and open questions}
\label{secOpen}

We collect here what the preceding sections leave undecided.

\subsection*{The strong theory for restricted Euclidean balls}

Three implications that hold for the Lebesgue measure fail in a general space
of homogeneous type: that a Fujii--Wilson $A_\infty^{FW}$ weight satisfies a strong reverse H\"older inequality 
\rf{eqStrongRH} for some $p>1$; the converse; and the self-improvement
\rf{eqStrongRH} $\Rightarrow$ $RH_{p+\ve}$. The witnesses are
\cite[\S7.2]{AndersonHytonenTapiola}, \cite[Ex.~4.1]{KorteKansanen} and
\cite[\S7.3]{AndersonHytonenTapiola}, respectively, and all three are built on
the $\ell^\infty$ metric in a way that the introduction describes. For
restricted Euclidean balls, Example \ref{exLadder} settles the second: the
converse fails here too. The first and the third we leave open.

For the third, one obstruction is worth recording. If a level set $S$ of
$w$ is such that the proportion $\mu(\Delta\cap S)/\mu(\Delta)$ runs through
an interval as $\Delta$ varies, then testing \rf{eqStrongRH} on the ball where
that proportion equals the ratio of the two levels of $w$ bounds the
oscillation of $w$, and no counterexample of the two-piece kind survives. That
proportion is continuous in the radius as soon as $\mu$ charges no sphere
centred at the point in question; and the ball that breaks the strong Gehring
lemma in \cite[\S7.3]{AndersonHytonenTapiola} is precisely one whose boundary
carries the whole of the line $A$. So a Euclidean counterexample would have to
be built on a measure charging spheres.

By \cite[Thm.~5.2]{AndersonWeirich} a doubling weight satisfies the strong
Gehring lemma in any space of homogeneous type, and the same argument applied
to \cite[Cor.~5.5]{AndersonHytonenTapiola} covers the first implication. So
both open questions concern non-doubling weights only, and Example
\ref{exLadder} is a natural place to look: its measure $\nu$ is not doubling,
but its weight lies in $RH_\infty$ and is therefore too good. What one would
have to graft onto it is the two-piece profile of
\cite[\S7]{AndersonHytonenTapiola} --- a background level $\ve_k$ on a set of
comparable measure, carrying a spike whose $L^p$ norm is a large multiple of
its $L^1$ norm.

None of this touches the weak classes of \S\ref{secWeak}, where the two
averages are taken over different balls and no such collapse is expected; the
harmonic measures that motivate \rf{eqBLintro} are not doubling either.

\subsection*{The class $QC$ on balls and on cubes}

Proposition \ref{propoDoublingAllEquivalent} identifies the classes carrying a
quantifier on the two bases, but says nothing about the plain union
$QC$. As explained in Remark \ref{remQCnotQCD}, the transfer lemma degrades
the parameter $\delta$ by a factor $\widehat C_\mu$, and for small $\delta$ this
takes it outside the admissible range. We do not know whether
$QC(\mu)=QC(\mu,\BB_{\mathcal{D}})$ for every doubling $\mu$; a counterexample would
have to consist of a pair $(\mu,\nu)$ satisfying \rf{eqDefQC} on restricted
balls for one pair $(\delta,\ve)$, with $\delta$ necessarily small, and for no
pair over $\DD\cup3\DD$. The same question can be asked for the weak classes
of \S\ref{secWeak}, where the analogue of $A_\infty^{QC}$ --- the class
obtained from Definition \ref{defWeakClasses} by exchanging the two
quantifiers --- has not been considered here at all. The class $wRH_1^{QC}$ itself is what Kinnunen, Kurki and Mudarra call
weak $A_\infty$, and they show that it admits several equivalent
descriptions; in the notation of this paper, among them a weak Fujii--Wilson
condition and a weak reverse H\"older inequality for a power of $w$.

\subsection*{Sharpness of the thresholds}

Theorem \ref{theoAintro} requires $\ve_0\leq\ve_*$ with $\ve_*$ depending on
$d$ and $C_\mu$, and Lemma \ref{lemBothDoubling} requires
$\delta_0>1-C_\mu^{-1}$. Thresholds of this kind cannot be dispensed
with. In $\R^{d}$ let $S\coloneqq \{x:\, 0\leq x_{d}\leq1\}$ and $w=\chi_S$.
For every cube $Q$ meeting $S$, the sets $Q\cap S$ and $2Q\cap S$ are boxes
whose first $d-1$ sides have lengths $\ell(Q)$ and $2\ell(Q)$, while the last
side of $Q\cap S$ is no longer than that of $2Q\cap S$; hence
$$\nu(F)\leq\nu(Q)=\mu(Q\cap S)\leq2^{1-d}\mu(2Q\cap S)=2^{1-d}\nu(2Q)$$
for every measurable $F\subset Q$, that is, \rf{eqWeakQC} with $\delta=1$ and
$\ve=2^{1-d}$. Yet $w$ lies in no $wRH_p$: this is an example of Sawyer,
reported in \cite[Ex.~4.2]{KinnunenKurkiMudarra}. So $\ve_*\leq2^{1-d}$, and a
single pair $(\delta,\ve)$ can characterize the class only if $\ve$ is below a
dimensional constant --- no smallness of $\delta$ is involved.
Kinnunen, Kurki and Mudarra need a
threshold of the form $\eta<c_d^{-5}$, for a dimensional constant $c_d$,
in their own characterization. What the
optimal dependence on $C_\mu$ is, in any of these statements, we do not know.

%

%

\subsection*{Dropping the doubling hypothesis}

Doubling is used here in two distinct ways: to run the Calder\'on--Zygmund
iteration on $\DD\cup3\DD$, and to compare cubes with balls. Orobitg and
P\'erez \cite{OrobitgPerez} developed an $A_p$ theory for non-doubling
measures in $\R^d$ by replacing the doubling hypothesis with
$\mu(\partial Q)=0$ and using a Besicovitch-type Calder\'on--Zygmund
decomposition. Whether the weak classes of \S\ref{secWeak}, and in particular
the Bennewitz--Lewis condition, admit a theory in that setting we leave open;
the interest is that harmonic measure, which is the source of \rf{eqBLintro},
need not be doubling.


\section*{On the use of artificial intelligence}

The mathematics of this paper is the author's own, and every statement and
every proof in it has been verified by him; he is solely responsible for them.
An artificial intelligence assistant (Claude, Anthropic) was used as a tool in
its preparation, in four roles, which we record for the reader's information.

First, as a pre-referee: reading the argument and reporting gaps, inconsistent
notation, hypotheses used but not stated, and constants that did not match
between a lemma and its application. Second, in the reorganization of material
the author had written earlier, together with editorial comments on it; the
unified form of Lemma \ref{lemLevelToQ}, the explicit range of exponents in
Remark \ref{remRangeOfQ}, and the isolation of Lemmas
\ref{lemBesicovitchSelection} and \ref{lemInfPhi} came out of that process.
Third, in contrasting the results with the literature, which produced most of
the attributions of \S\ref{secIntro} and the comparisons recorded throughout.
Fourth, in the search for counterexamples:  Example \ref{exLadder} together
with Proposition \ref{propoThinAnnuli}, were arrived at in that dialogue.

\section*{Acknowledgements}

This paper is a by-product of my effort to understand the theory of harmonic measure and to make it less forbidding for newcomers to enter it. I am grateful to Xavier Tolsa, with whom I am writing a book on the subject, to appear soon, for the many conversations that shaped this work. These tools took shape while I was supervising the master's thesis of Roger Rams Domènech, whom I thank warmly, as I do Marta Cano Cagigas, a student in the Barcelona Introduction to Mathematical Research programme, who had the patience to read through this material carefully. I also thank Carlos Pérez and Carlos Mudarra for their conversations and for their insight into the relevant literature.

\bibliographystyle{alpha}
\bibliography{../../../bibtex/Llibres}

\end{document}